\documentclass[11pt,twoside]{article}

\usepackage[square,numbers]{natbib}

\iffalse %%%%%%%%%%%%%%%%%%%%%%%%%%%%%%%%%%%%%%%%%%%%%%%%%%%%%%%%%%%%%%%%%%%%%%%%
\usepackage[autostyle]{csquotes}
\usepackage[
    backend=biber,
    style=numeric,
    firstinits = true,
    url=false, 
    doi=true,
    eprint=true
]{biblatex}
\renewbibmacro{in:}{%
  \ifentrytype{article}{}{%
  \printtext{\bibstring{in}\intitlepunct}}}

\DeclareFieldFormat[article,unpublished,book]{pages}{#1}  

\DeclareFieldFormat[article]{title}{#1}
\DeclareFieldFormat[unpublished]{title}{#1}
\DeclareFieldFormat[misc]{title}{#1}
\fi %%%%%%%%%%%%%%%%%%%%%%%%%%%%%%%%%%%%%%%%%%%%%%%%%%%%%%%%%%%%%%%%%%%%%%%%

\newcommand*{\doi}[1]{\href{http://dx.doi.org/#1}{doi: #1}}

\usepackage{lipsum} % Package to generate dummy text throughout this template

\usepackage{microtype} % Slightly tweak font spacing for aesthetics

\usepackage{lmodern}
\usepackage[sans]{dsfont}
\usepackage[latin1]{inputenc}
\usepackage{color}
\usepackage{delimdelim}

\usepackage[hmarginratio=1:1,top=32mm,columnsep=20pt]{geometry} % Document margins
\usepackage{booktabs} % Horizontal rules in tables
\usepackage{hyperref} % For hyperlinks in the PDF

\usepackage{paralist} % Used for the compactitem environment which makes bullet points with less space between them

\usepackage{abstract} % Allows abstract customization
\newcommand{\ep}{\epsilon}
\newcommand{\m}{\mbox{d}}
\definecolor{fu}{RGB}{0, 0, 0}
\definecolor{gr}{RGB}{21, 194, 21}

\usepackage{titlesec} % Allows customization of titles
\usepackage{fancyhdr} % Headers and footers
\usepackage{graphicx}
\usepackage{amsmath}
\usepackage{amssymb}
\usepackage{amsthm}
\usepackage{subfigure}
\usepackage{textcomp}
\usepackage{listings}
\usepackage{setspace}
\usepackage{graphicx}
\usepackage{epstopdf}
\usepackage{eurosym}
\usepackage{inputenc}
\usepackage{inputenc}
\usepackage{appendix}
\usepackage{tabularx}
\usepackage{array}
\usepackage{makeidx}
\usepackage{etoolbox}
\usepackage{mathtools}
\usepackage{empheq}
\usepackage{mathrsfs}
\usepackage{marginnote}

\DeclareGraphicsExtensions{.eps}
\DeclareGraphicsExtensions{.png}
\theoremstyle{plain}
\newtheorem{theorem}{Theorem}[section]
\newtheorem{lemma}[theorem]{Lemma}

\newtheorem{prop}[theorem]{Proposition}
\newtheorem{mydef}[theorem]{Definition}

\newtheorem*{assumg*}{Assumption~(G)}

\newtheorem*{assumng*}{Assumption~(NG)}
\newtheorem*{assumng1*}{Assumption~(NG-bis)}

\newcommand{\per}{\mathrm{per}}

\theoremstyle{definition}

\newtheorem{rem}[theorem]{Remark}
\hypersetup{citecolor=black}

 {\proof[#1]\leftskip=0.0cm\rightskip=0.0cm}
 {\endproof}

\label{mathrefs}

\newlength\tindent
\title{\vspace{-15mm}\fontsize{24pt}{10pt} \LARGE{\scshape{A regularised Dean--Kawasaki model: derivation and analysis}}} % Article title
\small{
\author{
\textsc{Federico Cornalba\footnote{e-mail: {\ttfamily F.Cornalba@bath.ac.uk}},\hspace{0.2 pc} Tony Shardlow\footnote{e-mail: {\ttfamily T.Shardlow@bath.ac.uk}},\hspace{0.2 pc} Johannes Zimmer\footnote{e-mail: {\ttfamily J.Zimmer@bath.ac.uk}}}%\thanks{Department of Mathematical Sciences, University of Bath, Bath, BA2 7AY, United Kingdom.} 
\vspace{0.5 pc}\\ \footnotesize{Department of Mathematical Sciences, University of Bath, Bath, BA2 7AY, United Kingdom}
\vspace{-5mm}
}}
\date{}
\begin{document}

\maketitle % Insert title
\thispagestyle{empty}

\pagestyle{fancy} % All pages have headers and footers

%----------------------------------------------------------------------------------------
%	ABSTRACT
%----------------------------------------------------------------------------------------

\renewenvironment{abstract}
 {\small
  \begin{center}
  \bfseries \abstractname\vspace{-0.0 pc}\vspace{0pt}
  \end{center}
  \list{}{%
    \setlength{\leftmargin}{10mm}% <---------- CHANGE HERE
    \setlength{\rightmargin}{\leftmargin}%
  }%
  \item\relax}
 {\endlist}

 \newenvironment{thanks}
 {%\small
  %\begin{center}
%  \bfseries \abstractname\vspace{-0.0 pc}\vspace{0pt}
  %\end{center}
  \list{}{%
    \setlength{\leftmargin}{0mm}% <---------- CHANGE HERE
    \setlength{\rightmargin}{\leftmargin}%
  }%
  \item\relax}
 {\endlist}

\begin{abstract}
  The Dean--Kawasaki model consists of a nonlinear stochastic partial differential equation featuring a conservative,
  multiplicative, stochastic term with non-Lipschitz coefficient, and driven by space-time white noise; this equation describes
  the evolution of the density function for a system of finitely many particles governed by Langevin dynamics. Well-posedness for
  the Dean--Kawasaki model is open except for specific diffusive cases, corresponding to overdamped Langevin dynamics. There, it
  was recently shown by Lehmann, Konarovskyi, and von Renesse that no regular (non-atomic) solutions exist.

  We derive and analyse a suitably regularised Dean--Kawasaki model of wave equation type driven by coloured noise, corresponding
  to second order Langevin dynamics, in one space dimension. The regularisation can be interpreted as considering particles of
  finite size rather than describing them by atomic measures. We establish existence and uniqueness of a solution. Specifically,
  we prove a high-probability result for the existence and uniqueness of mild solutions to this regularised Dean--Kawasaki model.

  

\end{abstract}

\section{Introduction}

%% CommentFC23 (old intro) here

Fluctuating hydrodynamics is concerned with the description of the evolution of a large number of particles by means of suitable
stochastic partial differential equations. We refer the reader to~\cite{Eyink1990a} and give as an example the
\emph{Dean--Kawasaki} model~\cite{Dean1996a,Kawasaki1998a}
\begin{align}
  \label{eq:1}
  \frac{\partial \rho}{\partial t}(x,t)=\underbrace{\nabla\cdot\left(\rho(x,t)\,
  \nabla\frac{\delta F(\rho)}{\delta \rho}\right)}_{=: 
  \mathscr{D}}+\underbrace{\nabla\cdot\left(\sigma\sqrt{\rho(x,t)}\,\xi\right)}_{=:\mathscr{S}}.
\end{align}

Here $\rho\colon D\times[0,T]\subset\mathbb{R}^d\times[0,+\infty]\rightarrow [0,+\infty]$ is the density of particles, $\sigma$ is
a small real parameter, $F$ is a free-energy functional, and $\xi$ is a space-time white noise. The deterministic term
$\mathscr{D}$ is a gradient-flow-driven term describing the average behaviour of the system, and can be derived from the
Fokker--Planck analysis. The stochastic term $\mathscr{S}$ accounts for fluctuations about the mean due to the finite number of
particles in the system. As a result of the divergence form, both the terms $\mathscr{D}$ and $\mathscr{S}$ account for
conservation of mass in the system, see also~\cite{Fehrman2017a,Fischer2018a} for similar models.

Equation~\eqref{eq:1} poses a fascinating mathematical challenge. On one side, this equation and its more complex incarnations are
widely simulated in physics; see for example~\cite[Eq.~(59)]{Thompson2011a},~\cite{Lutsko2012a}
and~\cite{Duran-Olivencia2017a}. On the other hand, very little is known about existence and uniqueness of solutions for this
class of problems, as discussed below.

We point out three main difficulties posed by~\eqref{eq:1} from a mathematical perspective. Firstly, the noise term $\mathscr{S}$
is defined by means of a formal divergence operator. The regularity of the argument of the divergence operator is \emph{a priori}
unknown. In particular, a standard $L^2(D)$-valued stochastic analysis for the argument $\sigma\sqrt{\rho(x,t)}\,\xi$ (in the
sense of~\cite{Prevot2007a,Da-Prato2014a}, for example) would not allow us to interpret the noise $\mathscr{S}$,
hence~\eqref{eq:1}, in a function setting.  Secondly, the derivation of~\eqref{eq:1} in the physics literature is formal and only
applicable to empirical (thus atomic) measures.  Whether a solution to~\eqref{eq:1} for smooth initial data exists is in general
not clear. Thirdly, the lack of Lipschitz continuity associated with the square root poses further difficulties.

Von Renesse and collaborators have studied regularised versions of~\eqref{eq:1} in the foundational
works~\cite{Renesse2009a,Andres2010a,Konarovskyi2017a,Konarovskyi2018a}.  They obtain existence results for measure-valued
martingale solutions for modifications of~\eqref{eq:1} (in~\cite{Andres2010a,Renesse2009a} for the Gibbs--Boltzmann entropy
functional $F$ scaled by $\mu>0$, and in~\cite{Konarovskyi2017a} for the case $F\equiv0$). These modifications affect the drift
of~\eqref{eq:1}, and they are associated with Dirichlet form arguments and with the Wasserstein geometry over the space of
probability densities.

%To consider regularisations to~\eqref{eq:1} is much more than a mere mathematical technicality. 
Very recently, Lehmann, Konarovskyi, and von Renesse~\cite{Lehmann2018} dispelled the belief that there are smooth solutions to
the purely diffusive Dean--Kawasaki equation.  More precisely, for~\eqref{eq:1} in one space dimension with free energy
$F:=\frac{N}{2}\int_{D}{\rho(x)\log(\rho(x))\m x}$, where~\eqref{eq:1} becomes
\begin{align*}
  \label{eq:dean-max1}
  \frac{\partial \rho}{\partial t}(x,t) = \frac N 2 \Delta \rho(x,t) + \nabla \cdot \left(\sqrt{\rho(x,t)}\,\xi\right),
\end{align*}
they showed that a unique measure-valued martingale solution exists if and only if $N\in\mathbb{N}$; in this case, the solution is
the empirical distribution associated with $N$ independent Brownian particles, so an atomic measure. The basis of this dichotomy
is the interplay of the particular geometry of diffusion and noise in the context of a stochastic Wasserstein gradient flow.  We
also mention that a similar setting later led the authors of~\cite{Lehmann2018} to obtain an analogous dichotomy in the case of
more general smooth drift potentials $F$~\cite{Lehmann2018b}.

The central differences to the approach presented below are that in~\cite{Lehmann2018}, the underlying particle dynamics is first
order (overdamped Langevin); the noise is derived from deep probabilistic arguments (describing Brownian motion in the space of
probability measures with finite second moment, i.e., relying on the Wasserstein geometry); and the noise is not regularised.

The original derivation of Dean--Kawasaki equations is mathematically opaque, with one noise being replaced by a stochastically
equivalent one, and with physical approximations closing the model in the density $\rho$ under the assumption of local equilibrium
(see \emph{Steps 2-3} in Subsection~\ref{sec:10} below); since the existence of solutions to this type of equations is so
delicate, we revisit the derivation, introduce physically motivated regularisations and then establish existence and uniqueness of
solutions (in a high probability sense). The starting point are undamped (second order) Langevin equations with on-site potential,
describing the motion of finitely many particles. A key point for modelling the particles is that we do not describe them by
atomic (Dirac) measures; instead, each particle is given by a Gaussian with standard deviation $\ep\ll 1$, centred on the particle
positions (see Figure~\ref{fig:1}). As a consequence, standard tools from stochastic calculus apply to the empirical density for
$N$ such particles. We find it useful to work with (a regularised version of) the empirical measure
$\frac 1 N \sum_{i=1}^N \delta(x-q_i)$ and remark that both~\cite{Lehmann2018} and~\cite{Dean1996a} use the different, but
equivalent, scaling $\sum_{i=1}^N \delta(x-q_i)$, see~\eqref{eq:2} below.  The advantage of the scaling chosen here is that the
limit of the number of particles $N\to\infty$ is well-defined, leading to the hydrodynamic scale, and that we work in the setting
of probability measures. Specifically, we study suitably combined limits of the number of particles $N$ going to $\infty$ and the
width parameter $\ep$ going to $0$. Then, the noise in the resulting equations scales with $N^{-1/2}$ and disappears in the limit
$N\to\infty$ (in contrast to~\eqref{eq:1}; the dependency on the scaling in the deterministic and stochastic operator
in~\eqref{eq:1} also plays a role in~\cite{Andres2010a,Renesse2009a,Lehmann2018}).  As in the original derivation by
Dean~\cite{Dean1996a}, we then replace a non-closed expression for the noise obtained by It\^o calculus with a stochastically
equivalent one; yet, in the framework we establish, the new noise can be compared to the original one and we obtain error bounds,
and show that their difference is small. In addition, we replace a non-closed component of the deterministic drift with a closed
expression by working in a low temperature regime for the Langevin system. We are then in a position to formulate, for large but
finite $N$, a regularised stochastic wave equation of Dean--Kawasaki type.  For this equation, we establish a high probability
existence and uniqueness result for mild solutions using a small-noise-regime analysis; more specifically, we invoke a Chebyshev
inequality argument to prove that the solution stays close to a suitable deterministic process which is positive and bounded away
from the non-Lipschitz noise singularity (i.e., from the identically vanishing density).

The general philosophy of this paper to derive stochastic equations describing the evolution of $N$ Gaussians with given variance
instead of $N$ Diracs seems to be novel. Yet it seems to be natural and potentially useful in a variety of situations. For
example, if one seeks to analyse the evolution of finitely many droplets in a suspension, then the description of a droplet by a
Gaussian seems at least as natural as a description by a Dirac. The stochastic equation derived and studied here describes the
evolution of such a system of particles. Additionally, the tightness arguments in $N$ and $\epsilon$ developed in
Subsection~\ref{sec:6} are of independent interest. While we use them as novel argument to compare noise expressions, they can
also be useful in an alternative derivation of the hydrodynamic limit, though we do not pursue this avenue in this article.

Before describing this approach in more detail, we sketch the derivation commonly taken in the physical literature.

\subsection{Original model derivation in dimension $d=1$}
\label{sec:10}

The \emph{Dean--Kawasaki} model~\cite{Dean1996a,Kawasaki1998a} arises in the mathematical description of a system of
\emph{finitely many} particles experiencing Langevin dynamics. We briefly discuss the derivation of this model by
following~\cite[Sec.~II]{Lutsko2012a}. Consider $N$ stochastically independent and identically distributed particles moving on the
real line, with position and velocity $\{(q_i,p_i)\}_{i=1}^{N}$. More precisely, their evolution is given by the Langevin dynamics
\begin{equation}
  \label{eq:24}
  \left\{
    \begin{array}{l}
      \displaystyle \dot q_i= p_i,  \vspace{0.2 pc}\\%\frac{\m q_i}{\m t}= p_i,  \vspace{0.2 pc}\\
      \displaystyle\dot p_i= \left(-\gamma p_i-V'(q_i)\right)+\sigma\,\dot\beta_i,\qquad i=1,\cdots,N,
    \end{array}
  \right.
\end{equation}
starting from independent and identically distributed initial conditions $\{(q_{i,0},p_{i,0})\}_{i=1}^{N}$. In~\eqref{eq:24},
$\{\beta_i\}_{i=1}^{N}$ is a family of independent standard Brownian motions on a probability space
$(\Omega,\mathcal{F},\mathbb{P})$, where $\sigma,\gamma>0$ are given constants satisfying the fluctuation-dissipation relation
$\sigma^2/(2\gamma)=k_{B}T_{e}$ (see for example~\cite{Chandler1987a}), and $V\colon\mathbb{R}\rightarrow\mathbb{R}$ is a
potential.  The particle system is described in terms of the global quantities
\begin{equation}
  \label{eq:2}
  \rho_N(x,t):=\sum_{i=1}^{N}{\delta(x-q_i(t))} \text{ and } j_N(x,t):=\sum_{i=1}^{N}{p_i(t)\delta(x-q_i(t))},\qquad x\in\mathbb{R},\,\,t
  \geq 0,
\end{equation}
representing the \emph{local density} and the \emph{momentum density}, respectively. These quantities, which are not rescaled in
$N$, are to be understood in the Schwartz distribution sense, due to the presence of the Dirac distributions, denoted by
$\delta$. We sketch below how this leads to~\eqref{eq:1}, the \emph{Dean--Kawasaki} stochastic partial differential
equation~\cite{Dean1996a,Kawasaki1998a}, following~\cite{Lutsko2012a}.

\emph{Step 1}. Evolution equations of first order in time~\cite[Eq.~(4)]{Lutsko2012a} are derived for both $\rho_N$ and $j_N$ by
means of standard It\^o calculus, in a distributional sense. These equations are a simple superposition of the stochastic
equations resulting from the Langevin dynamics~\eqref{eq:24} of each particle $i=1,\cdots,N$. The evolution equation for $\rho_N$
is a conservation law associated with the momentum density, and it reads $\partial\rho_N/\partial t=-\nabla\cdot j_N$. The
evolution equation for $j_N$ is, broadly speaking, an undamped equation perturbed by a particle-dependent stochastic noise.

\emph{Step 2.} The aforementioned particle-dependent noise featured in the stochastic equation~\cite[Eq.~(4)]{Lutsko2012a}
associated with $j_N$ is not of closed form (i.e., it cannot be expressed as a simple function of the quantities $\rho_N$ and
$j_N$). This noise is
\begin{align}
  \label{eq:84}
  \sigma\sum_{i=1}^{N}{\delta(x-q_i(t))\dot\beta_i}.
\end{align} 
For this reason, the above noise is \emph{formally} replaced by another noise preserving the spatial covariance structure
of~\eqref{eq:84}. The latter noise takes the shape
\begin{align}
  \label{eq:85}
  \sigma\sqrt{\rho_N(x,t)}\,\xi,
\end{align}
where $\xi$ is a space-time white noise.

\emph{Step 3.} The first order evolution equations for $\rho_N$, $j_N$ (with the noise replacement~\eqref{eq:85}) are then
analysed on the hydrodynamic scale under a local equilibrium assumption, thus giving equations in some new variables $\rho$ and
$j$~\cite[Eq.~(11)]{Lutsko2012a}. In one space dimension, this system reads
\begin{equation}
  \label{eq:80-L}
  \left\{
    \begin{array}{l}
      \vspace{-0.8 pc}\\
      \displaystyle \frac{\partial \rho}{\partial t}(x,t)  = -\frac{\partial j}{\partial x}(x,t),  \vspace{0.5 pc}\\
      \displaystyle\frac{\partial j}{\partial t}(x,t)  = \left(-\gamma 
      j(x,t) - \rho(x,t) \nabla\frac{\delta F(\rho)}{\delta \rho}\right) + \eta \sqrt{\rho(x,t)} \xi
    \end{array}
  \right.
\end{equation}
(in suitable units, with a small parameter $\eta$), where $F$ is a suitable free-energy functional, and $\delta$ denotes
variational differentiation. The equations in \eqref{eq:80-L} are then combined into a dissipative wave equation which is closed
in the variable $\rho$~\cite[Eq.~(12)]{Lutsko2012a}. This step provides the divergence operator for the stochastic noise
of~\eqref{eq:1}. The final evolution equation~\eqref{eq:1} is obtained by passing to the overdamped limit. We will not follow this
last step and instead study a stochastic damped wave equation which can be seen as regularisation of~\eqref{eq:80-L},
see~\eqref{eq:80} below. For details of the procedure just sketched, we refer the reader to~\cite[Secs.~IIA, IIB]{Lutsko2012a}
and~\cite{Dean1996a,Kawasaki1998a}.

%% Comment FC7 (in unused-comments.tex) was here

\subsection{Summary of the paper and main results}

We now summarise the contents and main results of this paper.

We set the notation in Subsection~\ref{sec:3}. In Subsection~\ref{sec:2}, we define two different sets of hypotheses regarding the
potential $V$, referred to as Assumption~(G) and Assumption~(NG). The first one is associated with a vanishing potential,
$V \equiv 0$, which makes some specific tools of the theory of Gaussian random variables applicable. The second assumption allows
for a polynomially diverging potential $V(q)\approx |q|^{2n}$, in the context of a Fokker--Planck analysis for~\eqref{eq:24}.

\emph{Derivation of the regularised Dean--Kawasaki model:} This is the content of Section~\ref{s:4}, and we proceed by adapting
the procedure sketched in \emph{Steps 1-2}, Subsection~\ref{sec:10}, to a \emph{function} context rather than the original
distributional setting~\cite{Dean1996a,Kawasaki1998a}. We resolve the formal replacement of the noise highlighted in Section
\ref{sec:10} by smoothing the defining components of $\rho_N$ and $j_N$. Specifically, we keep the Langevin particle system
\eqref{eq:24}, and consider the \emph{$\ep$-smoothed local density} and \emph{$\ep$-smoothed momentum density},
\begin{align}
  \label{eq:102} 
  \rho_{\ep}(x,t):=\frac{1}{N}\sum_{i=1}^{N}{w_{\ep}(x-q_i(t))} \text{ and }
  j_{\ep}(x,t):=\frac{1}{N}\sum_{i=1}^{N}{p_i(t)w_{\ep}(x-q_i(t))},\qquad x\in\mathbb{R},\,t\geq 0,
\end{align}
where $\ep>0$ and $w_{\ep}(x):=(2\pi\ep^2)^{-1/2}\exp\{-x^2/(2\ep^2)\}$ is the Gaussian kernel with mean $0$ and variance $\ep^2$,
see also Definition~\ref{la:1}. The kernels $w_{\ep}$ approximate the Dirac delta distribution for small values of $\ep$. Notice
that $\rho_{\ep}$ and $j_{\ep}$ include a rescaling in the number of particles, while $\rho_N$ and $j_N$ do not.

\begin{figure}[h]
  \begin{center}
    \includegraphics[width=16cm,height=5cm]{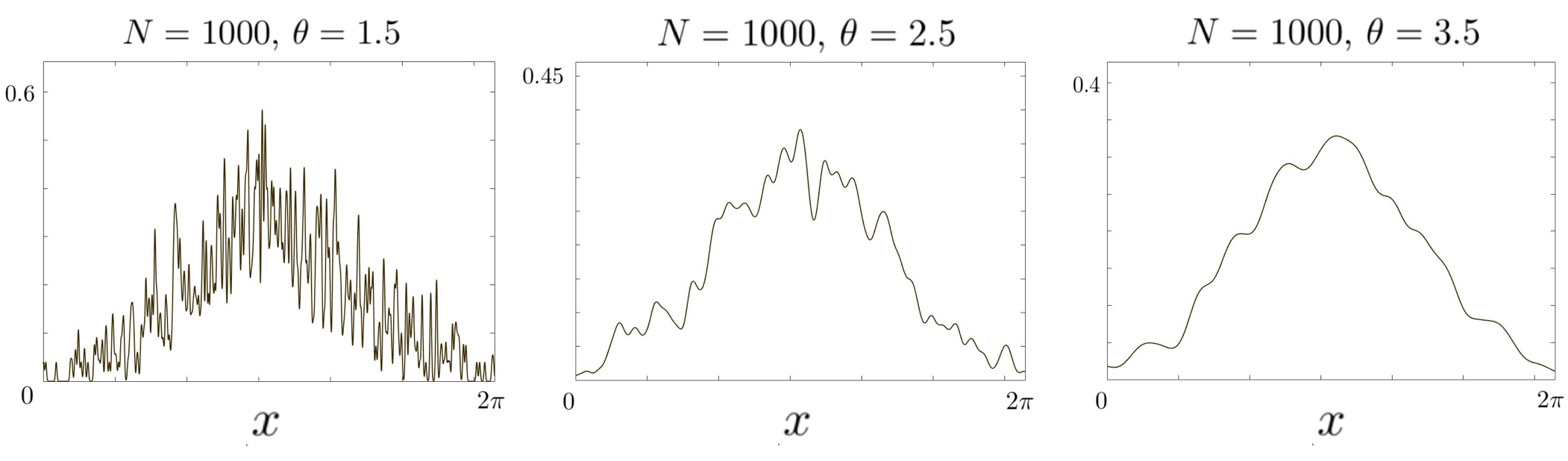}
    \caption{Numerical simulation of the \emph{$\ep$-smoothed} local density
      $\rho_{\ep}(\cdot,t)=N^{-1}\sum_{i=1}^{N}{w_{\ep}(\cdot-q_i(t))}$ defined in~\eqref{eq:102}, for a fixed time $t$, and on
      $D=[0,2\pi]$. In this specific example, $q_i(t)\sim\mathcal{N}(\pi,10^{0.2})$, $N=1000$, and $N$ and $\ep$ satisfy the
      scaling $N\ep^{\theta}=1$ for $\theta = 1.5$ (left), $\theta = 2.5$ (middle), $\theta = 3.5$ (right). The smoothness of the
      density increases with $\theta$.}     
      \label{fig:1}
  \end{center}
\end{figure}

We use the $\ep$-smoothed quantities~\eqref{eq:102} instead of the original quantities~\eqref{eq:2} and follow the same guidelines
described in \emph{Steps 1-2} of Subsection~\ref{sec:10} in order to derive the regularised Dean--Kawasaki model. There, we will
also consider the quantity
\begin{align}
  \label{eq:120}
  j_{2,\ep}(x,t):=\frac{1}{N}\sum_{i=1}^{N}{p^2_i(t)w'_{\ep}(x-q_i(t))}.
\end{align}

We do not adapt \emph{Step 3} of Subsection~\ref{sec:10}, as we will not combine the equations for $\rho_{\ep},j_{\ep}$ or use the
hydrodynamic limit theory.

We perform the analysis of the regularised Dean--Kawasaki model both for fixed values of $N$ and $\ep$, and also by means of a
simultaneous limit involving $N\rightarrow\infty$ and $\ep\rightarrow 0$, for $N$ and $\ep$ satisfying a prescribed scaling. We
first prove some preliminary uniform estimates for the three families of processes $\{\rho_{\ep}\}_{\ep}$ ,$\{j_{\ep}\}_{\ep}$,
$\{j_{2,\ep}\}_{\ep}$ given in~\eqref{eq:102} and~\eqref{eq:120}, as $\ep\rightarrow 0$. We have the following result.

\begin{prop}[Tightness of $\{\rho_{\ep}\}_{\ep}, \{j_{\ep}\}_{\ep}, \{j_{2,\ep}\}_{\ep}$]
  \label{p:1} 
  Let $T>0$, and let $D\subset \mathbb{R}$ be a bounded domain. Assume the validity of either Assumption~(G) or Assumption~(NG),
  given below in Subsection~\ref{sec:2}. Then the families of processes of $\{\rho_{\ep}\}_{\ep}, \{j_{\ep}\}_{\ep}$ are tight in
  $C(0,T;L^2(D))$ and $C(0,T;L^4(D))$, respectively, for $N\ep^{\theta}\geq 1$, with $\theta\geq 3$. In addition, the family
  $\{j_{2,\ep}\}_{\ep}$ is tight in $C(0,T;L^4(D))$ for $N\ep^{\theta}\geq 1$, with $\theta\geq 5$.
\end{prop} 

Proposition~\ref{p:1} yields relative compactness in law for the families of processes
$\{\rho_{\ep}\}_{\ep}, \{j_{\ep}\}_{\ep}, \{j_{2,\ep}\}_{\ep}$ as $\ep\rightarrow 0$. We show convergence for the family
$\{\rho_{\ep}\}_{\ep}$ as $\ep\rightarrow 0$ in the following result.%, again in Subsection~\ref{sec:6}.

\begin{prop}
  \label{p:22} 
  Let $T>0$, and let $D\subset \mathbb{R}$ be a bounded domain. Assume the validity of either Assumption~(G) or Assumption~(NG),
  as well as the scaling $N\ep^{\theta}\geq 1$, for some $\theta\geq 3$. For each $\ep>0$, let $\eta_{\ep}$ be the law of the
  process $\rho_{\ep}$ on $\mathcal{X}:=C(0,T;L^2(D))$. There exists a probability measure $\eta$ on $\mathcal{X}$ such that
  $\eta_{\ep}\stackrel{w}{\rightarrow} \eta\mbox{ in }\mathcal{X}\mbox{ as }\ep\rightarrow 0$. Here $\stackrel{w}{\rightarrow}$
  denotes weak convergence of measures.
\end{prop}

The proofs of Proposition~\ref{p:1} and~\ref{p:22} under Assumption~(G) are the content of Subsection~\ref{sec:6}.

The next step, covered in Subsection~\ref{sec:7}, is the analysis of the evolution equations for $\rho_{\ep}$ and $j_{\ep}$, namely
\begin{equation}
  \label{eq:80}
  \left\{
    \begin{array}{l}
      \vspace{-0.8 pc}\\
      \displaystyle \frac{\partial \rho_{\ep}}{\partial t}(x,t)  = -\frac{\partial j_{\ep}}{\partial x}(x,t),  \vspace{-0.8 pc}\\
      \displaystyle\frac{\partial j_{\ep}}{\partial t}(x,t) = \left(-\gamma 
      j_{\ep}(x,t)-j_{2,\ep}(x,t)-\frac{1}{N}\sum_{i=1}^{N}{V'(q_i(t))w_{\ep}(x-q_i(t))}\right)
      +\overbrace{\frac{\sigma}{N}\sum_{i=1}^{N}{w_{\ep}(x-q_i(t))\dot \beta_i}}^{=:\mathcal{\dot Z}_N(x,t)},%\label{eq:26}
    \end{array}
  \right.
\end{equation}
where $\mathcal{\dot Z}_N(x,t)$ is well-defined due to regularity of $w_{\ep}$ and of the processes
$\{q_{i}\}_{i=1}^{N}$. System~\eqref{eq:80} is analogous to the system of evolution equations for the original quantities
$\rho_N,j_N$ mentioned in \emph{Step 1}, see~\cite[Eq.~(4)]{Lutsko2012a}.

In analogy to the original derivation of the Dean--Kawasaki model, the noise $\mathcal{\dot Z}_N$ is not an elementary function of
$\rho_{\ep}$ and $j_{\ep}$. For this reason, we rewrite $\mathcal{\dot Z}_N$ as
\begin{align}
  \label{eq:3002}
  \mathcal{\dot Z}_N \sim%\stackrel{\mathcal{L}}{=}
  \overbrace{\!\frac{\sigma}{\sqrt{N}}\sqrt{\rho_{\ep/\sqrt{2}}}\,
  \underbrace{Q^{1/2}_{\sqrt{2}\ep}\,\xi}_{=:\tilde{\xi}_{\ep}}}^{=:\mathcal{\dot Y}_N}+\mathcal{\dot R}_N , 
\end{align}
where $\sim$ denotes equality in law, $\xi$ is again a space-time white noise, $Q_{\sqrt{2}\ep}$ is the
convolution operator with kernel $w_{\sqrt{2}\ep}$ on some spatial domain, and $\mathcal{\dot R}_N$ is a (small) stochastic
remainder. The noise $\mathcal{\dot Y}_N$ is properly defined for non-negative function $\rho_{\ep}$. The specific structure of
$\mathcal{\dot Y}_N$ is thoroughly discussed in Subsection~\ref{sec:7}. We estimate the ``difference'' between
$\mathcal{\dot Z}_N$ and $\mathcal{\dot Y}_N$ (i.e., the remainder $\mathcal{\dot R}_N$) with the following result.

\begin{theorem}[Error bounds for covariance structure in~\eqref{eq:80}]
  \label{thm:1}
  Assume the validity of either Assumption~(G) or Assumption~(NG). Let $D\subset\mathbb{R}$ be a bounded set, and let $T>0$. Let
  $N,\ep$ satisfy the scaling $N\ep^{\theta}= 1$, for some fixed $\theta\geq 7/2$. Let
  $Q_{\sqrt{2}\ep}\colon L^2(D)\rightarrow L^2(D)$ be the convolution operator with kernel $w_{\sqrt{2}\ep}$.
   \begin{enumerate}
   \item \label{it:cov1} There exists $C=C(D,T)$ such that the following estimates concerning the spatial covariance of
     $\mathcal{Z}_N$ and $\mathcal{Y}_N$ hold for any $t\in[0,T]$ and $x_1,x_2\in D$:
    \begin{align}
      \left|\mean{\mathcal{Z}_N(x_1,t)\mathcal{Z}_N(x_2,t)}-\mean{\mathcal{Y}_N(x_1,t)\mathcal{Y}_N(x_2,t)}\right| 
      & \leq \frac{C\sigma^2}{N}w_{\sqrt{2}\ep}(x_1-x_2)|x_1-x_2|^2,\label{eq:71}\\
      \left|\mean{\mathcal{Z}_N(x_1,t)\mathcal{Z}_N(x_2,t)}\right| & \leq \frac{C\sigma^2}{N}w_{\sqrt{2}\ep}(x_1-x_2)
                                                                     \label{eq:72}.
    \end{align}
  \item \label{it:cov2} $\mathcal{Z}_N$ and $\mathcal{Y}_N$ decay to 0 as $N\rightarrow\infty$ and $\ep\rightarrow
    0$. Specifically, for any $t\in[0,T]$ and any $x_1\in D$, we have
    \begin{align}
      \emph{Var}\left[\mathcal{Z}_N(x_1,t)\right]\leq C\ep^{\theta-1},\qquad \emph{Var}\left[\mathcal{Y}_N(x_1,t)\right]
      \leq C\ep^{\theta-1}.
    \end{align}
  \end{enumerate}
\end{theorem}

Theorem~\ref{thm:1}, which is proved in Subsection~\ref{sec:8} under Assumption~(G), quantifies the error introduced when
replacing the noise $\mathcal{\dot Z}_N$ with the multiplicative noise $\mathcal{\dot Y}_N$. More specifically, the bound
in~\eqref{eq:71} is negligible for $x_1$, $x_2$ close to each other, when compared with the bound in~\eqref{eq:72}. In addition,
both $\mathcal{\dot Z}_N$ and $\mathcal{\dot Y}_N$ are negligible for distant $x_1$ and $x_2$. In combination with
Proposition~\ref{p:1}, Theorem~\ref{thm:1} guarantees convergence of~\eqref{eq:80} to a deterministic system of equations, for
$N\rightarrow \infty$ and $\ep\rightarrow 0$. This differs from the original Dean--Kawasaki model, as we have rescaled in the
number of particles $N$.

\begin{rem}
  In the limit of infinitely many particles, $N\to\infty$, and under a local equilibrium assumption, one obtains as hydrodynamic
  limit~\eqref{eq:80-L} without the noise term and with the limit of $j_{2,\ep}$ being
  $j_2=\nabla \frac{\delta F(\rho)}{\delta \rho}$, for a suitable $F$. A justification of this can be found in the analysis of the
  Vlasov-Fokker-Planck equation, see for example~\cite{Monmarche2017a, Duong2013a}. In contrast to our setting, the
  Vlasov-Fokker-Planck equation is derived by relying on the empirical density defined on the combined position--momentum state
  space, $\tilde{\rho}_N(x,y,t)=N^{-1}\sum_{i=1}^{N}{\delta(x-q_i(t),y-p_i(t))}$. In this work, we only use the position-dependent
  quantities~\eqref{eq:102}--\eqref{eq:120}, as this results in a more reduced model with half the spatial dimension (i.e.,
  position as only space variable).  In addition, we do not perform the aforementioned hydrodynamic limit, but then have to close
  the processes $j_{2,\ep}$ (for fixed $N$) using an approximation in the context of a low temperature regime for the underlying
  Langevin dynamics, see Subsection~\ref{ss:200}.
\end{rem}

Subsection~\ref{sec:13} is devoted to adapting the proofs of Proposition~\ref{p:1}, Proposition~\ref{p:22}, and
Theorem~\ref{thm:1} under Assumption~(NG) instead of Assumption~(G).  Finally, in Subsection~\ref{ss:200} we give suitable
approximations of the components of~\eqref{eq:80} in order to obtain expressions closed in $\rho_{\ep}$, $j_{\ep}$, $V$.

\emph{Mild solutions to the regularised Dean--Kawasaki model in a periodic setting:} In Section~\ref{s:5}, we build on the
contents of Subsection~\ref{ss:200}. We work on a periodic domain, in the case of a large number of particles $N$. We define the
\emph{regularised Dean--Kawasaki model}
\begin{subequations}
  \label{eq:420}
  \begin{empheq}[left={}\empheqlbrace]{align}
    &\,\,  \displaystyle\frac{\partial \rho_{\ep}}{\partial t}(x,t)  = -\frac{\partial j_{\ep}}{\partial x}(x,t), 
    \qquad x\in D=[0,2\pi],\,\,t\in[0,T],\label{eq:420a} \vspace{0.7 pc}\\
    &\,\, \displaystyle\frac{\partial j_{\ep}}{\partial t}(x,t)  
    = -\gamma j_{\ep}(x,t)-\left(\frac{\sigma^2}{2\gamma}\right)\frac{\partial \rho_{\ep}}{\partial x}(x,t)-
    V_{\per}'(x)\rho_{\ep}(x,t)
    +\frac{\sigma}{\sqrt{N}}\sqrt{\rho_{\ep}(x,t)}\,\tilde{\xi}_{\per,\ep}, \label{eq:420b}\vspace{0.6 pc}\\
    &\,\,   \displaystyle\rho_{\ep}(x,0)=\rho_0(x),\quad j_{\ep}(x,0)=j_0(x).\nonumber
  \end{empheq}
\end{subequations}
Note that in addition to the approximations made in Subsection~\ref{ss:200}, we have also replaced $\tilde{\xi}_{\ep}$ and $V$
with $\tilde{\xi}_{\per,\ep}$ and $V_{\per}$, the latter two being $2\pi$-periodic versions of the former. This is a natural choice
for the analysis of the equations on a periodic domain. 

\begin{rem}
  Equation~\eqref{eq:420} is a stochastic wave equation. Yet, standard well-posedness results for stochastic partial equations
  cannot be applied in a straightforward way. Firstly, unlike the stochastic heat equation with non-Lipschitz noise
  coefficient~\cite{Shiga1994a}, equation~\eqref{eq:420} does not have a sufficiently regular Green function associated with its
  linear drift operator. This results in standard semigroup techniques not being able to provide well-posedness results
  for~\eqref{eq:420}, due to the presence of the non-Lipschitz noise in~\eqref{eq:420b}. Secondly, the theory of rough paths and
  paracontrolled distributions appears to be inapplicable, again due to the non-Lipschitz noise. Finally, the very nature of the
  wave equation does not seem to prevent $\rho$ from becoming negative (e.g., a suitable maximum principle appears to be
  unavailable), thus it unclear whether the noise is well-defined.
\end{rem}

We prove various preliminary results associated with the existence theory
for~\eqref{eq:420}. These include the semigroup analysis associated with the deterministic integrand
of~\eqref{eq:420} %(when $V\equiv0$)
in Subsection~\ref{ss:10}, a discussion on the choice of a spatially periodic noise in Subsection~\ref{ss:40}, the analysis of the
stochastic integrand of~\eqref{eq:420} in Subsection~\ref{ss:81}, preliminary existence and uniqueness results in
Subsection~\ref{ss:50}, and \emph{a priori} estimates in Subsections~\ref{ss:51} and~\ref{ss:52}. Our key result, provided in
Subsection~\ref{ss:100}, is the following.

\begin{theorem}[High-probability existence and uniqueness result]
  \label{thm:100}
  Let $D=[0,2\pi]$. Let $X_0=(\rho_0,j_0)\in H^1_{\per}(D)\times H^1_{\per}(D)$ be a deterministic initial condition, where
  $H^1_{\per}(D)$ denotes $2\pi$-periodic functions in $H^1(D)$. Assume that $\rho_0(x)\geq \eta$,$\mbox{ for all } x\in D$, for
  some $\eta>0$. Let the scaling $N\ep^{\theta}\geq 1$ be satisfied for some $\theta>7$, and let $\nu\in(0,1)$. It is possible to
  choose a sufficiently large number of particles $N$ such that there exists a unique $H^1_{\per}(D)\times H^1_{\per}(D)$-valued
  mild solution $X_{\ep}=(\rho_{\ep},j_{\ep})$ satisfying equation~\eqref{eq:420} up to a time $T=T(X_0)$ on a set
  $F_{\nu}\in\mathcal{F}$ such that $\mathbb{P}(F_{\nu})\geq 1-\nu$. That is to say, the \emph{regularised Dean--Kawasaki
    model}~\eqref{eq:420} is satisfied path-wise by a unique process $X_{\ep}$ on a set of probability at least $1-\nu$.
\end{theorem}

For the reader's convenience, we summarise how we addressed the three difficulties of the original Dean--Kawasaki model. Firstly,
we work in a function setting, thus the noise $\mathcal{\dot Y}_N$ is well-defined. Secondly, we do not combine the differential
equations associated with $\rho_{\ep}$~\eqref{eq:420a} and $j_{\ep}$~\eqref{eq:420b}, in contrast with~\cite{Lutsko2012a}. On the
contrary, we solve system~\eqref{eq:420} for the couple $(\rho_{\ep},j_{\ep})$, thus avoiding the formal application of the
divergence operator for the stochastic noise of~\eqref{eq:80}. Finally, we prove the above-mentioned high-probability existence
and uniqueness result for~\eqref{eq:420}.

The existence result of this paper is restricted to one spatial dimensional, $d=1$. This restriction comes from Sobolev
embeddings, as we point out in Section~\ref{s:5}.

Finally, Appendix~\ref{ap:1} contains basic facts about Gaussian random variables, while Appendix~\ref{sec:4} contains technical
auxiliary results that are repeatedly used for the derivation of the regularised Dean--Kawasaki model carried out in
Section~\ref{s:4}.

\begin{rem}
  The assumptions of our main results (i.e., Proposition~\ref{p:1} and \ref{p:22}, and Theorems~\ref{thm:1} and~\ref{thm:100}) are
  concerned with different scalings for the regularisation in $\ep$, namely $N\ep^{\theta}=1$ for some $\theta$, see
  Figure~\ref{fig:1}. The lower the value of $\theta$, the more general and less demanding the regularisation is. We motivate
  these scalings from the specific function spaces which are involved in the proofs of the aforementioned results. In this work,
  we do not fully analyse the optimality of such scalings (i.e., the indentification of the lowest admissible value of
  $\theta$). We limit ourselves to providing general comments on this matter in Remark~\ref{rem:101}.
 \end{rem}

%% Comment FC8 here

\section{Basic notation and assumptions}
\label{s:2}

\subsection{Basic Notation}
\label{sec:3}

We may use the same notation for different constants, even within the same line of computation. The dependence of a constant on
given parameters will be highlighted only when it is relevant. We use the symbol $\|\cdot\|$ to denote the norm in
$\mathbb{R}^d$. We use the symbol $\langle\cdot,\cdot\rangle$ to refer to the standard inner product in $\mathbb{R}^d$. For
$x\in\mathbb{R}$, we define $\langle x\rangle :=\sqrt{1+x^2}$.  The symbol $\mathbb{E}\left[X\right]$ denotes the expectation of a
$\mathbb{R}^d$-valued random variable $X$ defined on the probability space $(\Omega,\mathcal{F},\mathbb{P})$. For two
$\mathbb{R}^d$-valued random variables $X,Y$, we denote the covariance matrix (respectively, correlation matrix) of $X$ and $Y$ by
$\mbox{Cov}(X,Y)$ (respectively, $\mbox{Corr}(X,Y)$). For a real-valued random variable $X$, we abbreviate
$\mbox{Var}(X):=\mbox{Cov}(X,X)$.  We will use the symbol $\sim$ to indicate equivalence of laws for random variables. In
particular, we write $X\sim\mathcal{N}(\mu,\sigma^2)$ for a Gaussian random variable $X$ of mean $\mu$ and variance $\sigma^2$. We
write $\mathcal{G}(y,\mu,\sigma^2)$ to denote the probability distribution function of $X\sim\mathcal{N}(\mu,\sigma^2)$, namely
$\mathcal{G}(y,\mu,\sigma^2):=(2\pi\sigma^2)^{-1/2}\exp\left\{-(y-\mu)^2/(2\sigma^2)\right\}$. Quite often, we will use the
short-hand notation $w_{\epsilon}(y):=\mathcal{G}(y,0,\epsilon^2)$, for $\ep>0$. For $X\sim\mathcal{N}(\mu,\sigma^2)$, we define
its \emph{absolute} moments $M(n,\mu,\sigma^2):=\mathbb{E}\left[|X|^n\right]$ and \emph{plain} moments
$m(n,\mu,\sigma^2):=\mathbb{E}\left[X^n\right]$, for any $n\in\mathbb{N}\cup\{0\}$.  For a vector $\mu\in\mathbb{R}^d$ and a
symmetric semi-positive definite matrix $\Sigma\in\mathbb{R}^{d\times d}$, we write $X\sim\mathcal{N}(\mu,\Sigma)$ to denote an
$\mathbb{R}^d$-valued Gaussian random vector with mean $\mu$ and covariance matrix $\Sigma$. For a domain $A\subset\mathbb{R}$, we
use the standard notation $L^p(A)$ and $H^n(A)$ (for $p\in[1,\infty]$ and $n\in\mathbb{N}$) to denote the $L^p$-spaces on $A$ and
the Sobolev spaces of functions on $A$ with square integrable weak derivatives up to order $n$. We denote $n$ times continuously
differentiable functions on $A$ by $C^n(A)$ (for $n\in\mathbb{N}\cup\{\infty\}\cup\{0\}$).

\subsection{Assumptions on the Langevin dynamics}
\label{sec:2} 

We consider the following two different sets of assumptions associated with the Langevin dynamics~\eqref{eq:24}, and in particular
with the choice of potential $V$.

\begin{assumg*}[Gaussian setting for vanishing potential $V$]
  \label{ass:1} Let $T>0$. The potential $V$ vanishes, $V \equiv 0$. Moreover, the initial condition $(q_0,p_0)$ to~\eqref{eq:24}
  is such that the solution $(q(t),p(t))$ to~\eqref{eq:24} satisfies
  \begin{enumerate}
  \item \label{it:G1} $(q(t),p(t))\mbox{ is a bivariate Gaussian vector}$, for all $t\in[0,T]$.
  \item \label{it:G2} There exist $\iota>\nu>0$ such that $\nu\leq\mbox{\emph{Var}}[q(t)]\leq \iota$, for all $t\in[0,T]$.
  \item \label{it:G3} The following quantities are Lipschitz on $[0,T]$: the expected values $\mu_q(t):=\mean{q(t)}$ and
    $\mu_p(t):=\mean{p(t)}$, the variances $\sigma^2_q(t):=\emph{Var}[q(t)]$ and $\sigma^2_p(t):=\emph{Var}[p(t)]$, and the
    correlation $\chi(t):=\emph{Corr}(q(t),p(t))$.
  \end{enumerate}
\end{assumg*}
This assumption holds generically for the Ornstein-Uhlenbeck process dynamics, see Lemma~\ref{la:6}.

\begin{assumng*}[Non-Gaussian setting for rapidly diverging $V(q)\approx |q|^{2n}$] 
  \begin{enumerate}
  \item \label{it:NG1} The potential $V$ is a $C^{\infty}(\mathbb{R})$-function. Furthermore, there exists $n\in\mathbb{N}$ such
    that, for all $k\in\mathbb{N}$, there exists a constant $C_k$ such that
    \begin{align*}
      \left|\frac{\partial^k V(q)}{\partial q^k}\right|\leq C_k\left(1+\langle q \rangle^{2n-\min\{2,k\}}\right),
      \qquad \mbox{ for all } q\in\mathbb{R}.
    \end{align*}
  \item \label{it:NG2} There exist two constants $C_0(V), C_1(V)>0$ such that
    \begin{align*}
      V(q)\geq C^{-1}_0\langle q\rangle^{2n}-C_0,\qquad \left|\frac{\partial V(q)}{\partial q}\right|
      \geq C^{-1}_1\langle q\rangle^{2n-1}-C_1,\qquad \mbox{ for all } q\in\mathbb{R}.
    \end{align*}
  \item \label{it:NG3} The joint density $g_0$ of the initial condition $(q_0,p_0)$ to~\eqref{eq:24} coincides with
    $\overline{g}(\overline{t},q,p)$, where $\overline{t}$ is some positive time and $\overline{g}(\overline{t},q,p)$ is the
    solution at time $\overline{t}$ to the Fokker--Planck equation
    \begin{align}
      \label{eq:134}
      \frac{\partial g}{\partial t}=-\nabla\cdot(g\mu)+\frac{\sigma^2}{2}\frac{\partial^2 g}{\partial p^2},
      \qquad \mu:=\left(p,-\gamma p-V'(q)\right),\qquad g(0,q,p)=\overline{g}_0(q,p),
    \end{align}
    started from some initial condition $\overline{g}_0\in M^{1/2}H^{-5,-5}(\mathbb{R}^2)$. The notation $H^{s,s}(\mathbb{R}^2)$,
    $s>0$, denotes the $s^{th}$-order member of the isotropic Sobolev chain defined in \emph{\cite[Eq.~(3)]{Herau2004a}}, while
    the weight function $M(q,p)\propto\exp\left\{-(2\gamma/\sigma^2)\left(p^2/2+V(q)\right)\right\}$ is the Gibbs invariant
    measure of~\eqref{eq:134}.
  \item \label{it:NG4} We have that $\lim_{q\rightarrow+\infty}{V(q)/V(-q)}$ exists and is finite.
  \end{enumerate}
\end{assumng*}

Items~\ref{it:NG1} and~~\ref{it:NG2} of the Assumption~(NG) are slightly more restrictive than those of~\cite[Hypotheses
1]{Herau2004a}. In particular, we assume the potential $V$ to diverge at infinity with no less than quadratic growth. This is
encapsulated in the requirement $n\geq1$ (instead of the requirement $n>1/2$ made in~\cite[Hypotheses 1]{Herau2004a}). Item~\ref{it:NG3} implies regularity of the initial condition $g_0$.

We briefly justify the choice of the above two sets of hypotheses as follows. Assumption~(G) guarantees the applicability of tools
inherently associated with the theory of Gaussian random variables. Then many computations can be made explicit in a relatively
straightforward way. On the other hand, Assumption~(NG) is more general. Our analysis under Assumption~(NG) is an extension of the
argument previously carried out under Assumption~(G). Both these assumptions will play a role in the derivation of the regularised
Dean--Kawasaki model in Section~\ref{s:4}.

%% Comment FC1 here

\section{Derivation of the regularised Dean--Kawasaki model}
\label{s:4}

We now derive the \emph{regularised Dean--Kawasaki} model studied in this paper. In Subsection~\ref{sec:6}, under Assumption~(G),
we prove a tightness result for the relevant quantities~\eqref{eq:102},~\eqref{eq:120}, as well as uniqueness of the limit for the
family $\{\rho_{\ep}\}_{\ep}$. These results are Propositions~\ref{p:1} and~\ref{p:22}. The proof of Proposition~\ref{p:1} is
nontrivial but also technical, and might be skipped at a first reading. Subsection~\ref{sec:7} motivates the derivation of the
noise $\mathcal{\dot Y}_N$, which we introduced in~\eqref{eq:3002}. In Subsection~\ref{sec:8}, under Assumption~(G), we prove
Theorem~\ref{thm:1}, which quantifies the difference between the noises $\mathcal{\dot Y}_N$ and $\mathcal{\dot Z}_N$ (see
also~\eqref{eq:80}). In Subsection~\ref{sec:13} we adapt the proofs of Propositions~\ref{p:1},~\ref{p:22}, and Theorem~\ref{thm:1}
under Assumption~(NG). Finally, Subsection~\ref{ss:200} gathers the relevant information from the earlier parts of
Section~\ref{s:4} in order to define a regularised Dean--Kawasaki model.

\subsection{Tightness of leading quantities: proofs of Proposition~\ref{p:1} and Proposition~\ref{p:22}}
\label{sec:6}

We prove some Kolmogorov-type tightness estimates for the families $\{\rho_\ep\}_{\ep}$, $\{j_\ep\}_{\ep}$ and
$\{j_{2,\ep}\}_{\ep}$. The arguments are somewhat technical; as we are not aware of closely related results in the literature, we
describe the proofs in some detail.

\begin{proof}[Proof of Proposition~\ref{p:1} under Assumption~(G)]
  We verify the assumption of ~\cite[Corollary 14.9]{Kallenberg2002a} for the families $\{\rho_\ep\}_{\ep}$, $\{j_\ep\}_{\ep}$,
  $\{j_{2,\ep}\}_{\ep}$. More specifically, for each family, we prove a suitable Kolmogorov time-regularity condition, as well as
  tightness of the processes at time 0.

\emph{Step 1: Tightness of $\{\rho_\ep\}_{\ep}$}.
We use the expansion of a square and the independence of the particles to write
\begin{align*}
  & \mean{\| \rho_\epsilon(\cdot,t)-\rho_\epsilon(\cdot,s)\|_{L^2(\mathbb{R})}^2} =  
    \frac{1}{N^2}\mean{\!\int_{\mathbb{R}}{\!\sum_{i,j=1}^{N}{\!\left[w_\epsilon(x-q_i(t))
    -w_\epsilon(x-q_i(s))\right]\left[w_\epsilon(x-q_j(t))-w_\epsilon(x-q_j(s))\right]\m x}}}\nonumber\\
  & =  \frac{1}{N^2}\sum_{i=1}^{N}{\mean{\left\|w_\epsilon(\cdot-q_1(t))-w_\epsilon(\cdot-q_1(s))\right\|^2_{L^2(\mathbb{R})}}} \nonumber\\
  & \quad +\frac{1}{N^2}\sum_{i\neq j}{\int_{\mathbb{R}}{\mean{w_\epsilon(x-q_i(t))-w_\epsilon(x-q_i(s))}
    \mean{w_\epsilon(x-q_j(t))-w_\epsilon(x-q_j(s))}\m x}}.
\end{align*}

Given the identical distribution of the particles, we deduce 
\begin{align}
  \label{eq:3000}
  & \mean{\| \rho_\epsilon(\cdot,t)-\rho_\epsilon(\cdot,s)\|_{L^2(\mathbb{R})}^2} \nonumber\\
  & \quad =  \frac{1}{N}\mean{\left\|w_\epsilon(\cdot-q_1(t))-w_\epsilon(\cdot-q_1(s))\right\|^2_{L^2(\mathbb{R})}}
    +\frac{1}{N^2}\sum_{i\neq j}\left\|\mean{w_\epsilon(\cdot-q_1(t))-w_\epsilon(\cdot-q_1(s))}\right\|_{L^2(\mathbb{R})}^2\nonumber\\
  & \quad \leq  \frac{1}{N}\underbrace{\mean{\left\|w_\epsilon(\cdot-q_1(t))-w_\epsilon(\cdot-q_1(s))\right\|^2_{L^2(\mathbb{R})}}}_{=: I_1}
    +\underbrace{\left\|\mean{w_\epsilon(\cdot-q_1(t))-w_\epsilon(\cdot-q_1(s))}\right\|_{L^2(\mathbb{R})}^2}_{=: \mathsf{ct}}.
\end{align}
There are two main differences between the term $I_1$ and the ``cross-term'' contribution $\mathsf{ct}$. Firstly, term $I_1$ is of
the form $\mean{\|\cdot\|^p_{L^p(\mathbb{R})}}$, while term $\mathsf{ct}$ is of the form
$\|\mean{\cdot}\|^p_{L^{p}(\mathbb{R})}$. Secondly, term $\mathsf{ct}$ has no decaying scaling factor in $N$. This means that we
are forced to provide a bound for $\mathsf{ct}$ which is \emph{independent} of $\ep$. This bound is provided by invoking
Lemmas~\ref{lem:2} and~\ref{lem:1}. On the other hand, we are allowed to bound $I_1$ with quantities which might diverge in $\ep$
(these appear because of the form $\mean{\|\cdot\|^p_{L^p(\mathbb{R})}}$, as we will point out), as long as they can be
compensated by the scaling in $N$. These considerations are quite general, and we will apply similar reasonings at several points
later on in the proof, as well as point out the relevant analogies when needed.

We occasionally drop the particle index, because of the identical distribution. We proceed to bound $I_1$ and $\mathsf{ct}$. Using
the elementary inequality
\begin{align}
  \label{eq:6}
  1-e^{-x^2}\leq x^2,\quad\mbox{ for all } x\in\mathbb{R},
\end{align}
we rewrite $I_1$ as 
\begin{align}
  \label{eq:12}
  &\mean{\left\|w_\epsilon(\cdot-q(t))-w_\epsilon(\cdot-q(s))\right\|_{L^2(\mathbb{R})}^2} =
    \mean{\int_{\mathbb{R}}{w^2_{\ep}(x-q(t))+w^2_{\ep}(x-q(s))-2w_{\ep}(x-q(t))w_{\ep}(x-q(s))}}\nonumber\\
  & = \frac{1}{\sqrt{\pi\epsilon^2}}\,\mean{1-\exp\left(\frac{-(q(t)-q(s))^2}{4\epsilon^2}\right)}
    \leq  \frac{C}{\epsilon^3}\mathbb{E}\left[|q(t)-q(s)|^2\right]
    \leq\frac{C}{\epsilon^3}|t-s|^2,
\end{align}
where we have used Lemma~\ref{la:4} and an integration in $x$ in the last equality, and~\eqref{eq:6} in the first inequality. In
addition, $q$ satisfies, by definition, the integral equation $q(t)-q(s)=\int_{s}^{t}{p(z)\m z}$. The integrability properties of
$p$ (Assumption~(G)) and the H\"older inequality hence give the final inequality in~\eqref{eq:12}.  As for the cross-terms
$\mathsf{ct}$, we employ Lemma~\ref{lem:2}, estimate~\eqref{eq:4}, and then apply Lemma~\ref{lem:1} to deduce
\begin{align*}
  \norm{\mean{w_\epsilon(\cdot-q(t))-w_\epsilon(\cdot-q(s))}}_{L^2(\mathbb{R})}^2 
  =\int_{\mathbb{R}}{\left|\mathcal{G}(x,\mu(t),\sigma^2_q(t)+\epsilon^2)-\mathcal{G}(x,\mu(s),
  \sigma^2_q(s)+\epsilon^2)\right|^2\!\m x} \leq  C|t-s|^{2}.
\end{align*}
We combine the estimates for $\mathsf{ct}$ and $I_1$ and obtain, thanks to the prescribed scaling $N\ep^3\geq 1$,
\begin{align*}
  \mean{\|\rho_\epsilon(\cdot,t)-\rho_\epsilon(\cdot,s)\|_{L^2(\mathbb{R})}^2} 
  \leq C\left(\frac{1}{N\epsilon^3}+1\right)|t-s|^2\leq C|t-s|^2, 
\end{align*}
and the time regularity is settled using Kolmogorov's continuity theorem.  We now need to show that
$\{\rho_{\ep}(\cdot,0)\}_{\ep}$ is tight in $L^2(D)$. We rely on the compact embedding $H^1(D)\subset L^2(D)$, see~\cite[Theorem
6.3]{Adams2003a}, and we show that $\mean{\|\rho_{\ep}(\cdot,0)\|^2_{H^1(\mathbb{R})}}$ is uniformly bounded in $\ep$. A
computation analogous to~\eqref{eq:3000} gives
\begin{align}
  \label{eq:2002}
  &\mean{\|\rho_{\ep}(\cdot,0)\|^2_{H^1(\mathbb{R})}} = \mean{\|\rho_{\ep}(\cdot,0)\|^2_{L^2(\mathbb{R})}}+\mean{\left\|\frac{\partial \rho_{\ep}}{\partial x}(\cdot,0)\right\|
    ^2_{L^2(\mathbb{R})}} \nonumber\\
  & \quad \leq \frac{1}{N}\underbrace{\mean{\int_{\mathbb{R}}{w^2_{\ep}(x-q_1(0))\m x
    +w'^2_{\ep}(x-q_1(0))\m x}}}_{=:I_1} +\underbrace{\int_{\mathbb{R}}{\mean{w_{\ep}(x-q_1(0))}^2
    +\mean{w'_{\ep}(x-q_1(0))}^2\m x.}}_{=:\mathsf{ct}}
\end{align}
The bound $I_1\leq C\ep^{-3}$ follows from Lemma~\ref{la:4}, in combination with the integration in $x$ and the definition of the
Gaussian moments, see Lemma~\ref{la:5}. The term $\mathsf{ct}$ can be bounded uniformly in $\ep$ using Lemma~\ref{lem:2},
estimates~\eqref{eq:4} and~\eqref{eq:130}. The scaling $N\ep^3\geq 1$ finally implies tightness for $\{\rho_{\ep}\}_{\ep}$.

\emph{Step 2: Tightness of $\{j_\ep\}_{\ep}$}. For notational convenience, we define
\begin{equation*}
  \tau_i(x,s,t):=p_i(t)w_\epsilon(x-q_i(t))-p_i(s)w_\epsilon(x-q_i(s)),
\end{equation*}
so that $j_{\ep}(x,t)-j_{\ep}(x,s)=N^{-1}\sum_{i=1}^{N}{\tau_i(x,s,t)}$. In the same fashion as~\eqref{eq:3000}, we expand
\begin{align}
  \label{eq:11}
  & \mean{\norm{ j_\epsilon(\cdot,t)-j_\epsilon(\cdot,s)}_{L^4(\mathbb{R})}^4} \leq 
    \frac{1}{N^3}\underbrace{\int_{\mathbb{R}}{\mean{\tau_1(x,s,t)^4}
    \m x}}_{=:I_1}+\frac{C}{N^2}\underbrace{\int_{\mathbb{R}}{\mean{\left|\tau_1(x,s,t)\right|}\mean{\left|\tau^3_1(x,s,t)\right|}
    \m x}}_{=:I_2}\nonumber\\
  &\quad +  \frac{C}{N^2}\underbrace{\int_{\mathbb{R}}{\mean{\tau^2_1(x,s,t)}^2}\m x}_{=:I_3} + 
    \frac{C}{N}\underbrace{\int_{\mathbb{R}}{\mean{\tau_1(x,s,t)}^2\mean{\tau^2_1(x,s,t)}}\m x}_{=:I_4}%\nonumber\\
  +\underbrace{\int_{\mathbb{R}}{\mean{\tau_1(x,s,t)}^4}\m x}_{=:\mathsf{ct}}.
\end{align}
The discussion following~\eqref{eq:3000} applies analogously to the family of terms $I_1$, $I_2$, $I_3$ and $I_4$, which do
contain at least one term of the form $\mean{\tau_{i}(x,s,t)^p}$, and to the term $\mathsf{ct}$, which is of the form
$\|\mean{\cdot}\|^p_{L^{p}(\mathbb{R})}$. We thus provide an $\ep$-independent bound for $\mathsf{ct}$, and suitable
$\ep$-diverging bounds for $I_1$, $I_2$, $I_3$ and $I_4$.

The conditional density for bivariate Gaussian random variables, stated in Lemma~\ref{la:3}, implies
\begin{equation}
  \label{eq:15}
  f_{p(t)|q(t)}(p|q(t)=b)=\mathcal{G}\left(p,\mu_p(t)+\frac{\sigma_p(t)}{\sigma_q(t)}\chi(t)(b-\mu_q(t)),(1-\chi(t)^2)\sigma^2_p(t)\right),
  \qquad \mbox{for all }b\in\mathbb{R}.
\end{equation}
We use the law of total expectation and~\eqref{eq:15} %applied to $(q(t),p(t))$ 
to compute
\begin{align}
  \label{eq:5}
  \mean{p(t)w_{\epsilon}(x-q(t))} 
  & = \mean{\mean{p(t)w_{\epsilon}(x-q(t))|q(t)}} =  
    \mean{w_{\epsilon}(x-q(t))\left(\mu_p(t)+\frac{\sigma_p(t)}{\sigma_q(t)}\chi(t)(q(t)-\mu_q(t))\right)}
    \nonumber\\
  & =  a_1(t)\mean{w_{\epsilon}(x-q(t))}+a_2(t)\mean{w_{\epsilon}(x-q(t))q(t)},
\end{align}
where we set
\begin{equation*}
  a_1(t):=\mu_p(t)-\frac{\sigma_p(t)}{\sigma_q(t)}\chi(t)\mu_q(t),\quad a_2(t):=\frac{\sigma_p(t)}{\sigma_q(t)}\chi(t).
\end{equation*}
The time-dependent coefficients $a_1$ and $a_2$ are Lipschitz, thanks to Assumption~(G). Keeping in mind Remark~\ref{rem:1}, we
use Lemma~\ref{lem:2}, estimate~\eqref{eq:4} and then Lemma~\ref{lem:1}.  We deduce %% Comment FC9 here
\begin{align}
  \label{eq:7}
  \mathsf{ct}\leq C|t-s|^{1+\beta},
\end{align}
for some $\beta\in(0,1)$. 

We now treat the $\ep$-diverging terms $I_1$, $I_2$, $I_3$ and $I_4$ in~\eqref{eq:11}. By adding and subtracting the quantity
$2p(t)p(s)w_{\ep/\sqrt{2}}(x-(q(t)+q(s))/2)$, using~\eqref{eq:6}, and integrating in $x$, we obtain
\begin{align}
  \label{eq:1001}
  \int_{\mathbb{R}}{\mean{\tau^2_1(x,s,t)}\m x} 
  & =\frac{1}
    {\sqrt{4\pi \ep^2}}\mean{\int_{\mathbb{R}}{p^2(t)w_{\frac{\ep}{\sqrt{2}}}(x-q(t))
    +p^2(s)w_{\frac{\ep}{\sqrt{2}}}(x-q(s))}\m x}+\nonumber\\
  & \quad - \frac{1}{\sqrt{4\pi\ep^2}}\mean{\int_{\mathbb{R}}{2p(t)p(s)\exp\left\{-\frac{(q(t)-q(s))^2}{4\ep^2}\right\}
    w_{\frac{\ep}{\sqrt{2}}}\left(x-\frac{q(t)+q(s)}{2}\right)}\m x}\nonumber\\
  & = \frac{1}{\sqrt{4\pi\ep^2}}\mean{|p(t)-p(s)|^2}+\frac{1}{\sqrt{4\pi\ep^2}}\mean{2p(s)p(t)
    \left(1-\exp\left\{-\frac{(q(t)-q(s))^2}{4\ep^2}\right\}\right)}\nonumber\\
  & \leq \frac{1}{\sqrt{4\pi\ep^2}}\mean{|p(t)-p(s)|^2}+\frac{C}{\ep^3}\mean{2p(s)p(t)\left|q(t)-q(s)\right|^2}.
\end{align}
The first expectation in the last line of~\eqref{eq:1001} satisfies $\mean{|p(t)-p(s)|^2}\leq C|t-s|$. This is implied by the
It\^o isometry, which we invoke because $p$ satisfies, by definition, the stochastic integral equation
$p(t)-p(s)=\int_{s}^{t}{-\gamma p(z)\m z}+\sigma\int_{s}^{t}{\m \beta(z)}$. Note the difference in time regularity with the
previously discussed $\mean{|q(t)-q(s)|^2}$, see~\eqref{eq:12}. As for the second expectation in the last line of~\eqref{eq:1001},
we may use the H\"older inequality on the probability space to separate $p(s)p(t)$ from $\left|q(t)-q(s)\right|^2$. Using again
the integrability of $p$ granted by Assumption~(G) and the H\"older inequality in time for $q(t)-q(s)$, we deduce
\begin{align}
  \label{eq:8}
  \int_{\mathbb{R}}{\mean{\tau^2_1(x,s,t)}\m x}\leq \frac{C}{\ep}|t-s|+\frac{C}{\ep^3}|t-s|^{2}.
\end{align}

In addition, we have the bound ${\mean{\tau_1(x,s,t)}^2}\leq C|t-s|$, where $C$ is independent of $x$ and $\ep$. This can be
justified by relying on~\eqref{eq:5}, using the fact that right-hand-side of~\eqref{eq:4} (for $X$ being the process $q$) is
Lipschitz in time, with Lipschitz constant independent of $\ep$ and $x$, as explained in Remark~\ref{rem:1}. Hence,
using~\eqref{eq:8}, we deduce that
\begin{align*}
  I_4\leq \frac{C}{N\ep^3}|t-s|^2.
\end{align*}

We have completed the analysis for $I_4$, which is the term that requires the most care, due to the fact that it is paired with
the slowest decay in $N$ as coefficient. As for the other terms $I_1$, $I_2$ and $I_3$, we need not provide sharp bounds. By
repeatedly applying the H\"older inequality on the probability space $\Omega$, we deduce that $I_2$ and $I_3$ are bounded by
$I_1$. We therefore only need to provide an estimate for $I_1$ in order to conclude Step~(ii).  We write
\begin{align}
  \label{eq:13}
  I_1 
  \leq C\,\mean{\int_{\mathbb{R}}{(p(t)-p(s))^4w^4_{\epsilon}(x-q(t))}\m x} +  
  C\,\mean{\int_{\mathbb{R}}{p(s)^4(w_{\epsilon}(x-q(t))-w_{\epsilon}(x-q(s)))^4}\m x}.
\end{align}

We reuse some algebraic computations from~\eqref{eq:12} to continue as
\begin{align*}
  I_1& \leq  C\,\mean{\int_{\mathbb{R}}{(p(t)-p(s))^4w^4_{\epsilon}(x-q(t))}\m x} 
       + C\,\mean{\int_{\mathbb{R}}{p(s)^4(w_{\epsilon}(x-q(t))-w_{\epsilon}(x-q(s)))^4}\m x}\\
     & \leq  \frac{C}{\epsilon^4}\mean{(p(t)-p(s))^4} 
       + C\mean{p^4(s)\frac{C}{\epsilon^2}\int_{\mathbb{R}}{(w_{\epsilon}(x-q(t))-w_{\epsilon}(x-q(s)))^2}\m x}\\
     & \leq  \frac{C}{\epsilon^4}\mean{(p(t)-p(s))^4} 
       + \frac{C}{\epsilon^2}\mean{p^4(s)\frac{1}{\epsilon}\left(1-\exp\left(-\frac{(q(t)-q(s))^2}{4\epsilon^2}\right)\right)}\\
     & \leq  \frac{C}{\epsilon^4}\mean{(p(t)-p(s))^4} + \frac{C}{\epsilon^5}\mean{p^4(s)(q(t)-q(s))^2}\\
     & \leq \frac{C}{\epsilon^4}|t-s|^2+\frac{C}{\epsilon^5}\mean{p^8(s)}^{1/2}\mean{(q(t)-q(s))^4}^{1/2}\leq \frac{C}{\epsilon^5}|t-s|^2.
\end{align*}
In particular, we have used the bound $\max_{y}{w_{\ep}(y)}\leq C\ep^{-1}$ in the second inequality, Lemma~\ref{la:4} in the third
inequality,~\eqref{eq:6} in the fourth inequality, and integrability properties of $p$ and $q$ in the fifth and sixth inequality.
The scaling $N\ep^3\geq 1$ concludes the time regularity analysis for $\{j_{\ep}\}_{\ep}$. As for the tightness of
$\{j_{\ep}(\cdot,0)\}_{\ep}$, we deal with the analogous expression of~\eqref{eq:2002} for $\{j_{\ep}\}_{\ep}$. The analysis is
similar, apart from the use of Lemma~\ref{la:3} prior to the use of Lemma~\ref{lem:2} (for the corresponding term $\mathsf{ct}$)
and the use of the compact embedding $H^1(D)\subset L^4(D)$.

\emph{Step 3: Tightness of $\{j_{2,\ep}\}_{\ep}$}. For notational convenience, we define 
\begin{align*}
  \tau_i(x,s,t):=p^2_i(t)w'_\epsilon(x-q_i(t))-p^2_i(s)w'_\epsilon(x-q_i(s)),
\end{align*}
so that $j_{2,\ep}(x,t)-j_{2,\ep}(x,s)=N^{-1}\sum_{i=1}^{N}{\tau_i(x,s,t)}$. In the same fashion as~\eqref{eq:11}, we expand
\begin{align}
  \label{eq:14}
  & \mean{\norm{ j_{2,\epsilon}(\cdot,t)-j_{2,\epsilon}(\cdot,s)}_{L^4(\mathbb{R})}^4} 
    \leq \frac{1}{N^3}\underbrace{\int_{\mathbb{R}}{\mean{\tau_1(x,s,t)^4}}\m x}_{=:I_1}
    +\frac{C}{N^2}\underbrace{\int_{\mathbb{R}}{\mean{\left|\tau_1(x,s,t)\right|}
    \mean{\left|\tau^3_1(x,s,t)\right|}}\m x}_{=:I_2}\nonumber\\
  & \quad + \frac{C}{N^2}\underbrace{\int_{\mathbb{R}}{\mean{\tau^2_1(x,s,t)}^2}\m x}_{=:I_3} 
    + \frac{C}{N}\underbrace{\int_{\mathbb{R}}{\mean{\tau_1(x,s,t)}^2\mean{\tau^2_1(x,s,t)}}\m x}_{=:I_4}%\nonumber\\
  +  \underbrace{\int_{\mathbb{R}}{\mean{\tau_1(x,s,t)}^4}\m x}_{=:\mathsf{ct}}.
\end{align}
The considerations for $I_1$, $I_2$, $I_3$ and $I_4$ and $\mathsf{ct}$ are analogous to the ones for the homonymous counterparts
in~\eqref{eq:11}. In order to estimate $\mathsf{ct}$, we need to compute $\mean{p^2(t)w'_{\epsilon}(x-q(t))}$. We again rely on
the conditional law~\eqref{eq:15} and the law of total expectation to write
\begin{align}
  \label{eq:16}
  \mean{p^2(t)w'_{\epsilon}(x-q(t))} 
  & = \mean{\mean{p^2(t)w'_{\epsilon}(x-q(t))|q(t)}}\nonumber\\
  & = \mean{w'_{\epsilon}(x-q(t))\left\{(\mu_p(t)+\frac{\sigma_p(t)}{\sigma_q(t)}\chi(t)(q(t)-\mu_q(t)))^2
    +(1-\chi^2(t))\sigma^2_p(t)\right\}}.
\end{align}
The right-hand-side of~\eqref{eq:16}, thanks to Assumption~(G), Lemma~\ref{lem:2} and Remark~\ref{rem:1}, is of the form
prescribed by Lemma~\ref{lem:1}.  Hence we deduce %% Comment FC10 here
\begin{align*}
  \mathsf{ct}\leq C|t-s|^{1+\beta},\qquad\mbox{for some }\beta>0.
\end{align*}

The analysis of terms $I_1$, $I_2$, $I_3$, $I_4$ in~\eqref{eq:14} is similar to the one we carried out for the homonymous terms
in~\eqref{eq:11}. We set $\tilde{q}:=(q(t)+q(s))/2$ and use Lemma~\ref{la:4} to compute
\begin{align*}
  \int_{\mathbb{R}}{\mean{\tau^2_1(x,s,t)}\m x} 
  & = \frac{1}{\sqrt{4\pi \ep^2}}\frac{1}{\ep^4}\left\{\mean{\int_{\mathbb{R}}{p^4(t)w_{\frac{\ep}{\sqrt{2}}}(x-q(t))(q(t)-x)^2
    +p^4(s)w_{\frac{\ep}{\sqrt{2}}}(x-q(s))(q(s)-x)^2}\m x}\right.+\nonumber\\
  & \quad -\left.2\mean{\int_{\mathbb{R}}{p^2(t)p^2(s)\exp\left\{-\frac{(q(t)-q(s))^2}{4\ep^2}\right\}
    w_{\frac{\ep}{\sqrt{2}}}\left(x-\tilde{q}\right)\underbrace{(q(t)-x)(q(s)-x)}_{=:T_1}}\m x}\right\}.
\end{align*}
We add and subtract $\tilde{q}$ in both brackets of $T_1$. Similarly to the argument in~\eqref{eq:1001}, we rely on the
$x$-integration with Gaussian kernels, the trivial bound $e^{z}\leq 1$ for $z\leq 0$, and we continue the above estimate
\begin{align}
  \label{eq:51} 
  \int_{\mathbb{R}}{\mean{\tau^2_1(x,s,t)}\m x} 
  &\leq\frac{C}{\ep^3}\mean{p^4(t)+p^4(s)-2p^2(t)p^2(s)+2p^2(t)p^2(s)
    \left(1-\exp\left\{-\frac{(q(t)-q(s))^2}{4\ep^2}\right\}\right)}\nonumber\\
  & \quad +\frac{C}{\ep^5}\mean{\int_{\mathbb{R}}{p^2(t)p^2(s)\exp\left\{-\frac{(q(t)-q(s))^2}{4\ep^2}\right\}
    \left|q(t)-q(s)\right|^2}\m x}\nonumber\\
  & \leq \frac{C}{\ep^3}\mean{\left|p^2(t)-p^2(s)\right|^2}+\frac{C}{\ep^5}\mean{p^2(t)p^2(s)|q(t)-q(s)|^2}.
\end{align}
Similarly to the argument for~\eqref{eq:1001}, we get 
\begin{align}
  \label{eq:10}
  \int_{\mathbb{R}}{\mean{\tau^2_1(x,s,t)}\m x} \leq \frac{C}{\ep^3}|t-s|+\frac{C}{\ep^5}|t-s|^2.
\end{align}

Using an identical argument to the proof concerning $\{j_{\ep}\}_{\ep}$, we have that ${\mean{\tau_1(x,s,t)}^2}\leq C|t-s|$, where
$C$ is independent of $x$ and $\ep$. In combination with~\eqref{eq:10}, this yields
\begin{align*}
  I_4\leq \frac{C}{\ep^5}|t-s|^2.
\end{align*}
By repeatedly applying the H\"older inequality on the probability space $\Omega$, we deduce that $I_2,I_3$ are bounded by
$I_1$. We therefore only need to provide an estimate for $I_1$ in order to conclude Step~(iii). We write
%% Comment FC11 here
\begin{align}
  \label{eq:17}
  I_1 
  \leq C\,\mean{\int_{\mathbb{R}}{(p^2(t)-p^2(s))^4w'^4_{\epsilon}(x-q(t))}\m x}
  +  C\,\mean{\int_{\mathbb{R}}{p(s)^8(w'_{\epsilon}(x-q(t))-w'_{\epsilon}(x-q(s)))^4}\m x}. 
\end{align}
We notice that $\max_{y}\left|w'_\epsilon(y)\right|\leq C\epsilon^{-2}$. We rely on some computations in~\eqref{eq:51} and bound
$I_1$ as
\begin{align*}
  I_1& \leq  C\,\mean{\int_{\mathbb{R}}{(p^2(t)-p^2(s))^4w'^4_{\epsilon}(x-q(t))}\m x} 
       + C\,\mean{\int_{\mathbb{R}}{p(s)^8(w'_{\epsilon}(x-q(t))-w'_{\epsilon}(x-q(s)))^4}\m x}\\
     & \leq  \frac{C}{\epsilon^{8}}\mean{(p(t)-p(s))^4(p(t)+p(s))^4} 
       + \frac{C}{\epsilon^4}\,\mean{p^8(s)\int_{\mathbb{R}}{\left|w'_{\epsilon}(x-q(t))-w'_{\epsilon}(x-q(s))\right|^2}\m x}\\
     & \leq  \frac{C}{\epsilon^{8}}\mean{(p(t)-p(s))^4(p(t)+p(s))^4} + \frac{C}{\epsilon^4}\,\mean{\frac{C}{\ep^5}p^8(s)|q(t)-q(s)|^2} \\
       %\leq \frac{C}{\epsilon^{9}}|t-s|^{1+\beta}\\
     & \leq  \frac{C}{\epsilon^{8}}\mean{(p(t)-p(s))^8}^{1/2}\mean{(p(t)+p(s))^8}^{1/2} 
       + \frac{C}{\epsilon^9}\mean{p^{16}(s)}^{1/2}\mean{|q(t)-q(s)|^4}^{1/2} \leq \frac{C}{\epsilon^{9}}|t-s|^{1+\beta},
\end{align*}
where we have also used the Burkholder-Davis-Gundy inequality to estimate $\mean{(p(t)-p(s))^8}$. The required time regularity is
established. As for the tightness of $\{j_{2,\ep}(\cdot,0)\}_{\ep}$, we can deal with the analogous expression of~\eqref{eq:2002}
for $\{j_{2,\ep}\}_{\ep}$. The analysis is similar, apart from the use of Lemma~\ref{la:3} prior to the use of Lemma~\ref{lem:2}
(for the corresponding term $\mathsf{ct}$) and the use of the compact embedding $H^1(D)\subset L^4(D)$.
\end{proof}

\begin{rem}
  \label{rem:3]}
  The scaling $N^{-1}$ involved in the definitions of $\rho_{\ep}$ and $j_{\ep}$ is crucial for the tightness for
  $\{\rho_\ep\}_{\ep}$, $\{j_\ep\}_{\ep}$ and $\{j_{2,\ep}\}_{\ep}$. This scaling differs from the original Dean--Kawasaki
  derivation with non-rescaled leading quantities~\eqref{eq:2}.
\end{rem}

\begin{rem}
  \label{rem:2}
  The scaling (of $\ep$ and $N$) associated with the family $\{j_{2,\ep}\}_{\ep}$ is more restrictive than the one associated with
  the family $\{j_{\rho}\}_{\ep}$; this is due to the need to estimate quantities related to derivatives of the kernel
  $w_{\ep}$. The different hypotheses on $\theta$ are justified by the computations associated with term $I_1$ (in the case of
  $\{\rho_{\ep}\}_{\ep}$) and by the computations associated with term $I_4$ (in the case of $\{j_{\ep}\}_{\ep}$ and
  $\{j_{2,\ep}\}_{\ep}$).  The scalings of Proposition~\ref{p:1} are compatible with the assumptions of our key result,
  Theorem~\ref{thm:100}.
\end{rem}

\begin{proof}[Proof of Proposition~\ref{p:22} under Assumption~(G)] 
Prohorov's theorem~\cite[Theorem 14.3]{Kallenberg2002a} and Proposition~\ref{p:1} imply weak convergence up to subsequences for
the family $\{\eta_{\ep}\}_{\ep}$ in $\mathcal{X}$ as $\ep\rightarrow 0$. In order to conclude the proof, we need to prove
uniqueness of the weak limit $\eta$. Let us take two sequences $\{(a_{n}, N^a_n)\}_n$ and $\{(b_{n},N^b_n)\}_n$ satisfying the
scaling prescribed in the hypothesis, and such that $\eta_{a_n}\stackrel{w}{\rightarrow} \eta_1$ and
$\eta_{b_n}\stackrel{w}{\rightarrow} \eta_2$ in $\mathcal{X}$. In order to show that $\eta_1=\eta_2$, we just need to show that
the finite-dimensional laws coincide, see~\cite[Proposition 2.2]{Kallenberg2002a}. Let $\pi$ be a projection from $\mathcal{X}$
onto a finite but arbitrary number of times $0\leq t_1\leq\cdots\leq t_m\leq T$. Take a bounded Lipschitz function
$g\colon X^m:=[L^2(D)]^m\rightarrow \mathbb{R}$. Then
\begin{align}
  \label{eq:1002}
  & \left|\int_{\mathcal{X}}{g(\pi(p))\m\eta_{a_n}(p)}-\int_{\mathcal{X}}{g(\pi(p))\m\eta_{b_n}(p)}\right|^2 
    = \left|\mean{g(\pi(\rho_{a_n}))}-\mean{g(\pi(\rho_{b_n}))}\right|^2 \nonumber\\
  & \leq L(g)\mean{\left\|\pi(\rho_{a_n})-\pi(\rho_{b_n})\right\|_{[L^2(D)]^m}}^2
    \leq L(g)\sum_{j=1}^{m}{\mean{\int_{\mathbb{R}}{\left(\rho_{a_n}(x,t_j)-\rho_{b_n}(x,t_j)\right)^2\m x}}},
\end{align}
where we have used the H\"older inequality in the last step. Let us denote $N^M_n:=\max{\{N^a_n;N^b_n\}}$ and
$N^m_n:=\min{\{N^a_n;N^b_n\}}$. For each $j\in\{1,\cdots,m\}$, we expand the square of the sum of $N^a_n+N^b_n$ terms in the $j$th
term of~\eqref{eq:1002}. As $N^a_n$ and $N^b_n$ might differ, it is convenient to split the resulting $(N^a_n+N^b_n)^2$ product
terms into six different categories. We have
\begin{itemize}
\item 
  $N^a_n$ terms of type $(N^a_n)^{-2}w^2_{a_n}(x-q_i(t_j))$,
  \vspace{-0.3pc}
\item 
  $N^b_n$ terms of type $(N^b_n)^{-2}w^2_{b_n}(x-q_i(t_j))$,
  \vspace{-0.3pc}
\item 
  $2N^m_n$ terms of type $-(N^M_nN^m_n)^{-1}w_{a_n}(x-q_i(t_j))w_{b_n}(x-q_i(t_j))$,
  \vspace{-0.3pc}
\item 
  $N^a_n(N^a_n-1)$ terms of type $(N^a_n)^{-2}w_{a_n}(x-q_i(t_j))w_{a_n}(x-q_k(t_j))$, where $i\neq k$,
  \vspace{-0.3pc}
\item 
  $N^b_n(N^b_n-1)$ terms of type $(N^b_n)^{-2}w_{b_n}(x-q_i(t_j))w_{b_n}(x-q_k(t_j))$, where $i\neq k$,
  \vspace{-0.3pc}
\item 
  $2N^M_nN^m_n-2N^m_n$ terms of type $-(N^M_nN^m_n)^{-1}w_{a_n}(x-q_i(t_j))w_{b_n}(x-q_k(t_j))$, where $i\neq k$.
\end{itemize}
With the help of Lemma~\ref{la:4} and the scaling of $\{(a_{n}, N^a_n)\}_n$ and $\{(b_{n},N^b_n)\}_n$, we deduce that the contributions of the first three families to the right-hand-side of~\eqref{eq:1002} vanish in the limit $n\rightarrow\infty$. The contribution of the remaining three families is given by
\begin{align}
  \label{eq:2022}
  & \sum_{j=1}^{m}{\left\{\frac{N^a_n(N^a_n-1)}{(N^a_n)^2}\mean{w_{\sqrt{2}a_n}(q_1(t_j)-q_2(t_j))}
    +\frac{N^b_n(N^b_n-1)}{(N^b_n)^2}\mean{w_{\sqrt{2}b_n}(q_1(t_j)-q_2(t_j))}\right.}\nonumber\\
  & \quad \left.- \frac{2N^M_nN^m_n-2N^m_n}{N^M_nN^m_n}\mean{w_{\sqrt{a^2_n+b^2_n}}(q_1(t_j)-q_2(t_j))}\right\}.
\end{align}

The probability density functions of the random variables $q_1(t_j)-q_2(t_j)$, $j=1,\cdots,m$, which we denote by
$f_{q_1(t_j)-q_2(t_j)}$, belong to the Schwartz space $\mathcal{S}$ (i.e., the space of rapidly decaying real-valued functions on
$\mathbb{R}$). This can be justified as follows. The density of the sum of two continuous independent real-valued random variables
is given by the convolution of the densities of the two random variables. %~\cite[Theorem 7.1]{Grinstead2012a}.
In addition, for $f_1,f_2\in\mathcal{S}$ we have that also $f_1\ast f_2\in\mathcal{S}$. As a consequence of Assumption~(G), the
laws of $q_1(t_j)$ and $-q_2(t_j)$, $j=1,\cdots,m$, are Gaussian, and hence they belong to $\mathcal{S}$. We can then rewrite the
expectations in~\eqref{eq:2022} with dualities in $\mathcal{S}'$, and we deduce the convergence of the $j$th term of the sum to
\begin{align*}
  f_{q_1(t_j)-q_2(t_j)}(0)+f_{q_1(t_j)-q_2(t_j)}(0)-2f_{q_1(t_j)-q_2(t_j)}(0)=0,\qquad j\in\{1,\cdots,m\}
\end{align*} 
by means of the convergence
$w_{\ep}\rightarrow\delta$ in $\mathcal{S'}$ for $\ep\rightarrow 0$. This leads to
\begin{align*}
  \int_{X^m}{g(z)\m(\pi_{\ast}\eta_1)(z)} 
  &=\lim_{n\rightarrow \infty}{\int_{X^m}{g(z)\m(\pi_{\ast}\eta_{a_n})(z)}}
    =\lim_{n\rightarrow \infty}{\int_{X^m}{g(z)\m(\pi_{\ast}\eta_{b_n})(z)}}= \int_{X^m}{g(z)\m(\pi_{\ast}\eta_2)(z)},
\end{align*}
where $\pi_{\ast}$ indicates a push-forward of measures by $\pi$. Uniqueness of weak limits implies that $\pi_{\ast}\eta_1$ and
$\pi_{\ast}\eta_2$ (the projections of $\eta_1$ and $\eta_2$ onto $\{t_1,\cdots,t_m\}$) coincide. Since the times involved are
arbitrary, we deduce $\eta_1\equiv\eta_2$. This concludes the proof.
\end{proof}

\subsection{Noise replacement in evolution system for $(\rho_{\ep},j_{\ep})$}
\label{sec:7}

We now replicate the analysis described in \emph{Steps 1--2} of Subsection~\ref{sec:10} adapted to the setting considered here, in
order to derive a regularised Dean--Kawasaki model. It is straightforward to derive system~\eqref{eq:80} using the It\^o calculus
on $\rho_{\ep}$ and $j_{\ep}$.  System~\eqref{eq:80} is similar to the system of evolution equations for the original quantities
$\rho_N$ and $j_N$, see~\cite[Eq.~(4)]{Lutsko2012a}.  In particular, in analogy to the original derivation of the Dean--Kawasaki
model, the noise term $\mathcal{\dot Z}_N=\sigma N^{-1}\sum_{i=1}^{N}{w_{\ep}(x-q_i(t))\dot \beta_i}$ is not a closed expression
of the leading quantities $\rho_{\ep}$ and $j_{\ep}$. For this reason, we replace $\mathcal{\dot Z}_N$ with a multiplicative
noise, which we initially take to be of the form
\begin{align}
  \label{eq:29}
  \frac{\sigma}{\sqrt{N}}f(\rho_{\ep})Q^{1/2}_\ep\,\xi,
\end{align}
where $\xi$ is a space-time white noise, $f\colon\mathbb{R}\rightarrow \mathbb{R}$ is to be determined, and $Q_\ep$ is suitable
spatial operator to be determined as well. In order to understand the above chosen structure, we first compute the spatial
covariance for $\mathcal{Z}_N$. For given points $x_1,x_2\in\mathbb{R}$, we have
\begin{align*}
  \mean{\mathcal{Z}_{N}(x_1,t)\mathcal{Z}_{N}(x_2,t)} 
  & = \mean{\left(\int_{0}^{t}{\frac{\sigma}{N}\sum_{i=1}^{N}{w_{\ep}(x_1-q_i(u))
    \m \beta_i(u)}}\right)\left(\int_{0}^{t}{\frac{\sigma}{N}\sum_{i=1}^{N}{w_{\ep}(x_2-q_i(u))\m \beta_i(u)}}\right)}\nonumber\\
  & = \frac{\sigma^2}{N^2}\mean{\sum_{i=1}^{N}{\left(\int_{0}^{t}{w_{\ep}(x_1-q_i(u))\m\beta_i(u)}\right)
    \left(\int_{0}^{t}{w_{\ep}(x_2-q_i(u))\m\beta_i(u)}\right)}}\nonumber\\
  & \quad +\frac{\sigma^2}{N^2}\mean{\sum_{i\neq j}
    {\left(\int_{0}^{t}{w_{\ep}(x_1-q_i(u))\m\beta_i(u)}\right)\left(\int_{0}^{t}{w_{\ep}(x_2-q_j(u))\m\beta_j(u)}\right)}}\nonumber\\
  & = \frac{\sigma^2}{N^2}\mean{\sum_{i=1}^{N}{\int_{0}^{t}{w_{\ep}(x_1-q_i(u))w_{\ep}(x_2-q_i(u))\m u}}},
\end{align*}
where in the last equality we have used basic It\^o calculus, as well as the fact that stochastic integrals driven by independent
noises are uncorrelated. Lemma~\ref{la:4} gives
$w_{\ep}(x_1-q_i(u))w_{\ep}(x_2-q_i(u))=w_{\sqrt{2}\ep}(x_1-x_2)w_{\ep/\sqrt{2}}(q_i(u)-(x_1+x_2)/2)$, for all $i=1,\cdots,N$. By
summing over $i=1,\cdots,N$ and dividing by $N$, we conclude that
\begin{align*}
  N^{-1}\sum_{i=1}^{N}{w_{\ep}(x_1-q_i(u))w_{\ep}(x_2-q_i(u))}=w_{\sqrt{2}\ep}(x_1-x_2)\rho_{\ep/\sqrt{2}}((x_1+x_2)/2,u).
\end{align*} 

We deduce
\begin{align}
  \label{eq:28}
  \mean{\mathcal{Z}_{N}(x_1,t)\mathcal{Z}_{N}(x_2,t)} 
  = w_{\sqrt{2}\ep}(x_1-x_2)\int_{0}^{t}{\mean{\frac{\sigma^2}{N}\rho_{\ep/\sqrt{2}}\left(\frac{x_1+x_2}{2},u\right)}\m u}.
\end{align}
Equation~\eqref{eq:28} indicates how to define the multiplicative noise~\eqref{eq:29}. The term $ w_{\sqrt{2}\ep}(x_1-x_2)$ is
deterministic. It is then not unreasonable to assume that such a term can be associated with the covariance structure for the
stochastic noise in~\eqref{eq:29}. On the other hand, the random variable in the right-hand-side of~\eqref{eq:28} should,
according to It\^o calculus, be the square of the stochastic integrand of~\eqref{eq:29} evaluated at $(x_1+x_2)/2$. We thus
propose the following noise replacement for $\mathcal{Z}_{N}$
\begin{equation*}
  \mathcal{\dot Y}_N:=\frac{\sigma}{\sqrt{N}}\sqrt{\rho_{\ep/\sqrt{2}}}\,\underbrace{Q^{1/2}_{\sqrt{2}\ep}\,\xi}_{\tilde{\xi}_{\ep}},
\end{equation*} 
where $Q_{\sqrt{2}\ep}$ is a convolution operator with kernel $w_{\sqrt{2}\ep}$. The domain of such an operator is specified in
the statement of Theorem~\ref{thm:1}, whose proof is provided in the next subsection.

\begin{rem}
  Note that $\tilde{\xi}_{\ep}$ is a spatially correlated noise approximating the action of a space-time white noise for small
  values of $\ep$. Also note the scaling $\ep/\sqrt{2}$, as opposed to the original scaling $\ep$, characterising
  $\rho_{\ep/\sqrt{2}}$ in the definition of noise $\mathcal{\dot Y}_N$. The factor $\sqrt{2}$ appears for simple analytical
  reasons. This will not affect our considerations for the limit $\ep\rightarrow 0$, $N\rightarrow \infty$, as we will point out
  in Subsection~\ref{ss:200}.
\end{rem}

\subsection{Covariance error bound associated with noise replacement}
\label{sec:8}

The main modelling result concerns a thorough comparison of the stochastic noises $\mathcal{\dot Z}_N$ and the noise
$\mathcal{\dot Y}_N$ just introduced.  Specifically, we estimate the ``price'' one has to pay in order to replace $\mathcal{Z}_N$
with $\mathcal{Y}_N$ in~\eqref{eq:80}. More specifically, we are interested in quantifying the size of
$\mathcal{R}_N=\mathcal{Z}_N-\mathcal{Y}_N$ and $\mathcal{Y}_N$ in terms of $\ep,N$. Our goal is to prove that, in the limit of
$\ep \to 0$ and $N \to \infty$, the remainder $\mathcal{R}_{N}$ is negligible with respect to $\mathcal{Y}_N$. As a consequence,
exchanging the stochastic noises results in a negligible correction.

\begin{proof}[Proof of Theorem~\ref{thm:1} under Assumption~(G)] The convolution operator $Q_{\sqrt{2}\ep}$ is defined as
  $Q_{\sqrt{2}\ep}\colon L^2(D)\rightarrow L^2(D)\colon f\mapsto Q_{\sqrt{2}\ep}f(\cdot):=\int_{D}{w_{\sqrt{2}\ep}(\cdot-y)f(y)\m
    y}$.
  We compare the noises $\mathcal{Z}_N$, $\mathcal{Y}_N$ by means of their spatial covariance structures at any given time
  $t\in[0,T]$, for any couple of points $x_1,x_2\in D$. Following on the construction in the previous section, we have
\begin{align*}
  \mean{\mathcal{Z}_N(x_1,t)\mathcal{Z}_N(x_2,t)} 
  & = \frac{\sigma^2}{N}w_{\sqrt{2}\ep}(x_1-x_2)\int_{0}^{t}{\mean{\rho_{\ep/\sqrt{2}}\left(\frac{x_1+x_2}{2},s\right)}\m s},
\end{align*}
and with similar arguments one finds
\begin{align*}
  \mean{\mathcal{Y}_N(x_1,t)\mathcal{Y}_N(x_2,t)} 
  & =\frac{\sigma^2}{N}w_{\sqrt{2}\ep}(x_1-x_2)\int_{0}^{t}{\mean{\sqrt{\rho_{\ep/\sqrt{2}}(x_1,s)\rho_{\ep/\sqrt{2}}(x_2,s)}}\m s}.
\end{align*}
We notice that the two covariances share the common prefactor $\sigma^2N^{-1}w_{\sqrt{2}\ep}(x_1-x_2)$. Our analysis will thus be
focused on the terms where the two expressions differ. If we want to evaluate the difference of the two above covariance
expressions, it is useful to study, for any given time $s\in[0,t]$,
\begin{align}
  \label{eq:30}
  \mean{\left|\rho_{\ep/\sqrt{2}}\left(\frac{x_1+x_2}{2},s\right)-\sqrt{\rho_{\ep/\sqrt{2}}(x_1,s)\rho_{\ep/\sqrt{2}}(x_2,s)}\right|}.
\end{align}

For notational convenience, we define $m:=(x_1+x_2)/2$ and drop the time dependence for $\rho_{\ep/\sqrt{2}}$. We add and subtract
$\rho_{\ep/\sqrt{2}}(m)$ to both $\rho_{\ep/\sqrt{2}}(x_1)$ and $\rho_{\ep/\sqrt{2}}(x_2)$. As a result, the random variable
in~\eqref{eq:30} turns into
\begin{align*}
  \left|\rho_{\ep/\sqrt{2}}\left(m\right)-\sqrt{\rho^2_{\ep/\sqrt{2}}\left(m\right)+b(x_1,x_2)}\right| 
  & = \left|\rho_{\ep/\sqrt{2}}\left(m\right)\right|\left(1-\sqrt{1+\frac{b(x_1,x_2)}{\rho^2_{\ep/\sqrt{2}}\left(m\right)}}\right)
    \leq \frac{\left|b(x_1,x_2)\right|}{\rho_{\ep/\sqrt{2}}\left(m\right)},
\end{align*}
where we have defined 
\begin{align*}
b(x_1,x_2) & := \rho_{\ep/\sqrt{2}}(m)\!\left[\rho_{\ep/\sqrt{2}}(x_1)+\rho_{\ep/\sqrt{2}}(x_2)-2\rho_{\ep/\sqrt{2}}(m)\right] 
 + (\rho_{\ep/\sqrt{2}}(x_1)-\rho_{\ep/\sqrt{2}}(m))(\rho_{\ep/\sqrt{2}}(x_2)-\rho_{\ep/\sqrt{2}}(m)).
\end{align*}
We can thus bound the random variable in~\eqref{eq:30} by the sum
\begin{align}
  \label{eq:33}
  \left|\rho_{\ep/\sqrt{2}}(x_1)+\rho_{\ep/\sqrt{2}}(x_2)-2\rho_{\ep/\sqrt{2}}(m)\right| 
  + \frac{\left|\rho_{\ep/\sqrt{2}}(x_1)-\rho_{\ep/\sqrt{2}}(m)\right|\left|\rho_{\ep/\sqrt{2}}(x_2)
  -\rho_{\ep/\sqrt{2}}(m)\right|}{\rho_{\ep/\sqrt{2}}(m)}=:T_1+T_2.
\end{align}

\emph{Expected value of term $T_2$}. We use the H\"older inequality twice and we obtain
\begin{align}
  \label{eq:64}
  \mean{T_2}\leq\mean{\rho^{-2}_{\ep/\sqrt{2}}(m)}^{\frac{1}{2}}\mean{\left|\rho_{\ep/\sqrt{2}}(x_1)
  -\rho_{\ep/\sqrt{2}}(m)\right|^4}^{\frac{1}{4}}\mean{\left|\rho_{\ep/\sqrt{2}}(x_2)-\rho_{\ep/\sqrt{2}}(m)\right|^4}^{\frac{1}{4}}.
\end{align}
The first expectation in the right-hand-side of~\eqref{eq:64} can be bounded, independently of $N,\ep$, by means of
Proposition~\ref{p:2}. The two remaining expectations in~\eqref{eq:64} are identical up to a swap of $x_1$ and $x_2$, hence we
analyse just one of them.

In analogy to some computations previously carried out for~\eqref{eq:11} and~\eqref{eq:14}, we set
$ \tau(x_1,m):=w_{\ep}(x_1-q_1(s))-w_{\ep}(m-q_1(s))$. We expand
\begin{align}
  \label{eq:65}
  \mean{\left|\rho_{\ep/\sqrt{2}}(x_1)-\rho_{\ep/\sqrt{2}}(m)\right|^4} 
  & \leq \frac{1}{N^3}\underbrace{\mean{\tau^4(x_1,m)}}_{=:I_1}+\frac{C}{N^2}\underbrace{\mean{\left|\tau(x_1,m)\right|}
    \mean{\left|\tau^3(x_1,m)\right|}}_{=:I_2}\nonumber\\
  &\quad +  \frac{C}{N^2}\underbrace{\mean{\tau^2(x_1,m)}^2}_{=:I_3} 
    + \frac{C}{N}\underbrace{\mean{\tau(x_1,m)}^2\mean{\tau^2(x_1,m)}}_{=:I_4}
  +\underbrace{\mean{\tau(x_1,m)}^4}_{=:\mathsf{ct}}.
\end{align}
Note the absence of integration in $x$, as opposed to~\eqref{eq:11} and~\eqref{eq:14}. We use Lemma~\ref{lem:2} and a first-order
Taylor approximation in space together with Assumption~(G)~\ref{it:G2}, to deduce
\begin{align*}
  \left|\mean{\tau(x_1,m)}\right|=\left|\mathcal{G}(x_1,\mu_q(s),\sigma^2_q(s)+\ep^2)-\mathcal{G}(m,\mu_q(s),\sigma^2_q(s)
  +\ep^2)\right|\leq C|x_1-x_2|.
\end{align*}
We rely on Lemma~\ref{la:4}, Lemma~\ref{lem:2} to write $\mean{\tau^2(x_1,m)}$ as 
\begin{align*}
  & \frac{1}{\sqrt{4\pi\ep^2}}\mean{w_{\ep/\sqrt{2}}(x_1-q_1(s))+w_{\ep/\sqrt{2}}(m-q_1(s))
    -2w_{\ep/\sqrt{2}}\left[\frac{x_1+m}{2}-q_1(s)\right]\exp\left\{-\frac{(x_1-m)^2}{4\ep^2}\right\}}\nonumber\\
  & \quad = \frac{1}{\sqrt{4\pi\ep^2}}\left\{\mathcal{G}(x_1,\mu_q(s),\sigma^2_q(s)+\ep^2/2)+\mathcal{G}(m,\mu_q(s),\sigma^2_q(s)
    +\ep^2/2)+2\mathcal{G}\left(\frac{x_1+m}{2},\mu_q(s),\sigma^2_q(s)+\ep^2/2\right)\right\}\\
  & \quad\quad +\frac{2}{\sqrt{4\pi\ep^2}}\mathcal{G}\left(\frac{x_1+m}{2},\mu_q(s),\sigma^2_q(s)
    +\ep^2/2\right)\left\{1-\exp\left\{-\frac{(x_1-m)^2}{4\ep^2}\right\}\right\}.
\end{align*}
We use a second-order approximation of the type $\left|f(x_1)+f(m)-2f((x_1+m)/2)\right|\leq C|x_1-m|^2$ applied to
$f(x)=\mathcal{G}(x,\mu_{q}(s),\sigma_q^2(s)+\ep^2/2)$, as well as inequality~\eqref{eq:6}, to deduce
\begin{align}
  \label{eq:2000}
  \mean{\tau^2(x_1,m)}\leq C\left(\frac{1}{\ep}+\frac{1}{\ep^3}\right)|x_1-x_2|^2\leq \frac{C}{\ep^3}|x_1-x_2|^2.
\end{align}
The bound $\max_{y}{|w'_{\ep}(y)|}\leq C\ep^{-2}$, the mean-value theorem and~\eqref{eq:2000} allow us to deduce
\begin{align*}
  \mean{\tau^4(x_1,m)}\leq \frac{C}{\ep^4}|x_1-x_2|^2\mean{\tau^2(x_1,m)}\leq \frac{C}{\ep^7}|x_1-x_2|^4.
\end{align*}
The above estimate is the most demanding in terms of the scaling $N,\ep$, and justifies the hypothesis $\theta\geq 7/2$. Finally,
the terms $\mean{\left|\tau^3(x_1,m)\right|}, \mean{\left|\tau(x_1,m)\right|}$ can be bounded, by means of the H\"older
inequality, by $\mean{\tau^4(x_1,m)}^{3/4}$ and $\mean{\tau^4(x_1,m)}^{1/4}$ respectively. We can put all these estimates together
for the benefit of $I_1$, $I_2$, $I_3$, $I_4$ and $\mathsf{ct}$ in~\eqref{eq:65} and obtain
\begin{align*}
  \mean{\left|\rho_{\ep/\sqrt{2}}(x_1)-\rho_{\ep/\sqrt{2}}(m)\right|^4}\leq C|x_1-x_2|^4.
\end{align*}

The estimate for points $x_2$ and $m$ replacing $x_1$ and $m$ is identical. As a result of the above observations, we can bound
the left-hand-side in~\eqref{eq:64}, thus obtaining
\begin{align}
  \label{eq:68}
  T_2\leq C|x_1-x_2|^2,
\end{align}
for $C$ independent of $N$ and $\ep$.

\emph{Expected value of term $T_1$}. Using similar arguments to the analysis of $T_2$, it is not difficult to show that
\begin{align*}
  \mean{\left|\rho_{\ep/\sqrt{2}}(x_1)+\rho_{\ep/\sqrt{2}}(x_2)-2\rho_{\ep/\sqrt{2}}(m)\right|} 
  & \leq \mean{\left|\rho_{\ep/\sqrt{2}}(x_1)+\rho_{\ep/\sqrt{2}}(x_2)-2\rho_{\ep/\sqrt{2}}(m)\right|^2}^{1/2}\leq C|x_1-x_2|^2
\end{align*}
by using a fourth-order approximation of the type $|f(x_1)+f(x_2)+6f(m)-4f(m_1)-4f(m_2)|\leq C|x_1-x_2|^4$, where
$x_1<m_1<m<m_2<x_2$ are equi-distanced. We skip the details. We combine the estimates for $T_1$ and $T_2$ and deduce
\begin{align*}
  \left|\mean{\mathcal{Z}_N(x_1,t)\mathcal{Z}_N(x_2,t)}-\mean{\mathcal{Y}_N(x_1,t)\mathcal{Y}_N(x_2,t)}\right|
  \leq \frac{C\sigma^2}{N}w_{\sqrt{2}\ep}(x_1-x_2)|x_1-x_2|^2,
\end{align*}
which is exactly~\eqref{eq:71}. Using Lemma~\ref{lem:1}, it is also immediate to notice that
\begin{align*}
  \left|\mean{\mathcal{Z}_N(x_1,t)\mathcal{Z}_N(x_2,t)}\right|\leq \frac{C\sigma^2}{N}w_{\sqrt{2}\ep}(x_1-x_2),
\end{align*}
which is~\eqref{eq:72}, and the proof of Theorem~\ref{thm:1}~\ref{it:cov1} is complete. The proof of~\ref{it:cov2} is a
straightforward consequence of the estimate $N^{-1}w_{\sqrt{2}\ep}(x_1,x_2)\leq \ep^{\theta-1}$ and
of~\eqref{eq:71},~\eqref{eq:72}.
\end{proof} 

\begin{rem}
  The proof of Theorem~\ref{thm:1} employs a multiplicative approach for the estimation of the random variable
  in~\eqref{eq:30}. We rely on the estimate $\left|\sqrt{a^2}-\sqrt{a^2+c}\right|\leq |c/a|$, instead of using the standard
  estimate
  \begin{align}
    \label{eq:70}
    \left|\sqrt{a^2}-\sqrt{a^2+c}\right|\leq\sqrt{|c|}.
  \end{align}
  In our specific case, we have $a:=\rho_{\ep/\sqrt{2}}(m)$ and $c:=b(x_1,x_2)$. The multiplicative approach has the disadvantage
  of having the term $a^{-1}$ ($\rho^{-1}_{\ep/\sqrt{2}}(m)$ for us) in the bound. For this reason, we need to prove that $a$ is
  bounded away from 0, and this is the reason why Proposition~\ref{p:2} is needed. On the other side, the multiplicative approach
  provides sharper estimates (in terms of orders of power of $|x_1-x_2|$) for the estimation of the difference of the spatial
  covariances of noises $\mathcal{Z}_N$ and $\mathcal{Y}_N$ in~\eqref{eq:71}, if compared to what we would get if we relied
  on~\eqref{eq:70}. For these reasons, we chose the multiplicative approach.
\end{rem}

\begin{rem}
  The replacement of $\mathcal{Z}_N$ with $\mathcal{Y}_N$ gives a negligible error. This error is given
  by~\eqref{eq:71},~\eqref{eq:72}, depending on the distance $|x_1-x_2|$. We split the analysis in three cases.
  \begin{itemize}
  \item \emph{Points $x_1,x_2\in D$ such that $|x_1-x_2|^2\leq \ep^{2}$}.  Estimates~\eqref{eq:71},~\eqref{eq:72} directly imply
    \begin{align*}
      \left|\mean{\mathcal{Z}_N(x_1,t)\mathcal{Z}_N(x_2,t)}-\mean{\mathcal{Y}_N(x_1,t)\mathcal{Y}_N(x_2,t)}\right|  
      \leq \frac{C}{N}\cdot\frac{1}{\ep}\cdot\ep^{2} 
      & \approx \mathcal{O}(\ep^{\theta+1}),\\
      \left|\mean{\mathcal{Z}_N(x_1,t)\mathcal{Z}_N(x_2,t)}\right| \leq \frac{C}{N}\cdot\frac{1}{\ep} 
      & \approx \mathcal{O}(\ep^{\theta-1}).
    \end{align*}
  \item \emph{Points $x_1,x_2\in D$ such that $|x_1-x_2|^2\in(\ep^2,\ep)$}.  Estimates~\eqref{eq:71},~\eqref{eq:72} directly imply
    \begin{align*}
      \left|\mean{\mathcal{Z}_N(x_1,t)\mathcal{Z}_N(x_2,t)}-\mean{\mathcal{Y}_N(x_1,t)\mathcal{Y}_N(x_2,t)}\right|  
      \leq \frac{C}{N}\cdot\frac{1}{\ep}\cdot \ep & \approx \mathcal{O}(\ep^{\theta}),\\
      \left|\mean{\mathcal{Z}_N(x_1,t)\mathcal{Z}_N(x_2,t)}\right| \leq \frac{C}{N}\cdot\frac{1}{\ep} 
                                                  & \approx \mathcal{O}(\ep^{\theta-1}).
    \end{align*}
  \item \emph{Points $x_1,x_2\in D$ such that $|x_1-x_2|^2\geq \ep$}. The prefactor $N^{-1}w_{\sqrt{2}\ep}(x_1-x_2)$ decays
    exponentially in $\ep$, and both $\mathcal{Z}_N$, $\mathcal{Y}_N$ are negligible, and hence interchangeable.
  \end{itemize}
\end{rem}

\subsection{Non-vanishing potential $V(q)$: modifications of proofs of main results}
\label{sec:13}

We show that Proposition~\ref{p:1}, Proposition~\ref{p:22} and Theorem~\ref{thm:1} also hold with Assumption~(G) replaced by
Assumption~(NG).

\begin{proof}[Adaptation of the proof of Proposition~\ref{p:1} under Assumption~(NG)] 
  In the proof of Proposition~\ref{p:1}, we deal with three time-regularity estimates for the families $\{\rho_{\ep}\}_{\ep}$,
  $\{j_{\ep}\}_{\ep}$, $\{j_{2,\ep}\}_{\ep}$. In each one of them, we expand an $L^p$-norm of the relevant
  quantities~\eqref{eq:102},~\eqref{eq:120}. In each case, we end up with upper bounds consisting of sums of terms labelled as
  $\mathsf{ct}$, $I_1$ (and also $I_2$, $I_3$ and $I_4$ when applicable). If we now assume that $V$ satisfies Assumption~(NG), we
  can use Proposition~\ref{p:3}, bounds~\eqref{eq:104}--\eqref{eq:105}, to deduce the bound $\mathsf{ct}\leq |t-s|^{1+\beta}$ for
  all the three estimates. As for the remaining terms $I_1$ (and $I_2, I_3$ and $I_4$ when applicable), we use
  Proposition~\ref{p:3}, bounds~\eqref{eq:108}--\eqref{eq:117}, to control all terms $\mean{\tau_{1}(x,s,t)}^2$ as
  $\mean{\tau_1(x,s,t)}^2\leq C|t-s|$, with $C$ independent of $x$ and $\ep$. It only remains to consider the integrals of the
  form
\begin{equation*}
    \begin{cases}
      \displaystyle\int_{\mathbb{R}}{\mean{\left(w_{\ep}(x-q(t))-w_{\ep}(x-q(s))\right)^2}\m x}, & \mbox{for Step~1}\vspace{0.2 pc}\\
      \displaystyle\int_{\mathbb{R}}{\mean{\tau_1(x,s,t)^c}\m x},\,\,\,c\in\{2, 3, 4\}, & \mbox{for Steps~2 and~3}.
    \end{cases}
\end{equation*}
The algebraic steps involved in the $x$-variable integration remain unaltered. As for the expected value of the resulting
$(q(t),p(t),q(s),p(s))$-dependent quantities, the time-regularity estimates also do not change. This is a consequence of the
rapidly decaying probability density function $g(t,q,p)$ and the polynomial growth of $V$. These facts guarantee the existence
(and the correct time-dependency) of all the required moments of $q(t)-q(s)$ and $p(t)-p(s)$. As for the proofs of tightness of
$\{\rho_{\ep}(\cdot,0)\}_{\ep}$, $\{j_{\ep}(\cdot,0)\}_{\ep}$, $\{j_{2,\ep}(\cdot,0)\}_{\ep}$, these can be adapted by using
Remark~\ref{rem:200} for the estimates of the terms labelled $\mathsf{ct}$, see for instance~\eqref{eq:2002}.
\end{proof}

\begin{proof}[Adaptation of the proof of Proposition~\ref{p:22} under Assumption~(NG)]
  The only change in the proof is the justification of the probability density functions of $q_1(t_j)$ and $-q_2(t_j)$,
  $j=1,\cdots,m$, belonging to $\mathcal{S}$. This is stated in~\cite[Theorem 0.1]{Herau2004a}.
\end{proof}

%% Comment FC13 here
\begin{proof}[Adaptation of the proof of Theorem~\ref{thm:1} under Assumption~(NG)] The proof is identical up to, and including,
  estimate~\eqref{eq:64}. After that, we work on~\eqref{eq:65} by using the adaptation of Proposition~\ref{p:2} under
  Assumption~(NG), whose proof is included in Subsection~\ref{ss:2}. We also need to provide estimates for the terms $I_1$, $I_2$,
  $I_3$, $I_4$ and $\mathsf{ct}$ without relying on the Gaussian setting. We define $\tilde{g}_t$ to be the probability density
  function of $q(t)$. We begin with $\mathsf{ct}$, and bound
  \begin{align*}
    \left|\mean{\tau(x_1,m)}\right| 
    & = \left|\int_{\mathbb{R}}{\left(w_{\ep}(x_1-y)-w_{\ep}(m-y)\right)\tilde{g}_t(y)
      \m y}\right|=\left|\int_{\mathbb{R}}{w_{\ep}(x_1-y)\left(\tilde{g}_t(y)-\tilde{g}_t(y+m-x_1)\right)\m y}\right|\\
    & \leq \|w_{\ep}(x_1-\cdot)\|_{L^1(\mathbb{R})}\|\tilde{g}_t(\cdot)-\tilde{g}_t(\cdot+m-x_1)\|_{L^{\infty}(\mathbb{R})}\leq C|x_1-x_2|,
\end{align*}
where we have used the change of variables for $q$ in the second equality (shift by $m-x_1$), and the boundedness of
$(\partial/\partial q){g}(q,p,t)$ provided by~\eqref{eq:137}. This concluded the analysis of the term $\mathsf{ct}$. We now turn
to
\begin{align*}
  & \mean{\tau(x_1-m)^2}\\
  & \quad = \frac{1}{\sqrt{4\pi\ep^2}}\mean{w_{\ep/\sqrt{2}}(x_1-q_1(s))+w_{\ep/\sqrt{2}}(m-q_1(s))
    -2w_{\ep/\sqrt{2}}\left[\frac{x_1+m}{2}-q_1(s)\right]\exp\left\{-\frac{(x_1-m)^2}{4\ep^2}\right\}}\\
  & \quad\leq  \frac{1}{\sqrt{4\pi\ep^2}}\int_{\mathbb{R}}{w_{\frac{\ep}{\sqrt{2}}}(x_1-y)\left(\tilde{g}_t(y)+\tilde{g}_t(y+m-x_1,t)
    -2\tilde{g}_t\left(y+\frac{x_1+m}{2}-x_1,t\right)\right)\m y}\\
  & \quad\quad + \frac{1}{\sqrt{4\pi\ep^2}}\frac{(x_1-m)^2}{4\ep^2}\int_{\mathbb{R}}{w_{\frac{\ep}{\sqrt{2}}}
    \left(y-\frac{x_1+x_2}{2}\right)\tilde{g}_t(y)\m y}\leq C\left(\frac{1}{\ep}+\frac{1}{\ep^3}\right)|x_1-x_2|^2
    \leq \frac{C}{\ep^3}|x_1-x_2|^2.
\end{align*}
We have used~\eqref{eq:6}, suitable changes of variables for $q$, and a second-order Taylor approximation for $\tilde{g}_t$ in the
first inequality, as well as boundedness of suitable derivatives of $g(q,p,t)$ by means of~\eqref{eq:137} in the second
inequality.  This settles term $I_3$. The remaining terms $I_1$, $I_2$ and $I_4$ are dealt with in the same way as in the original
proof. The estimation of term $T_1$ can be performed with the same techniques used above in the adaptation of the analysis for
term $T_2$.
\end{proof}

\subsection{Defining the regularised Dean--Kawasaki model}
\label{ss:200}

An immediate consequence of Theorem~\ref{thm:1} is that, in a simultaneous limit of $N\rightarrow \infty$ and $\ep\rightarrow 0$,
the stochastic noise $\mathcal{Z}_N$ in system~\eqref{eq:80} vanishes. This differs from the original Dean--Kawasaki model.
However, a close approximation of such a model is recovered for a large but fixed number of particles $N$, by means of
Theorem~\ref{thm:1}.  We make some additional approximations to~\eqref{eq:80}. These approximations are aimed at deriving a
closed-expression formulation, in the variable $(\rho_{\ep},j_{\ep})$, for our regularised version of the Dean--Kawasaki model.

\emph{Approximation 1}. We replace the noise $\mathcal{Z}_N$ with the noise $\mathcal{Y}_N$ (i.e., we neglect the remainder
$\mathcal{R}_N$). This has been discussed in detail in Subsections~\ref{sec:7} and~\ref{sec:8}.

\emph{Approximation 2}. With respect to the noise $\mathcal{Y}_N$, we replace $\{\rho_{\ep/\sqrt{2}}\}_{\ep}$ with
$\{\rho_{\ep}\}_{\ep}$. This is justified by the fact that both families admit the same limit in distribution in
$\mathcal{X}=C(0,T;L^2(D))$ thanks to Proposition~\ref{p:22}. In addition, the noise $\mathcal{Y}_N$ features the vanishing
rescaling $N^{-1/2}$, which provides an additional contribution in reducing the error caused by the replacement of
$\rho_{\ep/\sqrt{2}}$ with $\rho_{\ep}$.

%% Comment FC2 here

\emph{Approximation 3}. We replace the term $j_{2,\ep}$ with a multiple of $\frac{\partial\rho_{\ep}}{\partial x}$. This can be seen as a replacement of the
random quantity $p_i^2(t)$ with its expected value. Indeed, the equilibrium state of the particle system $\{(q_i,p_i)\}_{i=1}^{N}$
is identified by the joint density
\begin{align*}
  C(N,V,\sigma,\gamma)\prod_{i=1}^{N}{\exp\left\{-\frac{2\gamma}{\sigma^2}\left(\frac{p_i^2}{2}+V(q_i)\right)\right\}}
  =C(N,V,\sigma,\gamma)\prod_{i=1}^{N}{M(q_i,p_i)}.
\end{align*}
The equilibrium state shows independence between position and velocity of particles. This allows to write
\begin{align*}
  \mean{j_{2,\ep}(x,t)}=\mean{p_1^2(t)}\mean{\frac{\partial\rho_{\ep}}{\partial x}(x,t)}=\frac{\sigma^2}{2\gamma}\mean{\frac{\partial\rho_{\ep}}{\partial x}(x,t)},
\end{align*}
which suggests the replacement of $j_{2,\ep}$ with a multiple of $\rho_{\ep}'$. We stress the fact that at no point in this work
do we assume to be working with the steady state of the particle system~\eqref{eq:24}. Nevertheless, at least under
Assumption~(NG), the dynamics of~\eqref{eq:24} tends to the steady state for $t\rightarrow \infty$, see~\cite[Theorem
0.1.]{Herau2004a}. In the case $\sigma^2\ll 2\gamma$ (i.e., for the overdamped Langevin dynamics), this entails that
\begin{align*}
  \mbox{Var}[p^2_i(t)]\leq C\sigma^4/(2\gamma)^2\ll \sigma^2/(2\gamma)\approx\mean{p^2_i(t)}\approx 0.
\end{align*} 
It is then natural to replace $p_i^2$ with $\frac{\sigma^2}{2\gamma}$ on the probability space $\Omega$, hence to replace $j_{2,\ep}$
with $\frac{\sigma^2}{2\gamma}\frac{\partial\rho_{\ep}}{\partial x}$.

\emph{Approximation 4}. We replace the term $N^{-1}\sum_{i=1}^{N}{V'(q_i(t))w_{\ep}(x-q_i(t))}$ with the term
$V'(x)\rho_{\ep}(x,t)$. This is justified by the following result, which the reader may skip on a first reading.

\begin{lemma}
  Let the scaling of $N$ and $\ep$ be such that $\ep\rightarrow 0$ as $N\rightarrow \infty$. For each $x\in D$ and $t\in[0,T]$, we
  have $\lim_{N\rightarrow \infty}\mean{\left|V'(x)\rho_{\ep}(x,t)-N^{-1} \sum_{i=1}^{N}{V'(q_i(t))w_{\ep}(x-q_i(t))}\right|}=0$.
\end{lemma}

\begin{proof}
  The claim is trivial under Assumption~(G). Let us then consider Assumption~(NG). The particles being identically distributed, we
  only have to show that $\mean{\left|V'(q_1(t))-V'(x)\right|w_{\ep}(x-q_1(t))}\rightarrow0$ as $\ep\rightarrow 0$. We
  use~\eqref{eq:137} to deduce that $f_q\in L^{\infty}(\mathbb{R})$, where $f_{q}$ is the probability density function of
  $q_1(t)$. We set $\alpha:=2n-2\geq 0$, where $n$ is given in Assumption~(NG). In addition, we set
  $D_{\tau}(\ep):=[-\ep^{-\tau},+\ep^{-\tau}]$ for some $\tau\in(0,\alpha^{-1})$ whenever $\alpha>0$, or for some $\tau>0$ when
  $\alpha=0$. We compute
  \begin{align}
    \label{eq:801}
    & \mean{\left|V'(q_1(t))-V'(x)\right|w_{\ep}(x-q_1(t))} =\int_{\mathbb{R}}{\left|V'(y)-V'(x)\right|
      w_{\ep}(x-y)f_q(y)\m y}\nonumber\\
    & \quad \leq C\int_{D_{\tau}(\ep)}{\left|V'(y)-V'(x)\right|w_{\ep}(x-y)\m y}+C\int_{D^c_{\tau}(\ep)}{\left|V'(y)-V'(x)\right|
      w_{\ep}(x-y)\m y}.
\end{align}
We notice that $w_{\ep}(x-y)\leq C(x,\tau)w_{\tilde{\ep}}(x-y)$ for all $y\in D^c_{\tau}(\ep)$, the complement of $D_{\tau}(\ep)$,
where $0<\ep\leq \tilde{\ep}:=(|x|+1)^{-1/\tau}$. Moreover, Assumption~(NG) implies that
$\left|V'(y)\right|\leq C(\alpha)(1+|y|^{\alpha+1})$ and $\left|V''(y)\right|\leq C(\alpha)(1+|y|^{\alpha})$, for all
$y\in\mathbb{R}$. With respect to~\eqref{eq:801}, we bound the integral on $D_{\tau}(\ep)$ by using the mean-value theorem and the
control on $V''$, and we bound the integral on $D^c_{\tau}(\ep)$ by relying on the kernel $w_{\tilde{\ep}}$ and the control on
$V'$. We obtain
\begin{align}
  \label{eq:800}
  & \mean{\left|V'(q_1(t))-V'(x)\right|w_{\ep}(x-q_1(t))} \nonumber\\
  &  \quad \leq C\ep^{-\alpha\tau}\int_{D_{\tau}(\ep)}{\left|y-x\right|w_{\ep}(x-y)\m 
    y}+C(x,\tau,\alpha)\int_{D^c_{\tau}(\ep)}{(1+|y|^{\alpha+1})w_{\ep}(x-y)\m y}\nonumber\\
  & \quad \leq C\ep^{-\alpha\tau+1}+C(x,\tau,\alpha)\int_{D^c_{\tau}(\ep)}{(1+|y|^{\alpha+1})w_{\tilde{\ep}}(x-y)\m y},
\end{align}
where we have used Lemma~\ref{la:5} in the last inequality. The right-hand-side of~\eqref{eq:800} tends to $0$ as
$\ep\rightarrow 0$ due the choice of $\tau$ and the dominated convergence theorem. This concludes the proof.
\end{proof}

The approximations discussed above yield the system of equations
\begin{subequations}
\label{eq:300}
\begin{empheq}[left={}\empheqlbrace]{align}
  \,\,\displaystyle\frac{\partial \rho_{\ep}}{\partial t}(x,t) & = -\frac{\partial j_{\ep}}{\partial x}(x,t),\label{eq:300a}\\
  \,\,\displaystyle\frac{\partial j_{\ep}}{\partial t}(x,t) & = 
  -\gamma j_{\ep}(x,t)-\left(\frac{\sigma^2}{2\gamma}\right)\frac{\partial \rho_{\ep}}{\partial x}(x,t)-V'(x)\rho_{\ep}(x,t)
  +\frac{\sigma}{\sqrt{N}}\sqrt{\rho_{\ep}(x,t)}\,\tilde{\xi}_{\ep}, \label{eq:300b}\\
  \,\, \displaystyle\rho_{\ep}(x,0) & =\rho_0(x),\quad j_{\ep}(x,0)=j_0(x)\nonumber,
  \end{empheq}
\end{subequations}
where $ x\in D,\,t\in[0,T]$, and $\tilde{\xi}_{\ep}=Q^{1/2}_{\sqrt{2}\ep}\xi$ is an $L^2(D)$-valued $Q$-Wiener process, and
$\rho_0$, $j_0$ are suitable initial conditions. System~\eqref{eq:300} is one step away from being our regularised Dean--Kawasaki
model. This final step is illustrated in the final section, as the need for it shows while trying to establish existence of
solutions to~\eqref{eq:300}.

%% Comments FC3 here. 

\section{Mild solutions to the regularised Dean--Kawasaki model in a periodic setting}
\label{s:5}

We investigate existence and uniqueness of mild solutions to system~\eqref{eq:420}, which we refer to as a \emph{regularised
  Dean--Kawasaki model}. System~\eqref{eq:420} is the $2\pi$-periodic equivalent of~\eqref{eq:300}.  The reason for considering
the spatially periodic case will be discussed below. Note that the quantities $\rho_{\ep}$, $j_{\ep}$ in~\eqref{eq:300}
and~\eqref{eq:420} are no longer associated with the definitions given in~\eqref{eq:102} but are the unknown solutions to the two
systems.
%equations in~\eqref{eq:420}.

We rewrite~\eqref{eq:300} as a stochastic partial differential equation of the type
\begin{equation}
  \label{eq:301}
  \left\{
    \begin{array}{l}
      \m X_{\ep}(t)=(AX_{\ep}(t)+\alpha X_{\ep}(t))\m t+B_{N}(X_{\ep}(t))\m W_{\ep}, \\
      X_{\ep}(0)=X_0,
    \end{array}
  \right.
\end{equation}
where $X_{\ep}(t):=(\rho_{\ep}(\cdot,t),j_{\ep}(\cdot,t))$, $X_0=(\rho_0,j_0)$, and $W_{\ep}:=(W_{\ep,1},W_{\ep,2})$ is a suitable
stochastic noise, with
\begin{equation*}
  AX_{\ep}(t):=\left(-\frac{\partial j_{\ep}}{\partial x}(\cdot,t),\,-\gamma j_{\ep}(\cdot,t)-\left(\frac{\sigma^2}{2\gamma}\right)\frac{\partial \rho_{\ep}}{\partial x}(\cdot,t)\right),\qquad 
  \alpha X_{\ep}(t):=\left(0,-V'(\cdot)\rho_{\ep}(\cdot,t)\right),
\end{equation*}
and $B_{N}$ is some suitable integrand specified below.

Subsection~\ref{ss:10} is devoted to the analysis of the operator $A$ by means of the $C_0$-semigroup theory. We define and
analyse the periodic equivalents $W_{\per,\ep}$ and $\alpha_{\per}$ of $W_{\ep}$ and $\alpha$ in Subsection~\ref{ss:40}.  We
describe the relevant properties of the stochastic integrand $B_N$ in Subsection~\ref{ss:81}, and prove existence and uniqueness
of mild solutions to a suitable locally Lipschitz approximation of~\eqref{eq:420} in Subsection~\ref{ss:50}. We then prove
suitable small-noise regime estimates in Subsections~\ref{ss:51} and~\ref{ss:52}. We finally prove the main existence and
uniqueness result, Theorem~\ref{thm:100}, in Subsection~\ref{ss:100}.

In this section, we set $D:=[0,2\pi]$. We fix $k_{B}T_e=\sigma^2/(2\gamma):=1$ for notational simplicity, even though all our
conclusions hold for arbitrary positive ratio $\sigma^2/(2\gamma)$.

%% Comment FC4 here

\subsection{Semigroup analysis for the operator $A$ in $\mathcal{W}=H^1_{\per}(D)\times H^1_{\per}(D)$}
\label{ss:10}

We characterise the semigroup associated with the operator $A$, which can be done in a straightforward manner. For any
$2\pi$-periodic function $f\colon \mathbb{R}\rightarrow \mathbb{R}$ such that $f|_D\in L^2(D)$, we write its Fourier coefficients
as $\hat{f}_m:=(2\pi)^{-1}\int_{D}{e^{-imx}f(x)\m x}$, for any $m\in\mathbb{Z}$. We consider the Sobolev spaces of $2\pi$-periodic
functions
\begin{align*}
  H^n_{\per}(D) & := \left\{f=\sum_{m\in\mathbb{Z}}{\hat{f}_me^{-imx}}\colon 
                  \sum_{m\in\mathbb{Z}}{\left(1+m^2\right)^n \hat{f}_m^2}<\infty\right\},\quad n\in\mathbb{N},
\end{align*}
endowed with standard norms and inner products. We also consider the spaces 
\begin{align*}
  C^n_{\per}(D) & := \left\{f : f\in C^n(\mathbb{R}), f\mbox{ is periodic with period $2\pi$}\right\},
                  \quad n\in\mathbb{N}\cup\{0\},
\end{align*}
where $C^{0}_{\per}(D)$ is endowed with its standard norm. We also recall the following Sobolev embedding theorem, valid only in
one space dimension.
\begin{prop}
  \label{p:10}
  The embedding $H_{\per}^1(D)\subset C^0_{\per}(D)$ is continuous.
\end{prop}

As an immediate consequence of Proposition~\ref{p:10}, we deduce that, for $f\in H^n_{\per}(D)$, $n\geq 1$,
\begin{align*}
  \frac{\m^k}{\m x^k}f(0)=\frac{\m^k}{\m x^k}f(2\pi),\qquad\mbox{ for all } k=0,1,\cdots,n-1.
\end{align*}
We also recall the spaces
\begin{align*}
  \mathcal{W}:=H^1_{\per}(D) \times H^1_{\per}(D),
  &\quad \langle (u_1,v_1),(u_2,v_2)\rangle_{\mathcal{W}}                                              
    :=\langle u_1,u_2\rangle_{H_{\per}^1(D)}+\langle v_1,v_2\rangle_{H_{\per}^1(D)}, \\
  \mathcal{W}\supset\mathcal{D}(A):=H^2_{\per}(D)\times H^2_{\per}(D),
  &\quad \langle (u_1,v_1),(u_2,v_2)\rangle_{\mathcal{D}(A)}:=\langle u_1,u_2\rangle_{H_{\per}^2(D)}+\langle v_1,v_2\rangle_{H_{\per}^2(D)}.
\end{align*}

\begin{lemma}\label{lem:30}
  The operator $A\colon\mathcal{D}(A)\subset\mathcal{W}\rightarrow \mathcal{W}$ defines a $C_0$-semigroup of contractions
  $\{S(t)\}_{t\geq 0}$.
\end{lemma}

\begin{proof}
We verify the assumptions of the Hille--Yosida Theorem, as stated in~\cite[Theorem 3.1]{Pazy1983a}. This is a straightforward
step, and might be skipped on a first reading.

%\emph{$A$ is a closed operator} Consider a sequence $\mathcal{D}(A)\ni(\rho_n,j_n)_n\rightarrow (\rho,j)$ in $\mathcal{W}$, and
%such that $A(\rho_n,j_n)\rightarrow (y_1,y_2)$ in $\mathcal{W}$. This assumption immediate implies that $\{\rho_n\}_n$,
%$\{j_n\}_n$ converge in $H^2_{\per}(D)$, hence $\rho_n\rightarrow c_1$, $j_n\rightarrow c_2$ in $H^2_{\per}(D)$. But then by
%uniqueness of limits, we have $(\rho,j)=(c_1,c_2)\in\mathcal{D}(A)$. It is a routine task to check that also $(y_1,y_2)=A(\rho,j)$
%by using the linearity of $A$ and the convergences of $\{\rho_n\}_n,\{j_n\}_n$.

\emph{$A$ is a closed operator, and $\mathcal{D}(A)$ is dense in $\mathcal{W}$.} This is easily checked. 
%\emph{$\mathcal{D}(A)$ is dense in $\mathcal{W}$.}  We know that $C^{\infty}_0(D)$ in dense in $H^1_0(D)$. By periodic
%extension, we deduce that $C^{\infty}_{0,per}(D)$ is dense $H^1_{0,per}(D):=\{f\in H^1_{\per}(D):\,\,f(0)=f(2\pi)=0\}$. For a
%function $f\in H^1_{\per}(D)$, we can find a sequence $C^{\infty}_{0,per}(D)\ni\varphi_n\rightarrow f-f(0)\in H^1_{0,per}(D)$ in
%$H^1_{\per}(D)$. It is then straightforward to notice that $H^2_{\per}(D)\ni\varphi_n+f(0)\rightarrow f$ in $H^1_{\per}(D)$. This
%proves the density of $H^2_{\per}(D)$ in $H^1_{\per}(D)$.

\emph{The resolvent set of $A$ contains the positive half line.} For every $\lambda>0$, we consider
$A^{-1}_{\lambda}:=(A-\lambda I)^{-1}$, whenever this is well-defined. We first prove that it exists, by showing injectivity of
$A_{\lambda}:=A-\lambda I$. Let then assume that $A_{\lambda}(\rho,j)=(0,0)$. We multiply the first component of
$A_{\lambda}(\rho,j)$ by $\rho$ and the second component of $A_{\lambda}(\rho,j)$ by $j$, and we obtain
\begin{equation*}
  (-j'-\lambda\rho)\rho+(-(\lambda+\gamma)j-\rho')j=-\lambda\rho^2-(\lambda+\gamma)j^2-(\rho j)'=0.
\end{equation*}
Integrating over $D$ and using the periodic boundary conditions for $\rho$ and $j$, we obtain
\begin{equation*}
  \lambda\|\rho\|^2_{L^2(D)}+(\lambda+\gamma)\|j\|^2_{L^2(D)}=0.
\end{equation*}
Since $\lambda,\gamma>0$, we deduce that $(\rho,j)=(0,0)$. We now show that $A^{-1}_{\lambda}$ is a bounded operator. Consider
$A^{-1}_{\lambda}(a,b)=(\rho,j)$. This implies
\begin{align}
  \lambda \rho & =-a-j',\label{eq:400}\\
  (\lambda+\gamma)j & = -b-\rho',\label{eq:401}\\
  \lambda \rho' & =-a'-j'',\label{eq:402}\\
  (\lambda+\gamma)j' & = -b'-\rho'',\label{eq:403} 
\end{align}
where~\eqref{eq:402} (respectively~\eqref{eq:403}) is obtained by differentiating~\eqref{eq:400} (respectively~\eqref{eq:401}). We
multiply~\eqref{eq:400} by $\rho$,~\eqref{eq:401} by $j$,~\eqref{eq:402} by $\rho'$,~\eqref{eq:403} by $j'$, and sum the four
equalities. An integration of the resulting expression over $D$ yields
\begin{align}
  \label{eq:310}
  \lambda\|(\rho,j)\|^2_{\mathcal{W}}\leq \lambda\|\rho\|^2_{H^1_{\per}(D)}+(\lambda+\gamma)\|j\|^2_{H^1_{\per}(D)}
  =\int_{D}{-a\rho} \m x +\int_{D}{-bj} \m x +\int_{D}{-a'\rho'} \m x +\int_{D}{-b'j'} \m x,
\end{align}
where we have also used the periodic boundary conditions for $\rho$, $j$, $\rho'$, $j'$. We now use the Cauchy-Schwartz inequality
and the Young inequality $|xy|\leq \theta^2 x^2+(1/4\theta^2)y^2$ with $\theta^2:=\lambda/2$ to bound the four integrals in the
right-hand-side of~\eqref{eq:310}. This directly gives
$(\lambda/2)\|(\rho,j)\|^2_{\mathcal{W}}\leq (1/2\lambda)\|(a,b)\|^2_{\mathcal{W}}$, which implies
\begin{align}
  \label{eq:312}
  \|A^{-1}_{\lambda}\|_{\mathcal{L}(\mathcal{W},\mathcal{W})}\leq \frac{1}{\lambda},
\end{align}
so $A^{-1}_{\lambda}$ is bounded. We now show that $\mbox{Dom}(A^{-1}_{\lambda})$ is dense in $\mathcal{W}$. Let us fix
$(a,b)\in H^2_{\per}(D)\times H^1_{\per}(D)$. %C^{\infty}_0(C)\times C^{\infty}_0(C)$.
We consider the system of equations $A_{\lambda}(\rho,j)=(a,b)$, namely
\begin{align*}
  -j'-\lambda\rho = a, \qquad -(\lambda+\gamma)j-\rho' = b.
\end{align*}
We rewrite the first equation as $\rho=(-j'-a)/\lambda$ and substitute into the second equation, obtaining
\begin{align}
  \label{eq:315}
  -\frac{j''}{\lambda}+(\lambda+\gamma)j=\frac{a'}{\lambda}-b\in H^1_{\per}(D).
\end{align}
The elliptic theory provides existence of a unique solution $j\in
H^3_{\per}(D)$ for~\eqref{eq:315}. From $\rho:=(-j'-a)/\lambda$, we immediately deduce that $\rho\in
H^2_{\per}(D)$. We have shown that, for every $(a,b)$ in a dense subset of $\mathcal{W}$ (namely $H^2_{\per}(D)\times
H^1_{\per}(D)$), the operator $A^{-1}_{\lambda}$ is well-defined.

\emph{Inequality~\cite[(3.1)]{Pazy1983a} is satisfied:} This is precisely~\eqref{eq:312}.
\end{proof}

\subsection{Introducing periodic noise and periodic potential drift}
\label{ss:40}

We now define the noise $W_{\ep}$ for~\eqref{eq:301} in accordance with the noise in~\eqref{eq:300b}. We set
\begin{align*}
  \dot{W}_{\ep}:=\left(0,\,\tilde{\xi}_{\ep}\right)=\left(0,\,Q^{1/2}_{\sqrt{2}\ep}\xi\right).
\end{align*} 
The second component of $\dot{W}_{\ep}$ agrees with the noise in~\eqref{eq:300b}. Since~\eqref{eq:300a} is a deterministic
equation, we set the first component of $\dot{W}_{\ep}$ to zero.  We represent $W_{\ep}$ as~\cite[Proposition 2.1.10]{Prevot2007a}
\begin{align}
  \label{eq:415}
  W_{\ep}=\sum_{j=1}^{\infty}{\sqrt{\lambda_j}(0,e_j)\beta_j(t)},
\end{align} 
where $\{e_j\}_{j}$ and $\{\lambda_j\}_j$ refer to the families of eigenfunctions and eigenvalues of the Hilbert-Schmidt integral
operator $Q_{\sqrt{2}\ep}$ on $L^2(D)$. Unfortunately, the eigenfunctions $\{e_j\}_j$ are not $2\pi$-periodic. To verify this, one
can rely on Mercer's Theorem and evaluate the kernel expansion $w_{\sqrt{2}\ep}(x-y)=\sum_{j=1}^{\infty}{\lambda_je_j(x)e_j(y)}$
for the pairs $(x,y)=(0,0)$ and $(x,y)=(0,2\pi)$. We deduce that the $Q$-Wiener process $W_{\ep}$ does \emph{not} necessarily take
values in the space associated with the semigroup analysis of $A$, i.e., in $\mathcal{W}$. In order to resolve this issue, we
identify the end-points of the interval $[0,2\pi]$, thus thinking of $[0,2\pi]$ as a flat torus. We provide, for each $\ep>0$, a
$2\pi$-periodic kernel $p_{\sqrt{2}\ep}$ approximating $w_{\sqrt{2}\ep}$. A suitable choice lies in the von Mises distribution, a
$2\pi$-periodic distribution parametrised by $\mu\in\mathbb{R}$, $\kappa>0$, and given by the probability density function
%% Comment FC14 here
\begin{align*}
  f(x,\mu,\kappa)=\frac{e^{\kappa\cos(x-\mu)}}{2\pi I_0(\kappa)},\qquad I_0(\kappa):=\frac{1}{2\pi}\int_{D}{e^{\kappa\cos(x)}\m x}.
\end{align*}
The von Mises distribution~\cite{Forbes2011a} approximates the Gaussian kernel in the following way
\begin{align*}
  \lim_{\kappa\rightarrow +\infty}{\left\|f(x,\mu,\kappa)-\frac{1}{\sqrt{2\pi\sigma^2}}\exp\left\{-\frac{(x-\mu)^2}{2\sigma^2}\right\}
  \right\|_{C^0(\mu-\pi,\,\mu+\pi)}}=0,\quad\mbox{where }\sigma^2:=\kappa^{-1}.
\end{align*}
For this reason, we replace the kernel $w_{\sqrt{2}\ep}$, $\ep>0$, with the $2\pi$-periodic kernel
\begin{align*}
  p_{\sqrt{2}\ep}(x):=f\left(x,0,(2\ep^2)^{-1}\right)=\frac{e^{\frac{\cos(x)}{2\ep^2}}}
  {2\pi I_0(1/(2\ep^2))}=Z^{-1}_{\sqrt{2}\ep}e^{-\frac{\sin^2(x/2)}{\ep^2}},\qquad Z^{-1}_{\sqrt{2}\ep}
  :=\frac{e^{\frac{1}{2\ep^2}}}{2\pi I_0(1/(2\ep^2))}.
\end{align*} 
In the limit $\ep\rightarrow 0 $, the kernel $p_{\sqrt{2}\ep}$ recovers the Gaussian kernel $w_{\sqrt{2}\ep}$ on the flat torus.
We study the eigenfunctions and eigenvalues of the operator
\begin{align}
  \label{eq:1500}
  P_{\sqrt{2}\ep}\colon L^2(D)\rightarrow L^2(D),\qquad P_{\sqrt{2}\ep}f(x)=\int_{D}{p_{\sqrt{2}\ep}(x-y)f(y)\m y},\quad f\in L^2(D).
\end{align}
We obtain the eigenfunctions $\{e_{j,\ep}\}_{j\in\mathbb{Z}}$ and eigenvalues $\{\lambda_{j,\ep}\}_{j\in\mathbb{Z}}$ of
$P_{\sqrt{2}\ep}$ from~\cite[Section 4.2]{Duran-Olivencia2017a}, namely
\begin{align*}
  e_{j,\ep}(x) = e_{j}(x) =
  \begin{cases}
    \sqrt{\frac{1}{\pi}}\cos(jx), & \mbox{if } j>0, \vspace{0.3 pc}\\
    \sqrt{\frac{1}{\pi}}\sin(jx), & \mbox{if } j<0, \vspace{0.3 pc}\\
    \sqrt{\frac{1}{2\pi}}, & \mbox{if } j=0,
  \end{cases}                   
\end{align*}
and
\begin{align}
 \label{eq:408}
  \lambda_{j,\ep} 
  & =
    \begin{cases}
      \displaystyle Z^{-1}_{\sqrt{2}\ep}\int_{D}{e^{-\frac{\sin^2(x/2)}{\ep^2}}\cos(jx)
      \m x}=C_2Z^{-1}_{\sqrt{2}\ep}e^{-\frac{1}{2\ep^2}}I_j\left(\{2\ep^2\}^{-1}\right), 
      & \mbox{if } j\neq 0, \vspace{0.3 pc}\\
      1, & \mbox{if } j=0, 
    \end{cases}
\end{align}
where $I_j(z):=(2\pi)^{-1}\int_{D}{e^{z\cos(x)}\cos(jx)\m x}$ is the \emph{modified Bessel function} of first kind and order $j$,
see~\cite[Eq.~(9.6.19)]{Abramowitz1964a}. It is immediate to notice that $\{e_j\}_j$ is an orthogonal basis of $H^1_{\per}(D)$,
and that the family $\{f_j\}_{j\in\mathbb{Z}}$
\begin{align}
  \label{eq:411}
  f_j(x) =
  \left\{
  \begin{array}{ll}
    e_j(x)/\sqrt{1+j^2}, & \mbox{if } j\neq 0, \vspace{0.3 pc}\\
    \sqrt{\frac{1}{2\pi}}, & \mbox{if } j=0,
  \end{array}
  \right.
\end{align}
is an orthonormal basis of $H^1_{\per}(D)$. This is crucial, as it will allow us to construct a $\mathcal{W}$-valued noise below.

We now turn to estimating relevant properties of $\{\lambda_{j,\ep}\}_j$.

\begin{lemma}
  \label{lem:20}
  Fix $n\in\mathbb{N}$. There exists $\ep_0>0$ such that for $0<\ep<\ep_0$ we have
  $\sum_{j\in\mathbb{Z}}{\lambda_{j,\ep}|j|^n}\leq C(n)\ep^{-(2n+3)}$.
\end{lemma}

\begin{proof}
  We start with bounding $Z_{\sqrt{2}\ep}$ from below as
\begin{align}
  \label{eq:406}
  Z_{\sqrt{2}\ep}=\int_{D}{e^{-\frac{\sin^2(x/2)}{\ep^2}}\m x}\geq \int_{D}{e^{-\frac{x^2}{4\ep^2}}\m x}
  \geq \int_{0}^{\sqrt{4\ep^2\ln 2}}{(1/2)\m x}=C\ep.
\end{align}

We now turn to $I_j$. We first of all notice that $I_1(z)\leq I_0(z)$ for any $z\geq 0$. In addition, we have
\begin{align*}
  I_0(z)=(2\pi)^{-1}\int_{D}{e^{z\cos(x)}\m x}\leq \int_{D}{e^{z}\m x}=Ce^z.
\end{align*} 
We use a recursive property of the modified Bessel functions of first kind~\cite[Eq.~(9.6.26)]{Abramowitz1964a}, namely
\begin{align}
  \label{eq:405}
  I_{j+1}(z)=I_{j-1}(z)-\frac{2j}{z}I_{j}(z),\qquad \mbox{ for all } z>0,\,\mbox{ for all } j\in \mathbb{N}.
\end{align}
Since the modified Bessel functions of first kind are always non-negative for non-negative
arguments~\cite[Eq.~(9.6.10)]{Abramowitz1964a}, we deduce from~\eqref{eq:405} that $I_j(z)\leq (z/2j)I_{j-1}(z)$. For $j>z$, we
have $I_j(z)\leq (1/2)I_{j-1}(z)$, which implies an exponential decay of $I_j(z)$ for $j>z$. Since $I_1(z)\leq I_0(z)$,
equality~\eqref{eq:405} also implies that $I_{j}(z)\leq I_0(z)$,$\mbox{ for all } j\in\mathbb{N}$. To sum up, we get the bounds
\begin{align}
  \label{eq:407}
  I_j(z) \leq
  \begin{cases}
    Ce^z & \mbox{if } j\leq z, \\
    Ce^z\left(\frac{1}{2}\right)^{j-z}, & \mbox{if } j>z.
  \end{cases}
\end{align} 
We take $z=(2\ep^2)^{-1}$, and we set $m(\ep):=\lceil (2\ep^2)^{-1}\rceil$. We feed~\eqref{eq:406} and~\eqref{eq:407}
into~\eqref{eq:408}, thus obtaining
\begin{align}
  \label{eq:409}
  \lambda_{j,\ep} \leq
     \begin{cases}
     C\ep^{-1}, & \mbox{if } j\leq m(\ep),\\
     C\ep^{-1}\left(\frac{1}{2}\right)^{j-m(\ep)}, & \mbox{if } j>m(\ep),
     \end{cases}
\end{align} 
where $C$ is a constant independent of $\ep$.  As a result of~\eqref{eq:409} we get, for $\ep$ sufficiently small,
\begin{align*}
  \frac{1}{2}\sum_{j\in\mathbb{Z}}{\lambda_{j,\ep}|j|^n}\leq\sum_{j=0}^{\infty}{\lambda_{j,\ep}j^n} 
  & =\sum_{j=0}^{m(\ep)}{\lambda_{j,\ep}j^n}+\sum_{j>m(\ep)}{\lambda_{j,\ep}j^n} \leq C(n)\ep^{-1}m(\ep)^{(n+1)}\\
  & \quad + C(n)\ep^{-1}\sum_{j>m(\ep)}{(1/2)^{j-m(\ep)}\left\{(j-m(\ep))^n+m(\ep)^{n}\right\}}\leq C(n)\ep^{-(2n+3)},
\end{align*}
and the proof is complete.
\end{proof}

These considerations show that the noise $\dot{W}_{\ep}$ given in~\eqref{eq:415} can be replaced, in a periodic setting, by the
noise $\dot{W}_{\per,\ep}=(0,\tilde{\xi}_{\per,\ep}):=(0,P^{1/2}_{\sqrt{2}\ep}\xi)$, where $P$ is defined in~\eqref{eq:1500}. This
noise is a $\mathcal{W}$-valued $Q$-Wiener process given by
\begin{align}
  W_{\per,\ep}=\sum_{j\in\mathbb{Z}}{\sqrt{\alpha_{j,\ep}}(0,f_j)}\beta_j,\qquad \alpha_{j,\ep}:=(1+j^2)\lambda_{j,\ep},
\end{align}
where $\{\beta_j\}_j$ is a family of independent one-dimensional standard Brownian motions. For consistency, we assume $V$ is
periodic, i.e., $V = V_{\per}\in C^2_{\per}(D)$. It is also immediate to notice that the operator
$\alpha_{\per}X_{\ep}(t):=\left(0,-V_{\per}'(\cdot)\rho_{\ep}(\cdot,t)\right)$ belongs to $L(\mathcal{W})$, i.e., to the set of
bounded linear operators on $\mathcal{W}$.

In the remaining of the paper, we investigate existence and uniqueness of solutions to the \emph{regularised Dean--Kawasaki model}
\begin{equation}
  \label{eq:501}
  \left\{
    \begin{array}{l}
      \m X_{\ep}(t)=(AX_{\ep}(t)+\alpha_{\per}X_{\ep}(t))\m t+B_{N}(X_{\ep}(t))\m W_{\per,\ep}, \\
      X_{\ep}(0)=X_0.
    \end{array}
  \right.
\end{equation}
System~\eqref{eq:501} is the equivalent of~\eqref{eq:300} in a periodic setting and is a functional rewriting of~\eqref{eq:420}.

\subsection{Locally Lipschitz stochastic integrand with respect to $\mathcal{W}$-topology}
\label{ss:81}

%% Comments FC5 here

In this subsection, we define and analyse the properties of the noise integrand $B_{N}$. It is natural to define
$B_{N}\colon\mathcal{W}\rightarrow \{f\colon \mathcal{W}\rightarrow L^2(D)\times L^2(D)\}$ as
\begin{align*}
  B_{N}((\rho,j))(a,b):=\frac{\sigma}{\sqrt{N}}\left(0,\, \sqrt{|\rho|}\cdot b\right).
\end{align*}

\begin{rem} 
  We see that
  \begin{align}
    \label{eq:317}
    & \int_{0}^{t}{B_{N}((X(s),Y(s)))\m W_{\per,\ep}(s)} = 
      \int_{0}^{t}{\sum_{j\in\mathbb{Z}}{\sqrt{\alpha_{j,\ep}}B_N((X(s),Y(s)))(0,f_j)\m\beta_j(s)}}\nonumber\\
    & \quad = \frac{\sigma}{\sqrt{N}}\int_{0}^{t}{\sum_{j\in\mathbb{Z}}{\sqrt{\alpha_{j,\ep}}\left(0,\sqrt{|X(s)|}f_j\right)\m\beta_j(s)}} = 
      \left(0,\int_{0}^{t}{\frac{\sigma}{\sqrt{N}}\sqrt{|X(s)|}\m P^{1/2}_{\sqrt{2}\ep}\xi(s)}\right).
  \end{align}
  The last expression of~\eqref{eq:317} is precisely the stochastic noise of~\eqref{eq:501}.
\end{rem}

The integrand $B_N$ poses several difficulties. Firstly, $B_N$ is not a mapping from $\mathcal{W}$ to $L^0_2(\mathcal{W})$, where
$L^0_2(\mathcal{W})$ denotes the set of Hilbert-Schmidt operators from $P^{1/2}_{\sqrt{2}\ep}\mathcal{W}\subset\mathcal{W}$ into
$\mathcal{W}$, see~\cite[Section 2.3]{Prevot2007a}.  Secondly, $B_N$ is not Lipschitz or locally Lipschitz with respect to
$(\rho,j)$. Both problems are due to the singularity of the square-root function. We address both problems by regularising this
singularity. For some $\delta>0$, we define
\begin{align*}
  B_{N,\delta}((\rho,j))(a,b):=\frac{\sigma}{\sqrt{N}}\left(0,\, h_\delta(\rho)\cdot b\right),
\end{align*}
where $h_{\delta}\colon\mathbb{R}\rightarrow\mathbb{R}$ is a $C^2$-Lipschitz modification of $\sqrt{|z|}$ in
$[-\delta,+\delta]$. In this way, $h_{\delta}$ is Lipschitz, and has bounded first and second derivatives. We characterise some
important features of $B_{N,\delta}$.

\begin{lemma}
  The following properties hold.
  \begin{enumerate}
  \item \label{it:reg-1} $B_{N,\delta}$ is a map from $\mathcal{W}$ to $L(\mathcal{W})$.
  \item \label{it:reg-2} $B_{N,\delta}$ is \emph{locally} Lipschitz with respect to the $L^0_2(\mathcal{W})$-norm.
  \item \label{it:reg-3} $B_{N,\delta}$ has sublinear growth at infinity with the respect to the $L^0_2(\mathcal{W})$-norm.
  \end{enumerate}
\end{lemma}

\begin{proof}
\emph{Statement~\ref{it:reg-1}}. Take $(u,v), (a,b)\in\mathcal{W}$. We use Proposition~\ref{p:10} and write
\begin{align*}
  \|B_{N,\delta}((u,v))(a,b)\|^2_{\mathcal{W}} 
  & = \frac{\sigma^2}{N}\|h_{\delta}(u)b\|^2_{H^1_{\per}(D)} 
    \leq \frac{\sigma^2}{N}\left\{\|h_{\delta}(u)b\|^2_{L^2(D)}+C(\delta,u)\|b'\|^2_{L^2(D)}\right.\\
  & \quad \left.+C(\delta)\|b\|^2_{C^0_{\per}(D)}\|u'\|^2_{L^2(D)}\right\}
    \leq \frac{\sigma^2}{N}C(\delta,u)\|b\|^2_{H^1_{\per}(D)}\leq\frac{\sigma^2}{N}C(\delta,u)\|(a,b)\|^2_{\mathcal{W}}.
\end{align*}
This settles the first claim.

\emph{Statement~\ref{it:reg-2}}. Take $(u_1,v_1),(u_2,v_2)\in\mathcal{W}$, such that
$\|(u_1,v_1)\|_{\mathcal{W}}\leq k, \|(u_2,v_2)\|_{\mathcal{W}}\leq k$. We have
\begin{align*}
  \left\|B_{N,\delta}((u_1,v_1))-B_{N,\delta}((u_2,v_2))\right\|^2_{L^0_2(\mathcal{W})} 
  & = \sum_{j\in\mathbb{Z}}{\left\|\sqrt{\alpha_{j,\ep}}\left\{B_{N,\delta}((u_1,v_1))-B_{N,\delta}((u_2,v_2))\right\}(0,f_j)
    \right\|^2_{\mathcal{W}}}\\
  & = \frac{\sigma^2}{N}\sum_{j\in\mathbb{Z}}{\alpha_{j,\ep}\left\|\left(0,\left\{h_{\delta}(u_1)-
    h_{\delta}(u_2)\right\}f_j\right)\right\|^2_{\mathcal{W}}}.
\end{align*}
The right-hand-side in the expression above is well-defined by~\ref{it:reg-1}. From~\eqref{eq:411}, we deduce that
$\|f_j\|_{L^{\infty}}\leq \pi^{-1/2}$, $\|f'_j\|_{L^{\infty}}\leq \pi^{-1/2}$,$\mbox{ for all } j\in\mathbb{Z}$. We use this fact,
as well as the boundedness of $h'_{\delta}$, to compute
\begin{align*}
  & \frac{\sigma^2}{N}\sum_{j\in\mathbb{Z}}{\alpha_{j,\ep}\left\|\left(0,\left\{h_{\delta}(u_1)
    -h_{\delta}(u_2)\right\}f_j\right)\right\|^2_{\mathcal{W}}}\\
  & \quad \leq 
    \frac{\sigma^2}{N}\left[\sum_{j\in\mathbb{Z}}{\alpha_{j,\ep}\left\|\left\{h_{\delta}(u_1)-h_{\delta}(u_2)\right\}f_j\right\|^2_{L^2(D)}}
    +\!\!\sum_{j\in\mathbb{Z}}{\alpha_{j,\ep}\left\|\frac{\m}{\m x}{\left(\left\{h_{\delta}(u_1)
    -h_{\delta}(u_2)\right\}f_j\right)}\right\|^2_{L^2(D)}}\right]\\
  & \quad \leq C\frac{\sigma^2}{N}\left[\sum_{j\in\mathbb{Z}}{\alpha_{j,\ep}\left\|h_{\delta}(u_1)-h_{\delta}(u_2)\right\|^2_{L^2(D)}}
    +\sum_{j\in\mathbb{Z}}{\alpha_{j,\ep}\left\|\frac{\m}{\m x}\left\{h_{\delta}(u_1)-h_{\delta}(u_2)\right\}\right\|^2_{L^2(D)}}\right]\\
  & \quad \leq C(\delta)\frac{\sigma^2}{N}\left(\sum_{j\in\mathbb{Z}}{\alpha_{j,\ep}}\right)\left\{\|u_1-u_2\|^2_{L^2(D)} 
    + \left\|h'_{\delta}(u_1)(u'_{1}-u'_{2})\right\|^2_{L^2(D)}+\left\|u'_{2}(h'_{\delta}(u_1)-h'_{\delta}(u_2))\right\|^2_{L^2(D)}\right\}.
\end{align*}
We use Proposition~\ref{p:10}, the boundedness of $h'_{\delta}$, $h''_{\delta}$, and Lemma~\ref{lem:20} to deduce
\begin{align*}
  & \left\|B_{N,\delta}((u_1,v_1))-B_{N,\delta}((u_2,v_2))\right\|^2_{L^0_2(\mathcal{W})} 
    \leq C(\delta)\frac{\sigma^2}{N}\left(\sum_{j\in\mathbb{Z}}{\alpha_{j,\ep}}\right)
    \left\{\|u_1-u_2\|^2_{L^2(D)}+\|u'_{1}-u'_{2}\|^2_{L^2(D)}\right.\\
  & \quad + \left.\|u'_{2}\|^2_{L^2(D)}\|u_1-u_2\|^2_{C^0_{\per}(D)}\right\}
    \leq C(\delta,k)\frac{\sigma^2}{N}\ep^{-7}\|u_1-u_2\|^2_{H^1_{\per}(D)}
    \leq C(\delta,k)\frac{\sigma^2}{N}\ep^{-7}\|(u_1,v_1)-(u_2,v_2)\|^2_{\mathcal{W}},
\end{align*}
which is the desired local Lipschitz property for $B_{N,\delta}$.

\emph{Statement~\ref{it:reg-3}}. We proceed similarly to the proof of~\ref{it:reg-2}, and compute
\begin{align*}
  \left\|B_{N,\delta}((u,v))\right\|^2_{L^0_2(\mathcal{W})} 
  & = \sum_{j\in\mathbb{Z}}{\left\|\sqrt{\alpha_{j,\ep}}B_{N,\delta}((u,v))(0,f_j)\right\|^2_{\mathcal{W}}}
    =\frac{\sigma^2}{N}\sum_{j\in\mathbb{Z}}{\alpha_{j,\ep}\left\|\left(0,h_{\delta}(u)f_j\right)\right\|^2_{\mathcal{W}}}\\
  & \leq \frac{\sigma^2}{N}\left[\sum_{j\in\mathbb{Z}}{\alpha_{j,\ep}\left\|{h_{\delta}(u)f_j}\right\|^2_{L^2(D)}}
    +\sum_{j\in\mathbb{Z}}{\alpha_{j,\ep}\left\|\frac{\m}{\m x}{\left(h_{\delta}(u)f_j\right)}\right\|^2_{L^2(D)}}\right]\\
  &\leq C\frac{\sigma^2}{N}\left[\sum_{j\in\mathbb{Z}}{\alpha_{j,\ep}\left\|h_{\delta}(u)\right\|^2_{L^2(D)}}
    +\sum_{j\in\mathbb{Z}}{\alpha_{j,\ep}\|h'_{\delta}(u)u'\|^2_{L^2(D)}}\right]\\
  & \leq C(\delta)\frac{\sigma^2}{N}\left[\sum_{j\in\mathbb{Z}}{\alpha_{j,\ep}}\right]
    \left(1+\|(u,v)\|^2_{\mathcal{W}}\right)=C(\delta)\frac{\sigma^2}{N}\ep^{-7}\left(1+\|(u,v)\|^2_{\mathcal{W}}\right),
\end{align*}
where the last inequality follows from the sublinearity of $h_{\delta}$ at infinity and the boundedness of $h'_{\delta}$. We
deduce
\begin{align}
  \label{eq:322}
  \left\|B_{\delta}((u,v))\right\|_{L^0_2(\mathcal{W})}
  \leq \sqrt{C(\delta)\frac{\sigma^2}{N}\ep^{-7}}\left(1+\|(u,v)\|_{\mathcal{W}}\right)=C(\delta)
  \sigma \underbrace{N^{-1/2}\ep^{-7/2}}_{=:M(\ep,N)}\left(1+\|(u,v)\|_{\mathcal{W}}\right).
\end{align}
This completes the proof.
\end{proof}

\begin{rem}
  The quantity $M(\ep,N)$ introduced in~\eqref{eq:322} is the justification of the scaling $\theta>7$ in Theorem~\ref{thm:100}.
\end{rem}

\subsection{Existence of mild solutions in the $\mathcal{W}$-topology up to random time}\label{ss:50}

We consider the following $\delta$-smoothed version of the regularised Dean--Kawasaki system~\eqref{eq:501}
\begin{equation}
  \label{eq:318}
  \left\{
    \begin{array}{l}
      \m X_{\ep,\delta}(t)=(AX_{\ep,\delta}(t)+\alpha_{\per}X_{\ep,\delta}(t))\m t+B_{N,\delta}(X_{\ep,\delta}(t))\m W_{\per,\ep}, \\
      X_{\ep,\delta}(0)=X_0.
    \end{array}
  \right.
\end{equation}

We prove the following result.

\begin{prop}
  \label{p:11}
  Let $T>0$. Let $X_0\in\mathcal{W}$ be deterministic. Then~\eqref{eq:318} admits a unique mild solution $X_{\ep,\delta}$ on
  $[0,T]$ with respect to the $\mathcal{W}$-topology. Moreover, the solution $X_{\ep,\delta}$ is \emph{c\`adl\`ag} in the
  $\mathcal{W}$-topology.
\end{prop}

Let $\{S(t)\}_{t\geq 0}$ be the $C_0$-semigroup generated by $A$ discussed in Lemma \ref{lem:30}. We recall that a \emph{mild
  solution} for~\eqref{eq:318} is~\cite[Chapter 7]{Da-Prato2014a} a predictable $\mathcal{W}$-valued process
$X_{\ep,\delta}(t)=(\rho_{\ep,\delta}(t),j_{\ep,\delta}(t))$, $t\in[0,T]$, such that
\begin{align}
  \label{eq:922}
  \mathbb{P}\left(\int_{0}^{T}{\left\|X_{\ep,\delta}(s)\right\|^2_{\mathcal{W}}\m s}<\infty\right)=1,
\end{align}
and, for arbitrary $t\in[0,T]$
\begin{align*}
  X_{\ep,\delta}(t)=S(t)X_0+\int_{0}^{t}{S(t-s)\alpha_{\per}X_{\ep,\delta}(s)\m s}
  +\int_{0}^{t}{S(t-s)B_{N,\delta}(X_{\ep,\delta}(s))\m W_{\per,\ep}},\qquad\mathbb{P}\mbox{-a.s.}
\end{align*}

\begin{proof}[Proof of Proposition~\ref{p:11}]
  We apply~\cite[Theorem 4.5]{Tappe2012a} and take into account~\cite[Remark 4.6]{Tappe2012a}.
\end{proof}

The mild solution $X_{\ep,\delta}$ to~\eqref{eq:318} is, in particular, c\`adl\`ag at time $t=0$ with respect to the
$\mathcal{W}$-norm. Let us fix a parameter $\eta>\delta>0$. In addition to the hypotheses already given for $X_0$ in
Proposition~\ref{p:11}, we also assume
\begin{align}
  \label{eq:320}
  \rho_0(x)\geq \eta,\qquad \mbox{ for all } x\in D.
\end{align}
Keeping in mind Proposition~\ref{p:10} and the c\`adl\`ag properties at time $t=0$, we deduce the existence of a random time
$\zeta(\omega)$ such that
\begin{align}
  \label{eq:319}
  \|\rho_0(\cdot)-\rho(t,\cdot)\|_{L^{\infty}(D)}\leq \eta-\delta,\qquad \mbox{ for all } t\in[0,\zeta(\omega)).
\end{align}
The bound~\eqref{eq:319} implies that $B_{N,\delta}(X_{\ep,\delta}(s))$ coincides with $B_N(X_{\ep,\delta}(s))$ for
$s\in[0,\zeta(\omega))$. We thus have

\begin{theorem}
  Let the hypotheses of Proposition~\ref{p:11} be satisfied, as well as~\eqref{eq:320}. Then the \emph{regularised Dean--Kawasaki
    model}~\eqref{eq:501} admits a unique mild solution with respect to the $\mathcal{W}$-topology up to a random time $\zeta$.
\end{theorem}

\subsection{Estimates for $X_{\ep,\delta}$}
\label{ss:51}

We now study some moment bounds for the real-valued random variables $\|X_{\ep,\delta}(t)\|_{\mathcal{W}}$, where $X_{\ep,\delta}$
solves~\eqref{eq:318}.

\begin{prop}
  \label{p:12}
  Let $T>0$, $\delta>0$, and $q>2$ be fixed. Let $X_0\in \mathcal{W}$ be a deterministic initial condition for~\eqref{eq:318}. Let
  $\Theta=\Theta(T,q,\sigma,\delta,\ep,N):=\left\{C(q,T)\|X_0\|_{\mathcal{W}}^q+TC(\sigma,\delta)M^q(\ep,N)\right\}
  e^{C(T,q)+C(T,\sigma,\delta)M^q(\ep,N)}$. Then
    \begin{align}
    \label{eq:321}
    \sup_{t\in[0,T]}\mean{\|X_{\ep,\delta}(t)\|_{\mathcal{W}}^q}\leq \Theta.
  \end{align}
\end{prop}

\begin{proof}
We rely on some ideas of the proof of~\cite[Theorem 7.2]{Da-Prato2014a}. We know from Proposition~\ref{p:11} that the paths of
$X_{\ep,\delta}$ are c\`adl\`ag in the $\mathcal{W}$-topology. It follows that the real-valued process $t\mapsto
\|X_{\ep,\delta}(t)\|_{\mathcal{W}}^q$ is also c\`adl\`ag. This fact, together with~\eqref{eq:322}, allows us to deduce
\begin{align}
  \label{eq:323}
  \int_{0}^{T}{\|B_{N,\delta}(X_{\ep,\delta}(s))\|_{L^0_2(\mathcal{W})}^q\,\m s}<\infty,\quad \int_{0}^{T}{\|\alpha_{\per}(X_{\ep,\delta}(s))\|_{\mathcal{W}}\,\m s}<\infty, \quad\mathbb{P}\mbox{-a.s.}
\end{align} 
For $R\in\mathbb{N}$, we define the stopping times 
\begin{equation*}
  \tau_R:=
  \inf\left\{t\in(0,T]:\int_{0}^{t}{\|B_{N,\delta}(X_{\ep,\delta}(s))\|_{L^0_2(\mathcal{W})}^q\,\m s}
    \geq R\,\,\quad\mbox{or}\quad\int_{0}^{t}{\|\alpha_{\per}(X_{\ep,\delta}(s))\|_{\mathcal{W}}\,\m s}\geq R\right\},
\end{equation*}
with the usual convention $\tau_R:=T$ whenever the above infimum acts on the empty set. If we set
$X_{\ep,\delta,R}(t):=\mathbf{1}_{[0,\tau_R]}(t)X_{\ep,\delta}(t)$, it is then clear that
\begin{align*}
  X_{\ep,\delta,R}(t) 
  & = \mathbf{1}_{[0,\tau_R]}(t)S(t)X_0+\mathbf{1}_{[0,\tau_R]}(t)\int_{0}^{t}{\mathbf{1}_{[0,\tau_R]}(s)S(t-s)
    \alpha_{\per}X_{\ep,\delta,R}(s)\m s}\\
  & \quad + \mathbf{1}_{[0,\tau_R]}(t)\int_{0}^{t}{\mathbf{1}_{[0,\tau_R]}(s)S(t-s)B_{N,\delta}(X_{\ep,\delta,R}(s))\m W_{\per,\ep}}.
\end{align*}
We rely on~\cite[Theorem 4.36]{Da-Prato2014a},~\eqref{eq:322}, and the H\"older inequality and deduce
\begin{align}
  & \mean{\|X_{\ep,\delta,R}(t)\|_{\mathcal{W}}^q} \nonumber\\
  & \quad  \leq C(q,V_{\per})\left\{\|S(t)X_0\|^q_{\mathcal{W}}+\mean{\left(\int_{0}^{t}{\left\|X_{\ep,\delta,R}(s)\right\|_{\mathcal{W}}}
    \m s\right)^q}\right.\nonumber\\
    & \quad \quad +\left.\mean{\left\|\int_{0}^{t}{\mathbf{1}_{[0,\tau_R]}(s)S(t-s)B_{N,\delta}(X_{\ep,\delta,R}(s))\m W_{\per,\ep}}\right\|^q_{\mathcal{W}}}\right\}\nonumber\\
  & \quad  \leq C(q,V_{\per})\left\{\|X_0\|^q_{\mathcal{W}}+\mean{\left(\int_{0}^{t}{\left\|X_{\ep,\delta,R}(s)\right\|_{\mathcal{W}}}
    \m s\right)^q}+\mean{\int_{0}^{t}{\left\|B_{N,\delta}(X_{\ep,\delta,R}(s))\right\|^2_{L^0_2(\mathcal{W})}}\,\m s}^{q/2}\right\}\label{eq:324}\\
  & \quad  \leq C(q,T,V_{\per})\left\{\|X_0\|^q_{\mathcal{W}}+\int_{0}^{t}{\mean{\|X_{\ep,\delta,R}(s)\|_{\mathcal{W}}^q}\m s}
    +C(\sigma,\delta)M^q(\ep,N)\mean{\int_{0}^{t}{\left(1+\left\|X_{\ep,\delta,R}(s)\right\|^q_{\mathcal{W}}\right)}\,\m s}\right\}\nonumber\\
  % &  \quad \leq C(q,T)\left\{\|X_0\|^q_{\mathcal{W}}+C(\sigma,\delta)M^q(\ep,N)\mean{\int_{0}^{t}{\left(1+\left\|X_{\ep,\delta,R}(s)\right\|%^q_{\mathcal{W}}\right)}\,\m s}\right\}.\label{eq:325}
  & \quad \leq g_1+\int_{0}^{t}{g_2\mean{\|X_{\ep,\delta,R}(s)\|_{\mathcal{W}}^q}\m s},\label{eq:325}
\end{align}
where $g_1:=C(q,T,V_{\per})\|X_0\|_{\mathcal{W}}^q+TC(\sigma,\delta)M^q(\ep,N)$ and $g_2:=C(T,q)+C(\sigma,\delta)M^q(\ep,N)$. The
definition of $X_{\ep,\delta,R}$ implies that~\eqref{eq:324} is finite, hence so is
$\mean{\|X_{\ep,\delta,R}(t)\|_{\mathcal{W}}^q}$. We use Gronwall's lemma in~\eqref{eq:325} to conclude
\begin{align}
  \mean{\|X_{\ep,\delta,R}(t)\|_{\mathcal{W}}^q}\leq 
  \left\{C(q,T)\|X_0\|_{\mathcal{W}}^q+TC(\sigma,\delta)M^q(\ep,N)\right\}e^{C(T,q)+C(T,\sigma,\delta)M^q(\ep,N)},
  \quad\mbox{for all } t\in[0,T].
\end{align}
The integrability property~\eqref{eq:323} implies that $\tau_R(\omega)=T$ for $R\geq R(\omega),\,\mathbb{P}\mbox{-a.s.}$ As a
result, we deduce
\begin{align*}
  \lim_{R\rightarrow +\infty}{X_{\ep,\delta,R}(t)}=X_{\ep,\delta}(t)\mbox{ in }\mathcal{W},\quad t\in[0,T],\quad\mathbb{P}\mbox{-a.s.}
\end{align*} 
We use Fatou's lemma and we obtain
\begin{align*}
  \mean{\|X_{\ep,\delta}(t)\|_{\mathcal{W}}^q} 
  & \leq \liminf_{R\rightarrow +\infty}{\mean{\|X_{\ep,\delta,R}(t)\|_{\mathcal{W}}^q}}\\
  & \leq \left\{C(q,T)\|X_0\|_{\mathcal{W}}^q+TC(\sigma,\delta)M^q(\ep,N)\right\}e^{C(T,q)+C(T,\sigma,\delta)M^q(\ep,N)},
    \qquad\mbox{ for all } t\in[0,T].
\end{align*}
Taking the supremum in time finally yields the result.
\end{proof}

We obtained~\eqref{eq:321} by using the c\`adl\`ag property of the solution $X_{\ep,\delta}$. This allows us to consider an
arbitrary $q>2$. If we only relied the definition of mild solution (see in particular~\eqref{eq:922}), the exponent $q=2$ would be
the maximum exponent we could take. This is exactly the case for the proof of uniqueness in~\cite[Theorem 7.2]{Da-Prato2014a},
from which we adapted the proof of Proposition~\ref{p:12}. The proof of~\cite[Theorem 7.2, (7.6)]{Da-Prato2014a}, which is exactly
our~\eqref{eq:321}, relies on a fixed point argument instead. We cannot use this argument, since we lack the global Lipschitz
property for the stochastic integrand $B_{N,\delta}$. The need for $q>2$, and not simply $q=2$, is motivated by~\cite[Proposition
7.3]{Da-Prato2014a}, which we will use in the next section.

\subsection{Small-noise regime analysis}
\label{ss:52}

In this subsection, we investigate the small-noise regime analysis for solutions $X_{\ep,\delta}$ to~\eqref{eq:318}.

\begin{prop}
  \label{p:13}
  Let the hypotheses of Proposition~\ref{p:11} be satisfied. In addition, assume the following scaling for $\ep,N$
  \begin{align}
    \label{eq:412}
    N\ep^{\theta}\geq 1,\qquad\mbox{for some }\theta>7.
  \end{align}
  For fixed $\delta>0$, $T>0$, $r>0$, $q>2$, we have 
  \begin{align*}
    \lim_{\ep\downarrow 0}{\mathbb{P}\left(\sup_{t\in[0,T]}\left\| X_{\ep,\delta}(t)-Z(t)\right\|^q_{\mathcal{W}}\geq r\right)}=0,
  \end{align*}
  where $Z$ is the unique (deterministic) solution of
  \begin{align}
    \label{eq:413}
    \left\{
    \begin{array}{l}
      \mbox{\emph{d}}Z(t)=(AZ(t)+\alpha_{\per}Z(t))\mbox{\emph{d}}t, \\
      Z(0)=X_0.
    \end{array}
    \right.
  \end{align}
\end{prop}

\begin{proof}
  We adapt the proof of~\cite[Proposition 12.1]{Da-Prato2014a}. The scaling~\eqref{eq:412} implies that $M(\ep,N)\rightarrow0$ in
  the simultaneous limit of $\ep$ and $N$. We write
\begin{align*}
  X_{\ep,\delta}(t)-Z(t)=\int_{0}^{t}{S(t-s)\alpha_{\per}(X_{\ep,\delta}(s)-Z(s))\m s}
  +\int_{0}^{t}{S(t-s)B_{N,\delta}(X_{\ep,\delta}(s))\m W_{\per,\ep}}.
\end{align*}
We use~\cite[Proposition 7.3]{Da-Prato2014a} and Proposition~\ref{p:12} to deduce
\begin{align}
  & \mean{\sup_{s\in[0,t]}{\left\|X_{\ep,\delta}(s)-Z(s)\right\|^q_{\mathcal{W}}}} \nonumber\\
  & \quad \leq C(T,q,V_{\per})\mean{\int_{0}^{t}{\left\|X_{\ep,\delta}(u)-Z(u)\right\|^q_{\mathcal{W}}\m u}}
    +\mean{\sup_{s\in[0,T]}{\left\|\int_{0}^{s}{S(t-s)B_{N,\delta}(X_{\ep,\delta})\m W_{\per,\ep}}\right\|^q}}\nonumber\\ 
  & \quad \leq C(T,q,V_{\per})\mean{\int_{0}^{t}{\left\|X_{\ep,\delta}(u)-Z(u)\right\|^q_{\mathcal{W}}\m u}}
    +C(\sigma,\delta,T,q)M^q(\ep,N)\mean{\int_{0}^{T}{(1+\|X_{\theta,\delta}\|^q_{\mathcal{W}})\m s}}\label{eq:910}\\
  & \quad \leq C(T,q,V_{\per})\int_{0}^{t}{\mean{\sup_{s\in[0,u]}{\left\|X_{\ep,\delta}(u)-Z(u)\right\|^q_{\mathcal{W}}}}\m u}
    +C(\sigma,\delta,T,q)M^q(\ep,N)T(1+\Theta)\label{eq:911},
%& \quad \quad + C(\sigma,\delta,T,q)M^q(\ep,N)\mean{\int_{0}^{T}{(1+\|X_{\theta,\delta}\|^q_{\mathcal{W}})\m s}}\\ 
%& \quad \leq C(\sigma,\delta,T,q)M^q(\ep,N)\mean{\int_{0}^{T}{(1+\|X_{\theta,\delta}\|^q_{\mathcal{W}})\m s}}\nonumber\\
%& \quad \leq C(\sigma,\delta,T,q)M^q(\ep,N)T(1+\theta),
\end{align}
where $\Theta$ is defined in Proposition~\ref{p:12}. Thanks to the same proposition,~\eqref{eq:910} is finite. The
scaling~\eqref{eq:412} also implies that $\Theta$ is bounded in $\ep,N$. We can apply the Gronwall inequality to~\eqref{eq:911} to
deduce that
\begin{align*}
  \mean{\sup_{s\in[0,T]}{\left\|X_{\ep,\delta}(s)-Z(s)\right\|^q_{\mathcal{W}}}}
  \leq C(\sigma,\delta,T,q)M^q(\ep,N)T(1+\theta)e^{C(T,q,V_{\per})}\rightarrow 0\qquad\mbox{as }\ep\rightarrow 0, \, N\rightarrow \infty.
\end{align*}
Chebyshev's inequality gives the result.
\end{proof}

The prescribed scaling in $N,\ep$ stated in Proposition~\ref{p:13} is compatible with the scalings of Propositions~\ref{p:1}
and~\ref{p:22}, and Theorem~\ref{thm:1}. See also Remark~\ref{rem:2}.

\subsection{Main existence and uniqueness result}
\label{ss:100}
We now turn to the key existence and uniqueness result for the regularised Dean--Kawasaki model~\eqref{eq:501}, or
equivalently~\eqref{eq:420}.

\begin{rem}
  \label{rem:100}
  Let us fix $\eta>\delta>0$. We first notice that, for a deterministic initial condition $X_0=(\rho_0,j_0)\in\mathcal{W}$ such
  that~\eqref{eq:320} is satisfied, there exists $T=T(X_0)\in(0,\infty]$ such that the solution $Z$ to~\eqref{eq:413} satisfies
  \begin{align*}
    Z(t,x)\geq \delta + (\eta-\delta)/2,\quad \quad\mbox{ for all } x\in D,\,\mbox{ for all } t\in[0,T).
  \end{align*}
  This is implied by the time-continuity of $Z$ with respect to the $\mathcal{W}$-norm, and by Proposition~\ref{p:10}.
\end{rem}

\begin{proof}[Proof of Theorem~\ref{thm:100}] 
Fix $\delta$ so that $0<\delta<\eta$ and consider $T(X_0)$ as indicated in Remark~\ref{rem:100}. Proposition~\ref{p:11} provides
existence of a solution $X_{\ep,\delta}$ to~\eqref{eq:318}.  For some $q>2$, we rely on Proposition~\ref{p:10} and write
\begin{align*}
  & \mathbb{P}\!\left(\sup_{t\in[0,T(X_0)]}\!\left\| X_{\ep,\delta}(t)-Z(t)\right\|_{C^0_{\per}(D)\times C^0_{\per}(D)}
    \geq \frac{\eta-\delta}{2}\right) \\
  & \quad = \mathbb{P}\!\left(\sup_{t\in[0,T(X_0)]}\!\left\| 
    X_{\ep,\delta}(t)-Z(t)\right\|^q_{C^0_{\per}(D)\times C^0_{\per}(D)}\geq \frac{(\eta-\delta)^q}{2^q}\right) \\
    & \quad \leq \mathbb{P}\left(\sup_{t\in[0,T(X_0)]}\left\| X_{\ep,\delta}(t)-Z(t)\right\|^q_{\mathcal{W}}
    \geq \frac{C^{-q}(\eta-\delta)^q}{2^q}\right)\leq \nu,
\end{align*}
where the last inequality holds for $\ep$ small enough (or equivalently $N$ big enough), thanks to Proposition~\ref{p:13}. It
follows that
\begin{equation*}
  \mathbb{P}\left(X_{\ep,\delta}(x,t)\geq \delta,\mbox{ for all } t\in[0,T(X_0)),\mbox{ for all } x\in D\right)\geq 1-\nu.
\end{equation*} 
This implies that
$\mathbb{P}\left(B_{N,\delta}(X_{\ep,\delta})=B_{N}(X_{\ep,\delta}),\mbox{ for all } t\in[0,T(X_0))\right)\geq 1-\nu$. We take
$X_{\ep}:=X_{\ep,\delta}$, and employ the existence and uniqueness results from Proposition~\ref{p:11} to conclude the proof.
\end{proof}

The dependence of $T$ on $X_0$ is yet to be properly investigated. In the special case of constant initial data
$X_0=(\rho_0,j_0)=(C,0)$, for some $C>\delta>0$, the solution is stationary, hence we can pick any finite $T(X_0)$.

\begin{rem}
  \label{rem:101}
  We have relied on scalings of type $N\ep^\theta=1$ (or $N\ep^\theta\geq 1$), for some $\theta>0$, to prove several results
  throughout
  the paper. Some of these scalings could be improved (i.e., $\theta$ could be lowered) in at least two points, specifically:\vspace{0.5 pc}\\
  (a) \emph{Tightness of $\{\rho_{\ep}\}_{\ep}$, Proposition~\ref{p:1}}: We relied on the compact embedding $H^1(D)\subset L^2(D)$
  to show that the initial conditions $\{\rho_{\ep}(\cdot,0)\}_{\ep}$ are tight in $L^2$. If one uses the compact embedding
  $H^{1/2+\delta/2}(D)\subset L^2(D)$ instead, for some $\delta\in(0,1)$, the scaling is less demanding, as
  $\|w_{\ep}(\cdot)\|_{H^{1/2+\delta/2}}\propto \ep^{-2-\delta}$.

  In addition, the time regularity estimate can be improved by computing the expectation first in the second-to-last inequality
  of~\eqref{eq:12}. In this case, the estimate proceeds with the bound
  \begin{align*} \underbrace{1-\frac{\sqrt{4\pi\epsilon^2}}{\sqrt{2\pi(2\epsilon^2+V_{s,t})}}}_{=:T_1}
+\underbrace{\frac{\sqrt{4\pi\epsilon^2}}{\sqrt{2\pi(2\epsilon^2+V_{s,t})}}\left(1-\exp\left\{-\frac{\mu_{s,t}^2}{2(2\epsilon^2+V_{s,t})}\right\}\right)}_{=:T_2}
\end{align*}
where \begin{align*}
\mu_{s,t} & :=\mathbb{E}\left[q(t)-q(s)\right]\leq C|t-s|, \qquad V_{s,t} := \mbox{Var}\left[q(t)-q(s)\right] \leq C|t-s|^2.
\end{align*} 
It is not difficult to bound $T_1$ and $T_2$ by $C\ep^{-1-\delta}|t-s|^{1+\delta}$, where $\delta$ can be chosen in $(0,1]$. Overall, the scaling $N\ep^{2+\delta}$, for some $\delta\in(0,1]$, is sufficient to provide tightness of $\{\rho_{\ep}\}_{\ep}$. We believe that similar arguments could be applied to $\{j_{\ep}\}_{\ep}$ and $\{j_{2,\ep}\}_{\ep}$ as well.\vspace{0.5 pc}\\
(b) \emph{Functional setting of Section~\ref{s:5}}. If we redefine $\mathcal{W}$ as $\mathcal{W}:=H^{1/2+\delta/2}_{\per}(D)\times H^{1/2+\delta/2}_{\per}(D)$, this could lead to a better scaling in Lemma \ref{lem:20}, for analogous reasons to point~(a). This would then lead to a better scaling in Theorem~\ref{thm:100}.
\end{rem}

%% Comment FC15 here

\appendix

\section{Gaussian tools}
\label{ap:1}

This appendix is devoted to a concise exposition of a few useful facts concerning Gaussian random variables.

\begin{mydef}\label{la:1}
  A \emph{Gaussian random vector} $X$ with mean $\mu\in\mathbb{R}^d$ and covariance matrix $\Sigma$, denoted as
  $X\sim\mathcal{N}(\mu,\Sigma)$, has the probability density function given by
  $\mathcal{G}(x,\mu,\Sigma)=\emph{\text{det}}(2\pi\Sigma)^{-1/2}\exp\left\{-\frac{1}{2}(x-\mu)^T\Sigma^{-1}(x-\mu)\right\}$.  In
  the real-valued case, i.e., for $X$ of mean $\mu$ and variance $\sigma^2$, the above is simply
  \begin{align*}
    \mathcal{G}(x,\mu,\sigma^2):=\frac{1}{\sqrt{2\pi\sigma^2}}\exp\left\{-\frac{(x-\mu)^2}{2\sigma^2}\right\}.
  \end{align*}
\end{mydef}

\begin{lemma}[Fourier Transform for Gaussians]
  \label{la:2}
  The Fourier transform of an $\mathbb{R}^d$-valued Gaussian random vector $Y\sim\mathcal{N}(\mu,\Sigma)$ is given by
  \begin{align*}
    \mathbb{R}^d\ni \xi\mapsto \mean{e^{-i\langle \xi,Y\rangle}}
    =\exp\left\{-i\langle \mu,\xi\rangle-\frac{1}{2}\langle\xi,\Sigma\xi\rangle\right\}.
  \end{align*}
\end{lemma}

\begin{lemma}[Conditional law for Gaussian vectors]
  \label{la:3}
  Let $b\in\mathbb{R}$. For a bivariate Gaussian random vector $Y=(Y_1,Y_2)$, the conditional density of $Y_1$ given
  $Y_2=b$ is
  \begin{align*}
    f_{Y_1|Y_2}(y_1|Y_2=b)=\mathcal{G}\left(y_1,\mu_{Y_1}+\frac{\sigma_{Y_1}}{\sigma_{Y_2}}
    \chi(b-\mu_{Y_2}),(1-\chi^2)\sigma^2_{Y_1}\right),%,\qquad \mbox{for all }b\in\mathbb{R},
  \end{align*}
  where $\chi=\mbox{\emph{Corr}}(Y_1,Y_2)$.
\end{lemma}

Lemma~\ref{la:2} can be found in~\cite[Chapter 16]{Jacod2003b}, and Lemma~\ref{la:3} can be found in~\cite[Section
4.7]{Bertsekas2002a}.

\begin{lemma}[Multiplication of Gaussian kernels]
  \label{la:4} 
  Given $f(x):=\mathcal{G}(x,\mu_f,\sigma^2_f)$ and $g(x):=\mathcal{G}(x,\mu_g,\sigma^2_g)$, we have the multiplication rule
  \begin{align*}
    f(x)g(x)=\mathcal{G}(x,\mu_{fg},\sigma^2_{fg})\frac{1}{\sqrt{2\pi(\sigma^2_f+\sigma^2_g)}}
    \exp\left\{-\frac{(\mu_f-\mu_g)^2}{2(\sigma^2_f+\sigma^2_g)}\right\},
  \end{align*}
  where we have set
  \begin{align*}
    \mu_{fg}:=\frac{\mu_f\sigma^2_g+\mu_g\sigma^2_f}{\sigma^2_f+\sigma^2_g},\qquad 
    \sigma^2_{fg}:=\frac{\sigma^2_f\sigma^2_g}{\sigma^2_f+\sigma^2_g}.
  \end{align*}
\end{lemma}
%{\color{blue} reference}

\begin{lemma}[Moments of Gaussian random variables]
  \label{la:5} 
  Let $X\sim\mathcal{N}(\mu,\sigma^2)$. For $n\in\mathbb{N}$, we have 
  \begin{align*}
    M(n,\mu,\sigma^2) & :=\mathbb{E}\left[|X|^n\right] \leq C(n)\left\{\mu^n+\sigma^n(n-1)!!\right\},\\
    m(n,\mu,\sigma^2) & :=\mathbb{E}\left[X^n\right] =\sum_{j\in\mathbb{N},\,2j\leq n}{(2j-1)!!\binom{n}{2j}\sigma^{2j}\mu^{n-2j}},
  \end{align*}
  where $n!!:=\sum_{k=0}^{\lceil n/2\rceil-1}{(n-2k)}$, for $n\in\mathbb{N}$.
\end{lemma}

Lemma~\ref{la:5} can be proved by induction on $n$, by splitting $X$ as $(X-\mu)+\mu$ and using the results for moments of
zero-mean Gaussian random variables. Lemma~\ref{la:4} follows from simple algebraic computations.

\begin{lemma}[Ornstein-Uhlenbeck process]
  \label{la:6}
  Let $A,\Sigma\in\mathbb{R}^{2\times 2}$, and let $W$ be a bivariate Brownian motion. For any $t\in[0,T]$, set
  $\Phi(t):=e^{At}$.
  \begin{enumerate}
  \item \label{it:OU1}
    The stochastic equation
    \begin{align}
      \label{eq:131}
      \emph{\m} X(t)= AX(t)\emph{\m} t+\Sigma\emph{\m} W(t),\qquad X(0)=X_0
    \end{align}
    has a unique solution $X(t)=(X_1(t),X_2(t))$ explicitly given by
    \begin{align}
      \label{eq:132}
      X(t)=\Phi(t)X_0+\Phi(t)\int_{0}^{t}{\Phi^{-1}(s)\Sigma\emph{\m}W(s)}.
    \end{align}
  \item \label{it:OU2} If $X_0$ is a Gaussian random vector independent of $W$, then $X(t)$ is a Gaussian random vector for any
    $t\in[0,T]$.
  \item \label{it:OU3} With the same assumption as in~\ref{it:OU2}, if in addition $\emph{\mbox{Cov}}(X_0,X_0)$ is positive
    definite, then there exists $\nu>0$ such that $\mbox{\emph{Var}}(X^1(t))\geq \nu$ and $\mbox{\emph{Var}}(X^2(t))\geq \nu$, for
    any $t\in[0,T]$.
  \item \label{it:OU4} With the same assumption as in~\ref{it:OU3}, the following quantities are Lipschitz on $[0,T]$: the mean of
    $X_1(t)$ and $X_2(t)$, the variance of $X_1(t)$ and $X_2(t)$, the correlation between $X_1(t)$ and $X_2(t)$.
  \end{enumerate}
\end{lemma}

\begin{proof}
\emph{Part~\ref{it:OU1}}. Existence and uniqueness of a solution is granted by~\cite[Theorem 5.2.1]{Oksendal2003a}. It is
straightforward to see that~\eqref{eq:132} is indeed the solution by computing the It\^o-differential of $X(t)$.

\emph{Part~\ref{it:OU2}}. The integrand $\Phi^{-1}(s)\Sigma$ being deterministic, we have that
$\Phi(t)\int_{0}^{t}{\Phi^{-1}(s)\Sigma\m W(s)}$ is a Gaussian process. In addition, $\Phi(t)X_0$ is a Gaussian vector by
linearity. Stochastic independence of $X_0$ and $W$ grants that the sum of the aforementioned two vectors is a Gaussian vector.

\emph{Part~\ref{it:OU3}}. Thanks to the independence of $W$ and $X_0$, we can limit ourselves to studying
$\mbox{Cov}(\Phi(t)X_0,\Phi(t)X_0)$. We observe that
\begin{align*}
  \mbox{Cov}(\Phi(t)X_0,\Phi(t)X_0)=\Phi(t)\mbox{Cov}(X_0,X_0)\Phi^{T}(t)=:B(t).
\end{align*}
Since $\mbox{Cov}(X_0,X_0)$ is definite positive, this entails that the continuous function $t\mapsto y^TB(t)y$ is strictly
positive on $[0,T]$ for any given $y\in\mathbb{R}^2\setminus{\{(0,0)\}}$. The claim then follows by taking $y=(1,0)$ and
$y=(0,1)$.

\emph{Part~\ref{it:OU4}}. We notice that
\begin{align*}
  \left\|\mean{X(t)-X(s)}\right\|=\left\|\mean{\left(\Phi(t)-\Phi(s)\right)X_0}\right\|\leq C(A)\mean{\|X_0\|}|t-s|,
\end{align*}
and the Lipschitz property for the mean of $X_1(t)$ and $X_2(t)$ is settled. As for the variances, we compute
\begin{align}
  \label{eq:133}
  \mbox{Cov}(X(t),X(t))-\mbox{Cov}(X(s),X(s)) 
  & = \Phi(t)\left[\int_{0}^{t}{\Phi^{-1}(u)\Sigma\Sigma^{T}\Phi^{-T}(u)\m u}\right]\Phi^T(t)\nonumber\\
  & \quad - \Phi(s)\left[\int_{0}^{s}{\Phi^{-1}(u)\Sigma\Sigma^{T}\Phi^{-T}(u)\m u}\right]\Phi^T(s)\nonumber\\
  & \quad + \Phi(t)\mbox{Cov}(X_0,X_0)\Phi^{T}(t)-\Phi(s)\mbox{Cov}(X_0,X_0)\Phi^{T}(s),
\end{align}
and the Lipschitz property for the variance of $X_1(t)$ and $X_2(t)$ follows from the Lipschitz property for $\Phi(t)$ and
$\int_{0}^{t}{\Phi^{-1}(u)\Sigma\Sigma^{T}\Phi^{-T}(u)\m u}$. As for the correlation between $X_1(t)$ and $X_2(t)$, the Lipschitz
property can be derived by using the definition
\begin{align*}
  \mbox{Corr}(X_1(t),X_2(t)):=\frac{\mbox{Cov}(X_1(t),X_2(t))}{\sqrt{\mbox{Var}(X_1(t))\mbox{Var}(X_2(t))}}
\end{align*}
and observing that $\mbox{Var}(X_1(t))$, $\mbox{Var}(X_2(t))$ are bounded away from $0$ (by \ref{it:OU3}), and that
$\mbox{Var}(X_1(t))$, $\mbox{Var}(X_2(t))$, $\mbox{Cov}(X_1(t),X_2(t))$ are Lipschitz by~\eqref{eq:133}.
\end{proof}

\section{Auxiliary tools}
\label{sec:4}

We list and prove some auxiliary tools used repeatedly in the proofs of the main results of Section~\ref{s:4}. We start with time
regularity of Gaussian moments, under Assumption~(G), in Subsection~\ref{ss:1}. We deal with time regularity for the
Fokker--Planck equation~\eqref{eq:134} under Assumption~(NG) in Subsection~\ref{sec:12}. We estimate the second moment of
$\rho^{-1}_{\ep}(x,t)$, where $\rho_{\ep}(x,t)$ is defined in~\eqref{eq:102}, giving a proof for both Assumption~(G) and
Assumption~(NG), in Subsection~\ref{ss:2}.

\subsection{Time regularity of specific Gaussian moments}
\label{ss:1}

\begin{lemma}
  \label{lem:1}
  Let $T>0$, $n\in\mathbb{N}$, $c\geq 2$, $\nu>0$ be real numbers. Let $\mu,\sigma^2\colon[0,T]\rightarrow\mathbb{R}$ be Lipschitz
  functions, with Lipschitz constant $L$. Let $\mathcal{Q}_{n,t}(x)$ be a polynomial of degree $n$ in $x$, and Lipschitz
  coefficients in $t$, again with Lipschitz constant $L$. Assume that $\sigma^2(t)\geq \nu,\mbox{ for all } t\in[0,T]$. Then there
  exists $\beta>0$ such that
  \begin{equation*} \int_{\mathbb{R}}{\left|\mathcal{Q}_{n,t}(x)\mathcal{G}(x,\mu(t),\sigma^2(t))
        -\mathcal{Q}_{n,s}(x)\mathcal{G}(x,\mu(s),\sigma^2(s))\right|^c\emph{\m} x}\leq C|t-s|^{1+\beta},\qquad \mbox{ for all } s,t\in[0,T],
  \end{equation*}
  for a constant $C=C(T,\nu,L,c)$. In addition, if $p=0$, $c=2$, and $\mathcal{Q}_{n,t}$ is a constant, then $\beta=1$.
\end{lemma}

\begin{proof} 
Because of the general inequality $\left|\sum_{i=0}^{n}{a_i}\right|^c\leq (n+1)^c\sum_{i=0}^{n}{\left|a_i\right|^c}$, it is
sufficient to prove the statement for each monomial composing $\mathcal{Q}_{n,t}(x)$. We can thus restrict ourselves to proving
the statement with the choice $\mathcal{Q}_{p,t}(x):=A(t)x^p$, for any $p\in\mathbb{N}$, and where $A$ is Lipschitz with constant
$L$.

We add and subtract relevant quantities in the integral we have to compute. As a result we get
\begin{align*}
  &\int_{\mathbb{R}}{\left|A(t)x^p\mathcal{G}(x,\mu(t),\sigma^2(t))-A(s)x^p\mathcal{G}(x,\mu(s),\sigma^2(s))\right|^c \m x}\\
  & \leq  2^c\underbrace{\int_{\mathbb{R}}{\left|(A(t)-A(s))x^p\mathcal{G}(x,\mu(t),\sigma^2(t))\right|^c \m x}}_{=:T_1}
    + 2^c\underbrace{\int_{\mathbb{R}}{\left|A(s)x^p\left(\mathcal{G}(x,\mu(t),\sigma^2(t))
    -\mathcal{G}(x,\mu(s),\sigma^2(s))\right)\right|^c \m x}}_{=:T_2}.
\end{align*} 
We estimate $T_1,T_2$ separately. Since $A$ is Lipschitz and $\sigma^2$ is bounded from below, we obtain
\begin{align*}
  T_1 & \leq  L^c|t-s|^c\int_{\mathbb{R}}{|x|^{cp}\mathcal{G}(x,\mu(t),\sigma^2(t))^c \m x} = 
        \frac{L^c}{c^{1/2}(2\pi\sigma^2(t))^{(c-1)/2}}\,M\left(cp,\mu(t),\frac{\sigma^2(t)}{c}\right)|t-s|^c\\
      & \leq \frac{L^c}{c^{1/2}(2\pi\nu)^{(c-1)/2}}\,C(T,p,c)|t-s|^c\leq C |t-s|^c,
\end{align*}
where we have also relied on Lemmas~\ref{la:4},~\ref{la:5}. In order to estimate $T_2$, we rewrite the integral as
\begin{align}
  \label{eq:1004}
  \int_{\mathbb{R}}{|A|^c(s)|x|^{cp}\!\left|\mathcal{G}(x,\mu(t),\sigma^2(t))
      -\mathcal{G}(x,\mu(s),\sigma^2(s))\right|^{\alpha}\!\cdot\!\left|\mathcal{G}(x,\mu(t),\sigma^2(t))
      -\mathcal{G}(x,\mu(s),\sigma^2(s))\right|^{c-\alpha}\m x}
\end{align}
for some $\alpha\in(c-2,c-1)$. We apply the H\"older inequality with conjugate exponents $\frac{2}{c-\alpha}$ and
$\frac{2}{2-c+\alpha}$ and obtain
\begin{align*}
  T_2 & \leq  \left(\int_{\mathbb{R}}{|A|^{2c/(2-c+\alpha)}(s)|x|^{2pc/(2-c+\alpha)}\!\left|\mathcal{G}(x,\mu(t),\sigma^2(t))
        -\mathcal{G}(x,\mu(s),\sigma^2(s))\right|^{2\alpha/(2-c+\alpha)}\m x}\right)^{(\alpha+2-c)/2}\\
      & \quad \times  \left(\int_{\mathbb{R}}{\left|\mathcal{G}(x,\mu(t),\sigma^2(t))
        -\mathcal{G}(x,\mu(s),\sigma^2(s))\right|^{2}\m x}\right)^{\frac{c-\alpha}{2}}.
\end{align*}

The first term can be controlled using the boundedness of $A$ and Lemmas~\ref{la:4},~\ref{la:5}, similarly to the argument for
$T_1$. We get
\begin{align*}
  &  \left(\int_{\mathbb{R}}{|A|^{2c/(2-c+\alpha)}(s)|x|^{2pc/(2-c+\alpha)}\!\left|\mathcal{G}(x,\mu(t),\sigma^2(t))
    -\mathcal{G}(x,\mu(s),\sigma^2(s))\right|^{2\alpha/(2-c+\alpha)}\m x}\right)^{(\alpha+2-c)/2} \\
  & \leq  C(A,c,p,\nu)\left\{M\left(\frac{2pc}{2-c+\alpha},\mu(t),\frac{\sigma^2(t)(2-c+\alpha)}{2\alpha}\right)
    +M\left(\frac{2pc}{2-c+\alpha},\mu(s),\frac{\sigma^2(s)(2-c+\alpha)}{2\alpha}\right)\right\}\\
  & \leq C(A,c,p,\nu,\alpha).
\end{align*}

As for the second term of the product bounding $T_2$, we rely on Fourier analysis and Taylor expansions. More precisely, we rely
on Parseval's equality, Lemma~\ref{la:2}, and some simple rearrangement to write
\begin{align*}
  & \int_{\mathbb{R}}{\left|\mathcal{G}(x,\mu(t),\sigma^2(t))-\mathcal{G}(x,\mu(s),\sigma^2(s))\right|^{2}\m x}
    =C\int_{\mathbb{R}}{\left|e^{-i\mu(t)\xi-\frac{1}{2}\sigma^2(t)\xi^2}-e^{-i\mu(s)\xi-\frac{1}{2}\sigma^2(s)\xi^2}\right|^{2}\m\xi}\\
  & \quad \leq  C \int_{\mathbb{R}}{\left|\left\{e^{-i\mu(t)\xi}-e^{-i\mu(s)\xi}\right\}e^{-\frac{1}{2}\sigma^2(t)\xi^2}\right|^{2}\m\xi} 
    + C\int_{\mathbb{R}}{\left|e^{-i\mu(s)\xi}\left\{e^{-\frac{1}{2}\sigma^2(t)\xi^2}-e^{-\frac{1}{2}\sigma^2(s)\xi^2}\right\}\right|^{2}\m\xi}\\
  & \quad \leq C \underbrace{\int_{\mathbb{R}}{\left|\left\{e^{-i\mu(t)\xi}-e^{-i\mu(s)\xi}\right\}
    e^{-\frac{1}{2}\sigma^2(t)\xi^2}\right|^{2}\m\xi}}_{=:T_3}         
    +\,C\underbrace{\int_{\mathbb{R}}{\left|e^{-\frac{1}{2}\sigma^2(t)\xi^2}-e^{-\frac{1}{2}\sigma^2(s)\xi^2}\right|^{2}\m\xi}}_{=:T_4}.
\end{align*}

For $T_3$, we use the mean value theorem applied to the map $y\mapsto e^{iy}$ and the Lipschitz properties of $\mu$ to deduce
\begin{align*}
  T_3 \leq L^2|t-s|^2\int_{\mathbb{R}}{\xi^2e^{-\sigma^2(t)\xi^2}\m \xi} = L^2|t-s|^2\sqrt{2\pi}\left[\frac{1}{\sigma(t)^2}\right]^{3/2}
  \leq C(L,\nu)|t-s|^2,
\end{align*} 
where have used the definition of the Gaussian kernel and the bound $\sigma^2(t)\geq \nu$. We move on to $T_4$. We rely on
Lemma~\ref{la:4} and we expand the square in the integrand to deduce
\begin{align*}
  T_4=\sqrt{\frac{\pi}{\sigma^2(t)}}+\sqrt{\frac{\pi}{\sigma^2(s)}}-2\sqrt{\frac{2\pi}{\sigma^2(t)+\sigma^2(s)}}\leq C(\nu)\left|\sigma^2(t)-\sigma^2(s)\right|^2\leq C(\nu)|t-s|^2.
\end{align*}
The second inequality above is the Lipschitz property of $\sigma^2$, while the first inequality is justified by the midpoint
estimate $f(\sigma^2(t))+f(\sigma^2(s))-2f([\sigma^2(t)+\sigma^2(s)]/2)\leq C(\nu)\left|\sigma^2(t)-\sigma^2(s)\right|^2$ for the
function $f\colon [\nu,\infty)\rightarrow\mathbb{R}\colon y\mapsto\sqrt{\pi/y}$. Such expansion is a consequence of the
superposition of the second-order Taylor expansions (with Lagrange remainder) of $f(\sigma^2(t))$ and $f(\sigma^2(s))$ centred
around $[\sigma^2(t)+\sigma^2(s)]/2$.  Putting $T_3$ and $T_4$ together, we deduce
\begin{align*}
  \left(\int_{\mathbb{R}}{\left|\mathcal{G}(x,\mu(t),\sigma^2(t))-\mathcal{G}(x,\mu(s),\sigma^2(s))\right|^{2}\m x}\right)^{\frac{c-\alpha}{2}}
  \leq C|t-s|^{2\cdot\frac{c-\alpha}{2}}= C|t-s|^{c-\alpha}.
\end{align*}
We rename $\beta:=c-\alpha-1\in(0,1)$. We combine the above estimates and we obtain
\begin{equation*}
  \int_{\mathbb{R}}{\left|A(t)x^p\mathcal{G}(x,\mu(t),\sigma^2(t))-A(s)x^p\mathcal{G}(x,\mu(s),\sigma^2(s))\right|^c \m x}
  \leq C|t-s|^{1+\beta},
\end{equation*}
as desired. If $p=0$, $c=2$, and $\mathcal{Q}_{n,t}$ is a constant, then $\beta=1$. This is because $T_1=0$, and one may simply
take $\alpha=0$ in~\eqref{eq:1004}.
\end{proof}

\begin{lemma}
  \label{lem:2}
  Let $X\sim\mathcal{N}(\mu,\sigma^2)$ and let $x\in\mathbb{R}$. Then 
  \begin{align}
    \mathbb{E}\left[w_{\epsilon}(x-X)X^n\right] 
    & =  \mathcal{G}(x,\mu,\epsilon^2+\sigma^2)\cdot\!m\left(n,\frac{x\sigma^2+\mu\ep^2}{\ep^2+\sigma^2},
      \frac{\ep^2\sigma^2}{\ep^2+\sigma^2}\right),\quad n\in\mathbb{N}\cup\{0\}.\label{eq:4}\\
    \mathbb{E}\left[w'_{\epsilon}(x-X)X^n\right] 
    & = \frac{\mathcal{G}(x,\mu,\epsilon^2+\sigma^2)}{\ep^2}\sum_{k=0}^{n}{\binom{n}{k}x^{n-k}m\left(k+1,\frac{(\mu-x)\ep^2}
      {\ep^2+\sigma^2},\frac{\ep^2\sigma^2}{\ep^2+\sigma^2}\right)},\quad n\in\mathbb{N}\cup\{0\}.\label{eq:130}\\
    \mathbb{E}\left[w''_{\epsilon}(x-X)X^n\right] 
    & = \mathcal{G}(x,\mu,\epsilon^2+\sigma^2)\sum_{k=0}^{n}{\binom{n}{k}x^{n-k}
      \left\{-\frac{1}{\ep^2}m\left(k,\frac{(\mu-x)\ep^2}{\ep^2+\sigma^2},\frac{\ep^2\sigma^2}{\ep^2+\sigma^2}\right)\right.}\nonumber\\
    & \quad +\left.\frac{1}{\ep^4}m\left(k+2,\frac{(\mu-x)\ep^2}{\ep^2
      +\sigma^2},\frac{\ep^2\sigma^2}{\ep^2+\sigma^2}\right)\right\}\label{eq:2001}.
  \end{align}
\end{lemma}

The proof of Lemma~\ref{lem:2} is a straightforward application of multiplication properties for Gaussian kernels and Gaussian
moments, as stated in Lemmas~\ref{la:4} and~\ref{la:5}.

\begin{rem}
  \label{rem:1}
  It is worth noticing that the right-hand-sides of~\eqref{eq:4}, \eqref{eq:130} and~\eqref{eq:2001} satisfy the requirements of
  Lemma~\ref{lem:1}. To see this, we notice that
  \begin{equation*}
    m\!\left(n,\frac{x\sigma^2+\mu\ep^2}{\ep^2+\sigma^2},\frac{\ep^2\sigma^2}{\ep^2+\sigma^2}\right)
  \end{equation*} 
  is a polynomial of degree $n$ (with $\ep$-dependent coefficients) in the variable $x$. For time dependent $\mu(t)$,
  $\sigma^2(t)$ satisfying the hypotheses of Lemma~\ref{lem:1}, it follows that $\ep^2+\sigma^2\geq \nu>0$ for any $\ep>0$. These
  facts imply that the right-hand-side of~\eqref{eq:4} can be written in the form
  $\mathcal{Q}_{\ep,n,t}(x)\mathcal{G}(x,\mu(t),\sigma^2(t)+\ep^2)$, where the polynomial $\mathcal{Q}_{\ep,n,t}(x)$ has
  time-Lipschitz coefficients whose Lipschitz constants are uniformly bounded as $\ep\rightarrow 0$. For these
  reasons,~\eqref{eq:4} satisfies the statement of Lemma~\ref{lem:1}, and the result of the application of Lemma~\ref{lem:1}
  on~\eqref{eq:4} is independent of $\ep$ as $\ep\rightarrow 0$. On a similar note, we notice that
  \begin{equation*}
    \sum_{k=0}^{n}{\binom{n}{k}x^{n-k}m\left(k+1;\frac{(\mu-x)\ep^2}{\ep^2+\sigma^2},\frac{\ep^2\sigma^2}{\ep^2+\sigma^2}\right)}
  \end{equation*}
  can be written as $\mathcal{Q}_{\ep,n,t}(x):=\ep^2\mathcal{P}_{\ep,n,t}(x)$, where the polynomial $\mathcal{P}_{\ep,n,t}(x)$ has
  time-Lipschitz coefficients whose Lipschitz constants are uniformly bounded as $\ep\rightarrow 0$. This is a consequence of the
  Gaussian moments of order at least one, for a Gaussian kernel with both mean $\frac{(\mu-x)\ep^2}{\ep^2+\sigma^2}$ and variance
  $\frac{\ep^2\sigma^2}{\ep^2+\sigma^2}$ featuring a multiplicative factor $\ep^2$. This factor can be cancelled out with that
  appearing in the right-hand-side of~\eqref{eq:130}, which can hence be written in the form
  $\mathcal{P}_{\ep,n,t}(x)\mathcal{G}(x,\mu(t),\sigma^2(t)+\ep^2)$.  For these reasons,~\eqref{eq:130} satisfies the statement of
  Lemma~\ref{lem:1}, and the result of the application of Lemma~\ref{lem:1} on~\eqref{eq:130} is independent of $\ep$ as
  $\ep\rightarrow 0$. Similar considerations apply for~\eqref{eq:2001}. The contents of this remark apply under Assumption~(G),
  for the time dependent $X$ being precisely the Langevin particle $q_i(t)$ satisfying~\eqref{eq:24}.\\
  In addition, the right-hand-sides of~\eqref{eq:4},~\eqref{eq:130} and~\eqref{eq:2001} are Lipschitz in time, with Lipschitz
  constant independent of $\ep$ (see discussion above) and $x$ (each one of the right-hand-sides being a product of a polynomial
  with a decaying exponential).
\end{rem}

\subsection{Fokker--Planck time regularity in the case of non-vanishing potential $V$}
\label{sec:12}

The contents of this subsection should be seen as the ``replacement'' of Lemma~\ref{lem:1}, Lemma~\ref{lem:2}, Remark~\ref{rem:1},
under Assumption~(NG). We consider the \emph{Fokker--Planck} equation associated with~\eqref{eq:24}, namely
\begin{equation}
  \label{eq:118}
  \left\{
    \begin{array}{l}
      \displaystyle \frac{\partial g}{\partial t}=-\nabla\cdot(g\mu)+\frac{\sigma^2}{2}\frac{\partial^2}{\partial p^2}g,  
      \vspace{0.2 pc}\\
      \displaystyle g(0,p,q)=g_0(p,q),%\frac{\m p_i}{\m t}= \left(-\gamma p_i-V'(q_i)\right)+\sigma\,\dot\beta_i,\qquad i=1,\cdots,N,
    \end{array}
  \right.
\end{equation}
where $g_0(p,q)$ is the law of $(q(0),p(0))$.

\begin{rem}
We comment on some consequences of~\cite[Theorem 0.1]{Herau2004a}. This result, among many things, implies the following bound
for the solution to~\eqref{eq:134}
\begin{align}
  \label{eq:135}
  \|\overline{g}(t,\cdot,\cdot)\|_{M^{1/2}H^{s,s}}\leq C(1+Q_s(t))e^{-\tau t}\|\overline{g}_0\|_{M^{1/2}H^{-s,-s}},
\end{align} 
where $\tau>0$, where $C=C(\gamma,\sigma,V,\tau,)$, and $Q_s(t)$ is a continuous positive function such that
$\lim_{t\rightarrow 0^+}{Q_s(t)}=+\infty$, $\lim_{t\rightarrow +\infty}{Q_s(t)}<+\infty$, and where $M^{1/2}H^{s,s}$ denotes the
weighted isotropic Sobolev Space of order $s$ with weight $M^{-1/2}$, as stated in Assumption~(NG). In addition, well-posedness
of~\eqref{eq:134} is proved in $M^{1/2}\mathcal{S}'(\mathbb{R}^{2d})$. The auxiliary initial condition $\overline{g}_0$ mentioned
in Assumption~(NG) may be used in~\eqref{eq:135} to deduce that
\begin{align}
  \label{eq:136}
  \|\overline{g}(s,\cdot,\cdot)\|_{M^{1/2}H^{5,5}}\leq C_{\overline{t}},\qquad \mbox{ for all } s\geq \overline{t}>0.
\end{align}

The well-posedness of~\eqref{eq:134} in $M^{1/2}\mathcal{S}'(\mathbb{R}^{2d})$, the choice of $\overline{g}_0$ made in
Assumption~(NG) and~\eqref{eq:136} imply the following bound for the solution to~\eqref{eq:118}
\begin{align}
  \label{eq:137}
  \|g(t,\cdot,\cdot)\|_{M^{1/2}H^{5,5}}=\|\overline{g}(\overline{t}+t,\cdot,\cdot)\|_{M^{1/2}H^{5,5}}\leq C_{\overline{t}},
  \qquad \mbox{ for all } t\geq 0,
\end{align}
We remind the reader that $g$ is the probability density function of a Langevin particle $(q_i(t),p_i(t))$
satisfying~\eqref{eq:24}.
\end{rem}

%% CommentFC17 here

\begin{lemma}
  \label{lem:4}
  Let $g(t,q,p)$ be the solution to~\eqref{eq:118}, and let Assumption~(NG) be satisfied. For some $\alpha\in(1/4,1/2)$ and any
  $0\leq s<t\leq T$, we have
  \begin{align}
    \|g(t,\cdot,\cdot)-g(s,\cdot,\cdot)\|_{L^2(\mathbb{R}^2)} 
    & \leq C|t-s|,\label{eq:110}\\
    \|M^{-\alpha}\left(g(t,\cdot,\cdot)-g(s,\cdot,\cdot)\right)\|_{L^{\infty}(\mathbb{R}^2)} & \leq C|t-s|,\label{eq:111}\\
    \left\|M^{-\alpha}(\partial/\partial q)\left(g(t,\cdot,\cdot)-g(s,\cdot,\cdot)\right)\right\|_{L^{\infty}(\mathbb{R}^2)} 
    & \leq C|t-s|.\label{eq:116}
  \end{align}
\end{lemma}

\begin{proof}
We write 
\begin{align}
  \label{eq:112}
  \|g(t,q,p)-g(s,q,p)\|^2_{L^2(\mathbb{R}^2)} 
  & \leq 2\left\|\int_{s}^{t}-\nabla\cdot(\mu g)
    \m z\right\|_{L^2(\mathbb{R}^2)}^2+2\left\|\int_{s}^{t}\frac{\sigma^2}{2}\frac{\partial^2}{\partial p^2}g\,
    \m z\right\|_{L^2(\mathbb{R}^2)}^2\nonumber\\
  & \leq 2|t-s|\int_{s}^{t}{\left\|\nabla\cdot(\mu)g+\mu\cdot\nabla g\right\|^2_{L^2(\mathbb{R}^2)}
    \m z}+2|t-s|\int_{s}^{t}{\left\|\frac{\sigma^2}{2}\frac{\partial^2}{\partial p^2}g\right\|^2_{L^2(\mathbb{R}^2)}\m z}\nonumber\\
  & \leq 2|t-s|\int_{s}^{t}{\left\|M^{1/2-\alpha}M^{-1/2+\alpha}\left(\nabla\cdot(\mu)g+\mu\cdot\nabla g\right)
    \right\|^2_{L^2(\mathbb{R}^2)}\m z}\nonumber\\
  & \quad +2|t-s|\int_{s}^{t}{\left\|M^{1/2-\alpha}M^{-1/2+\alpha}\frac{\sigma^2}{2}\frac{\partial^2}{\partial p^2}
    g\right\|^2_{L^2(\mathbb{R}^2)}\m z}.
\end{align}
Assumption~(NG) implies that $V$ has at most polynomial growth, while $M$ decays exponentially in $p,q$. This immediately implies
that $\|\nabla\cdot(\mu)M^{1/2-\alpha}\|_{L^{\infty}(\mathbb{R}^2)}<\infty$ and
$\||\mu|M^{1/2-\alpha}\|_{L^{\infty}(\mathbb{R}^2)}<\infty$. In addition, $M^{-1/2+\alpha}g$ is uniformly bounded in time in
$H^{2,2}(\mathbb{R}^2)$ thanks to~\eqref{eq:137}. This is enough to control the $L^2(\mathbb{R}^2)$-norm of the remaining terms
$M^{-1/2+\alpha}g$, $M^{-1/2+\alpha}\nabla g$, $M^{-1/2+\alpha}(\partial^2/\partial p^2)g$, and proceed in~\eqref{eq:112}
deduce~\eqref{eq:110}. As for~\eqref{eq:111}, we have
\begin{align}
  \label{eq:113}
  & \left\|M^{-\alpha}\!\left(g(t,q,p)-g(s,q,p)\right)\right\|_{L^{\infty}(\mathbb{R}^2)}  
    \leq \int_{s}^{t}{\!\left[\left\|M^{-\alpha}\nabla\cdot(\mu)g+M^{-\alpha}\mu\cdot
    \nabla g\right\|_{L^{\infty}(\mathbb{R}^2)}+\left\|M^{-\alpha}\frac{\sigma^2}{2}\frac{\partial^2}{\partial p^2}
    g\right\|_{L^{\infty}(\mathbb{R}^2)}\right]\m z}\nonumber\\
  & \leq \int_{s}^{t}{\left[\left\|M^{1/2-2\alpha}M^{-1/2+\alpha}\left(\nabla\cdot(\mu)g+\mu\cdot\nabla g\right)
    \right\|_{L^{\infty}(\mathbb{R}^2)}+\left\|M^{1/2-2\alpha}M^{-1/2+\alpha}\frac{\sigma^2}{2}\frac{\partial^2}{\partial p^2}
    g\right\|_{L^{\infty}(\mathbb{R}^2)}\right]\m z}.
\end{align}
The terms $\|\nabla\cdot(\mu)M^{1/2-2\alpha}\|_{L^{\infty}(\mathbb{R}^2)}$, $\||\mu|M^{1/2-2\alpha}\|_{L^{\infty}(\mathbb{R}^2)}$
are bounded. We then use~\eqref{eq:137} and the Sobolev embedding Theorem to deduce~\eqref{eq:111} from~\eqref{eq:113}. The proof
of~\eqref{eq:116} is analogous.
\end{proof}

\begin{prop}
  \label{p:3}
  Let $T>0$. Let Assumption~(NG) be satisfied. Let $(q,p)$ obey the Langevin dynamics~\eqref{eq:24}. Let $A(q,p):=p^{n_1}q^{n_2}$,
  for some $n_1,n_2\in\mathbb{N}$, and let $c\geq 2$. Then, for any $s,t\in[0,T]$, we have
  \begin{align}
    \int_{\mathbb{R}}{\left|\mean{w_\ep(x-q(t))A(q(t),p(t))-w_\ep(x-q(s))A(q(s),p(s))}\right|^c\emph{d} x} 
    & \leq C|t-s|^{1+\beta},\label{eq:104}\\   
    \int_{\mathbb{R}}{\left|\mean{w'_\ep(x-q(t))A(q(t),p(t))-w'_\ep(x-q(s))A(q(s),p(s))}\right|^c\emph{d} x} 
    & \leq C|t-s|^{1+\beta},\label{eq:105}
  \end{align}
  where $C$ is independent of $\ep>0$. We also have, for any $x\in\mathbb{R}$
  \begin{align}
    \left|\mean{w_\ep(x-q(t))A(q(t),p(t))-w_\ep(x-q(s))A(q(s),p(s))}\right| & \leq K|t-s|,\label{eq:108}\\
    \left|\mean{w'_\ep(x-q(t))A(q(t),p(t))-w'_\ep(x-q(s))A(q(s),p(s))}\right| & \leq K|t-s|,\label{eq:117}
  \end{align}
where $K$ is independent of $\ep>0$ and $x\in\mathbb{R}$.
\end{prop}

\begin{proof}
We rewrite the left-hand-side of~\eqref{eq:104} as
\begin{align*}
  & \int_{\mathbb{R}}{\left|\mean{w_\ep(x-q(t))A(q(t),p(t))-w_\ep(x-q(s))A(q(s),p(s))}\right|^c\m x}\\
  & = \int_{\mathbb{R}}{\left|\int_{\mathbb{R}}{\int_{\mathbb{R}}{w_{\ep}(x-q)A(q,p)(g(t,p,q)-g(s,p,q))\m p\m q}}\right|^c\m x} = \|w_{\ep}
    \ast \left(\tilde{g}(\cdot,t)-\tilde{g}(\cdot,s)\right)\|^{c}_{c},
\end{align*}
where $\tilde{g}(q,t):=\int_{\mathbb{R}}{A(q,p)g(t,q,p)\m p}$. Let us define
$h_{s,t}(q,p):=\left|\left(g(t,q,p)-g(s,q,p)\right)\right|$. We proceed as
\begin{align*}
  \|w_{\ep}\ast \left(\tilde{g}(\cdot,t)-\tilde{g}(\cdot,s)\right)\|^{c}_{c} 
  & \leq \|w_{\ep}\|^c_1\|\tilde{g}(\cdot,t)-\tilde{g}(\cdot,s)\|^c_c 
    \leq \int_{\mathbb{R}}{\left|\int_{\mathbb{R}}{\left|A(q,p)\right|h_{s,t}(q,p)\m p}\right|^c\m q}.
\end{align*}
Fix $\theta\in(1/c,2/c)\subset (0,1)$. We split $h_{s,t}(q,p)=h^{\theta}_{s,t}(q,p)h^{1-\theta}_{s,t}(q,p)$. We apply the H\"older
inequality for this splitting in the above inner $p$-spatial integral, and we get
\begin{align}
  \label{eq:106}
  \int_{\mathbb{R}}{\left|\int_{\mathbb{R}}{A(q,p)h_{s,t}(q,p)\m p}\right|^c\m q}
  \leq \int_{\mathbb{R}}{\left(\int_{\mathbb{R}}{h_{s,t}(q,p)^2\m p}\right)^{\theta c/2}
  \left(\int_{\mathbb{R}}{|A(p,q)|^{\theta'}h_{s,t}(q,p)^{\theta''}\m p}\right)^{c/\theta'}\m q},
\end{align}
where $\theta'':=(1-\theta)\theta'>0$, and $\theta'$ is conjugate to $2/\theta$. Let $\alpha\in(1/4,1/2)$. We use~\eqref{eq:111}
to deduce that
\begin{align*}
  \int_{\mathbb{R}}{|A(p,q)|^{\theta'}h_{s,t}(q,p)^{\theta''}\m p} 
  & = \int_{\mathbb{R}}{|A(p,q)|^{\theta'}M^{\alpha\theta''}M^{-\alpha\theta''}h_{s,t}(q,p)^{\theta''}\m p}\\
  & \leq K\int_{\mathbb{R}}{|A(p,q)|^{\theta'}M^{\alpha\theta''}\m p}\leq K |q|^{n_2\theta'}\exp\{-CV(q)\},
\end{align*}
for some $C=C(n_1,\theta,\theta',\gamma,\sigma,\alpha)>0$. We apply the H\"older inequality (in the $q$ variable)
in~\eqref{eq:106} to deduce
\begin{align*}
  \int_{\mathbb{R}}{\left|\mean{w_\ep(x-q(t))A(q(t),p(t))-w_\ep(x-q(s))A(q(s),p(s))}\right|^c\m x}
  \leq C\|h_{s,t}\|_{L^2(\mathbb{R}^2)}^{c\theta}\leq C|t-s|^{1+\beta},
\end{align*}
where we have used Lemma~\ref{lem:4}, estimate~\eqref{eq:110}, in the last inequality. We thus proved~\eqref{eq:104}. The proof
of~\eqref{eq:105} is similar. We can rewrite the left-hand-side of~\eqref{eq:105} as
\begin{align}
  \label{eq:115}
  & \int_{\mathbb{R}}{\left|\int_{\mathbb{R}}{\int_{\mathbb{R}}{w'_{\ep}(x-q)A(q,p)(g(t,p,q)-g(s,p,q))\m p\m q}}\right|^c\m x}\nonumber\\
  &\quad = \int_{\mathbb{R}}{\left|\int_{\mathbb{R}}{\int_{\mathbb{R}}{w_{\ep}(x-q)\frac{\partial}
    {\partial q}\left\{A(q,p)(g(t,p,q)-g(s,p,q))\right\}\m p\m q}}\right|^c\m x},
\end{align}
where we have also used integration by parts in the $q$ variable, and the fact that the integrands decay to $0$ for
$q\rightarrow\pm\infty$, by~\cite[Theorem 0.1]{Herau2004a}. From~\eqref{eq:115} onwards, the computations carried out
for~\eqref{eq:104} can now be adapted line by line with $\partial/\partial q\left\{A(q,p)g(t,q,p)\right\}$ replacing
$A(q,p)g(t,q,p)$. This is possible because the $q$-derivative introduces a polynomial-type correction to $A(q,p)g(t,q,p)$, which
can dealt with as above, using again the exponential decay of $M$.

We turn to~\eqref{eq:108}. We rely on~\eqref{eq:111}, and compute 
\begin{align*}
  & \left|\mean{w_\ep(x-q(t))A(q(t),p(t))-w_\ep(x-q(s))A(q(s),p(s))}\right| \\
  & \quad \leq \int_{\mathbb{R}}{\int_{\mathbb{R}}{\left|w_{\ep}(x-q)A(q,p)(g(t,q,p)-g(s,q,p))\right|\m q\m p}}
    \leq C|t-s|\int_{\mathbb{R}}{\int_{\mathbb{R}}{\left|w_{\ep}(x-q)A(q,p)M^{\alpha}\right|\m q\m p}}\\
  & \quad \leq C|t-s|\int_{\mathbb{R}}{\|w_{\ep}(x-\cdot)\|_{L^1}|p|^{n_1}\exp\{-C(\alpha,\gamma,\sigma) p^2/2\}\m p}=K|t-s|,
\end{align*}
which is the desired estimate. The proof of~\eqref{eq:117} is completely analogous, and it relies on integration by parts for
$w_{\ep}'$ and estimate~\eqref{eq:116}.
\end{proof}

\begin{rem}
  \label{rem:200}
  With the notation and assumptions of Proposition~\ref{p:3}, it is not difficult to adapt the proof of the same proposition to
  show that $\int_{\mathbb{R}}{\left|\mean{w_\ep(x-q(0))A(q(0),p(0))}\right|^c\m x}$,
  $\int_{\mathbb{R}}{\left|\mean{w'_\ep(x-q(0))A(q(0),p(0))}\right|^c\m x}$,
  $\int_{\mathbb{R}}{\left|\mean{w''_\ep(x-q(0))A(q(0),p(0))}\right|^c\m x}$ are uniformly bounded in $\ep$.
\end{rem}

\subsection{Estimate on negative powers of the density $\rho_{\ep}$}
\label{ss:2}

\begin{prop}
  \label{p:2}
  Assume the validity of either Assumption~(G) or Assumption~(NG). Let $N\ep^{\theta}=1$, for some $\theta> 3$, and let
  $\rho_{\ep}$ be as in~\eqref{eq:102}.  Let $D\subset\mathbb{R}$ be a bounded set, and let $T>0$ be fixed. As
  $N\rightarrow\infty$ and $\ep\rightarrow0$, we have
  \begin{align}
    \label{eq:52}
    \mean{\rho_{\ep}^{-2}(x,t)}\leq C(D,T),\qquad\mbox{ for all } x\in D,\mbox{ for all } t\in[0,T],
  \end{align}
where $C$ is independent of $N,\ep$.
\end{prop}

\begin{proof}[Proof of Proposition \ref{p:2} under Assumption~(G)] %% Comment FC18 here
We know that 
\begin{equation*}
  q_i(t)\sim\mathcal{N}(\mu_q(t),\sigma_q^2(t)),\qquad t\in[0,T]. 
\end{equation*}
Also, $\mu_q(t)$ is bounded on $[0,T]$. %., thanks to the Langevin dynamics.
We can think of the quantity $x-q_{i}(t)$ as being $(x-\mu_q(t))-(\mu_q(t)-q_i(t))$. This observation, together with the
distributional symmetry of Gaussian random variables with mean zero, allows us to prove the statement by considering the simpler
setting
\begin{align*}
  q_i(t) & \sim\mathcal{N}(0,\sigma^2_q(t)),\qquad\mbox{ for all } t\in[0,T],\\
  0\leq x & \leq\max_{y\in D}{|y|}+\max_{s\in[0,T]}|\mu_q(s)|=:M,
\end{align*}
without loss of generality. Notice that we have performed an abuse of notation with respect to $q_i$. We fix $t\in[0,T]$, and $x$
satisfying the above condition. With our scaling choice $N=\ep^{-\theta}$, we have
\begin{equation*}
  \rho_\epsilon(x,t) = C\epsilon^{\theta-1}\sum_{i=1}^{N}{\exp(-(q_i(t)-x)^2/2\epsilon^2)}.
\end{equation*}
For $\ep\leq1$, there exists $\kappa=\kappa(D,T)$ such that  
\begin{align}
  \label{eq:124}
  \kappa\cdot\ep\leq\underbrace{\mathbb{P}\left(q_i(t)\in(x-\epsilon,x+\epsilon)\right)}_{=:p_{x,t,\ep}},
  \qquad\mbox{ for all } t\in[0,T],\,\mbox{ for all } x\in[0,M].
\end{align}
A simple choice is $\kappa:=(2/(2\pi\iota))\exp\{-(M+1)^2/2\nu\}$, where we have used Assumption~(G).

The $N$ particles being independent, we have
\begin{equation*}
  n(x,t):=\#\{\mbox{particles in }(x-\epsilon,x+\epsilon)\mbox{ at time }t\}
  \sim \mbox{ Bi}(N,p_{x,t,\ep})=\mbox{Bi}(\epsilon^{-\theta},p_{x,t,\ep}).
\end{equation*}
We fix a positive real number $\eta$. It then follows that, on the set $\{n(x,t)\geq 1\}$, we have
\begin{equation*}
  \frac{1}{\rho_{\epsilon}^{\eta}(x,t)}\leq \frac{1}{(n(x,t)\epsilon^{\theta-1})^{\eta}}.
\end{equation*}

\emph{Estimate on the set $\{n=0\}$.} We now focus on the set $\{n(x,t)=0\}$. First of all, we notice that this event is
asymptotically highly unlikely. More precisely, using the independence of particles, we get
\begin{align}
  \label{eq:22}
  \mathbb{P}\left(n(x,t)=0\right) 
  & =  \mathbb{P}\left(\mbox{all particles in }(x-\epsilon,x+\epsilon)^C\mbox{ at time }t\right)=(1-p_{x,\ep,t})^{N}\nonumber\\
  & = (1-p_{x,\ep,t})^{\ep^{-\theta}} \leq (1-\kappa\ep)^{\ep^{-\theta}} \leq %\exp\{-\ep^{-5}\log(1/(1-\ep))\}\leq 
    \exp\left\{-\ep^{-(\theta-1)}\frac{\kappa}{2}\right\}.
\end{align}

Now that we have the asymptotic probability of finding no particles in $(x-\ep,x+\ep)$, we rely on the trivial bound
$\rho_{\ep}(x,t)\geq w_{\ep}(x-\tilde{q}(t))N^{-1}$, where $\tilde{q}(t)$ is the closest particle to $x$ at time $t$. In symbols,
$\tilde{q}(t):=q_{a}(t)$, where $a:=\arg\min_{i=1,\ldots,N}{|q_i(t)-x|}$. We compute the probability density function for
$|\tilde{q}(t)-x|$. For this purpose, we compute, for every $y\geq0$,
\begin{align*}
  \mathbb{P}\left(|x-\tilde{q}(t)|\leq y\right) 
  & =  1-\mathbb{P}\left(|\tilde{q}(t)-x|> y\right)=1-\mathbb{P}
    \left(\mbox{all particles in }(x-y;x+y)^C\mbox{ at time }t\right)\\
  & =  1-\mathbb{P}\left(q_1\mbox{ in }(x-y;x+y)^C\mbox{ at time }t\right)^N=1-(\Phi_t(x-y)+1-\Phi_t(x+y))^N,
\end{align*}
where we have set $\Phi_t(z):=\int_{-\infty}^{z}{\mathcal{G}(y,0,\sigma_q^2(t))\m y}$. In the rest of this proof only, we will
shorten $\mathcal{G}(y,0,\sigma_q^2(t))$ to simply $G_t(y)$. If we differentiate with respect to $y$, we get the probability
density function for $|\tilde{q}(t)-x|$
\begin{equation*}
  f_{|\tilde{q}(t)-x|}(y)=\mathbf{1}_{y\geq 0}\cdot N(\underbrace{\Phi_t(x-y)+1-\Phi_t(x+y)}_{=:Z_{x,t}(y)})^{N-1}(G_t(x-y)+G_t(x+y)).
\end{equation*}
We now rely on the inequality
\begin{equation*}
  \mean{\frac{1}{\rho^{\eta}_{\ep}(x,t)}}\leq \mean{\frac{N^{\eta}}{w^{\eta}_{\ep}(\tilde{q}(t)-x)}}.
\end{equation*}
We write the expectation on the right-hand-side using the probability density function for $|\tilde{q}-x|$.
\begin{equation}
  \label{eq:62}
  \mean{\frac{N^{\eta}}{w^{\eta}_{\ep}(\tilde{q}(t)-x)}}  =  N^{\eta}\int_{0}^{+\infty}{N(\Phi_t(x-y)+1-\Phi_t(x+y))^{N-1}(G_t(x-y)+G_t(x+y))\frac{1}{w^{2}_{\ep}(y)}\m y}.
\end{equation}
%% Comment FC19 here

Before we deal with~\eqref{eq:62}, we need to estimate $Z_{x,t}(y)$, at least for large values of $y$. It is immediate to see that
$Z_{x,t}(y)\leq Z_{M,t}(y),\mbox{ for all } y\geq 0$. We compute the derivative
\begin{align*}
  \frac{\m}{\m\alpha}\mathcal{G}(z,0,\alpha)=C\exp\left\{-z^2/(2\alpha)\right\}\alpha^{-3/2}\left(z^2\alpha^{-1}-1\right).
  % \geq 0\qquad\mbox{for }y\geq\iota,
\end{align*}
Thanks to Assumption~(G), this entails that 
\begin{align}
  Z_{M,t}(y)\leq Z_{M,\overline{t}}(y),\qquad \mbox{for }y\geq M+\sqrt{\iota},
\end{align}
where we have set $\overline{t}:=\arg\max_{s\in[0,T]}{\sigma^2_q(s)}$.  We now examine the ratio
$Z_{M,\overline{t}}(y)/G_{\overline{t}}(y-M)$. We use the de L'Hopital's rule and compute
\begin{align*}
  \lim_{y\rightarrow +\infty}{\frac{Z_{M,\overline{t}}(y)}{G_{\overline{t}}(y-M)}} 
  & =  \lim_{y\rightarrow +\infty}{\frac{Z'_{M,\overline{t}}(y)}{G_{\overline{t}}'(y-M)}}=\lim_{y\rightarrow 
    +\infty}{\frac{-G_{\overline{t}}(M-y)-G_{\overline{t}}(M+y)}{\frac{M-y}{\sigma_q^2(\overline{t})}\,G_{\overline{t}}(y-M)}} \\
  & = \lim_{y\rightarrow 
    +\infty}{\left\{\frac{\sigma_q^2(\overline{t})}{y-M}
    +\frac{\sigma_q^2(\overline{t})}{y-M}\exp\left(-\frac{4My}{2\sigma_q^2(t)}\right)\right\}}=0.
\end{align*}
This implies the existence of $\overline{y}=\overline{y}(D,T)>M+\sqrt{\iota}$ such that
\begin{align}
  \label{eq:121}
  Z_{x,t}(y)\leq Z_{M,\overline{t}}(y)\leq
  \begin{cases}
    1 & \mbox{ if } y \leq \overline{y}, \\
    \exp\left(-\frac{(y-M)^2}{2\iota}\right) & \mbox{ if } y\geq\overline{y}.
  \end{cases}
\end{align}

We are now able to compute~\eqref{eq:62} by splitting the integration on the two regions $[0,\overline{y}]$ and
$[\overline{y},+\infty)$ provided by~\eqref{eq:121}. We obtain
\begin{align*}
  \mean{\frac{N^{\eta}}{w^{\eta}_{\ep}(\tilde{q}(t)-x)}} 
  & = N^{\eta}\int_{0}^{\overline{y}}{N(\Phi_t(x-y)+1-\Phi_t(x+y))^{N-1}(G_t(x-y)+G_t(x+y))\frac{1}{w^{\eta}_{\ep}(y)}\m y}\\
  % & =  N^{\eta}\int_{0}^{\overline{y}}{N(\Phi_t(x-y)+1-\Phi_t(x+y))^{N-1}(G_t(x-y)+G_t(x+y))\frac{1}{w^{\eta}_{\ep}(y)}\m y}\\
  & \quad+  N^{\eta}\int_{\overline{y}}^{+\infty}{N(\Phi_t(x-y)+1-\Phi_t(x+y))^{N-1}(G_t(x-y)+G_t(x+y))\frac{1}{w^{\eta}_{\ep}(y)}\m y}\\
  & \leq  CN^{\eta}\underbrace{\int_{0}^{\overline{y}}{\frac{N}{w^{\eta}_{\ep}(y)}\m y}}_{=:T_1}\\
  & \quad 
    +  N^{\eta}\underbrace{\int_{\overline{y}}^{+\infty}{N\exp\left(-\frac{(y-M)^2(N-1)}{2\iota}\right)
    (G_t(x-y)+G_t(x+y))\frac{1}{w^{\eta}_{\ep}(y)}\m y}}_{=:T_2}.
\end{align*}

Integral $T_1$ can be bounded as
\begin{align*}
  \int_{0}^{\overline{y}}{\frac{N}{w^{\eta}_{\ep}(y)}\m y} 
  & =  CN\int_{0}^{\overline{y}}{\ep^{\eta}\exp\left(\frac{\eta y^2}{2\ep^2}\right)\m y}
    =  C\ep^{\eta+1} N\int_{0}^{(\overline{y})/\ep}{e^{z^2}\m z} \leq K_1(D,T,\eta)N\ep^{\eta}\exp\{K_2(D,T)\ep^{-2}\}.
\end{align*}

As for integral $T_2$, we notice that the scaling $N\ep^{\theta}=1$ and the condition $\overline{y}>M+\sqrt{\iota}$ provide the
bound
\begin{align*}
  \frac{\eta y^2}{2\ep^2}-\frac{(y-M)^2(N-1)}{2\iota}\leq-\frac{(y-M)^2}{4\iota/N}, 
  \qquad\mbox{ for }N\geq\overline{N}=\overline{N}(D,T).
\end{align*}
We can then estimate $I_2$ for $N\geq\overline{N}$, thus obtaining
\begin{align*}
  I_2 & \leq CN\ep^{\eta}\int_{0}^{+\infty}{\exp\left\{\frac{\eta y^2}{2\ep^2}-\frac{(y-M)^2(N-1)}{2\iota}\right\}\m y} 
         \leq CN\ep^{\eta}\int_{0}^{+\infty}{\exp\left\{-\frac{(y-M)^2}{4\iota/N}\right\}\m y}\leq CN^{1/2}{\ep}^{\eta}.
\end{align*}
%% Comment FC20 here

We combine the contributions of $T_1$ and $T_2$ and deduce
\begin{equation}
  \label{eq:23}
  \mean{\frac{N^{\eta}}{w^{\eta}_{\ep}(\tilde{q}(t)-x)}}\leq K_1(D,T)N^{\eta}\ep^{\eta}\left\{N^{1/2}+N\exp(K_2(D,T)\ep^{-2})\right\}.
\end{equation}
We set $\eta=4$ and we deduce that
\begin{align*}
  \mean{\rho^{-2}_{\ep}(x,t)\cdot\mathbf{1}_{\{n(x,t)=0\}}} 
  & \leq \mean{\rho^{-4}_{\ep}(x)}^{1/2}\mathbb{P}\left(n(x,t)=0\right)^{1/2}\\
  & \leq K_1(D,T)N^{2}\ep^2\left\{N^{1/2}+N\exp(K_2(D,T)\ep^{-2})\right\}^{1/2}
    \exp\left\{-\ep^{-(\theta-1)}\frac{\kappa}{4}\right\}\rightarrow 0,
\end{align*} 
as $N \to \infty$ and $\ep \to 0$. The scaling $N\ep^{\theta}=1$, with $\theta>3$, is used to show the convergence to 0 of the
above estimate. We have dealt with the expectation of $\rho^{-2}_{\ep}(x,t)$ on the set $\{n(x,t)=0\}$, uniformly over $x\in D$
and $t\in[0,T]$.
%% Comment FC21 here

\emph{Estimate on the set $\{n\geq 1\}$.} We now turn to the set $\{n(x,t)\geq 1\}$, and more precisely to estimating
$\mean{\rho^{-2}_{\ep}(x,t)\cdot\mathbf{1}_{\{n(x,t)\geq 1\}}}$. We have already noticed that on $\{n(x,t)\geq 1\}$ we have the
bound
\begin{equation*}
  \frac{1}{\rho_{\epsilon}^{2}(x,t)}\leq \frac{1}{(n(x,t)\epsilon^{\theta-1})^{2}}.
\end{equation*}
We use some tools from~\cite{Chao1972a}. In particular, we estimate $\mean{n(x,t)^{-2}}$ using~\cite[Corollary of Section 2, and
Section 3]{Chao1972a}. We have $\mean{(n(x,t)+2)^{-2}}=\int_{0}^{1}{g_2(z)\m z}$, where for $z\in[0,1]$
\begin{equation*}
  g_2(z):=z^{-1}\int_{0}^{z}{g_1(u)\m u},\qquad g_1(z):=t(q+pz)^N,
\end{equation*}
and where we have abbreviated $p:=p_{x,t,\ep}$, $q:=1-p_{x,t,\ep}$. We bound $g_2$ as
\begin{align*}
  g_2(z) & =  z^{-1}\int_{0}^{z}{u(q+pu)^N\m u}
           \leq \int_{0}^{z}{(q+pu)^N\m u}=p^{-1}\int_{0}^{z}{\frac{\m}{\m u}\left\{\frac{(q+pu)^{N+1}}{N+1}\right\}\m u}
           =  \frac{(q+pz)^{N+1}-q^{N+1}}{p(N+1)}.
\end{align*}
We use the scaling $N=\ep^{-\theta}$ and proceed as
\begin{align*}
  \mean{(n(x,t)+2)^{-2}}=\int_{0}^{1}{g_2(u)\m u} 
  & \leq \int_{0}^{1}{\frac{1}{p(N+1)}\left\{(q+pu)^{N+1}-q^{N+1}\right\}\m u}\\
  & \leq \frac{q^{N+1}}{p(N+1)}+\frac{1}{p^2(N+1)(N+2)} 
    \leq \frac{\ep^{\theta-1}}{\kappa}\exp\left\{-\ep^{-(\theta-1)}\frac{\kappa}{2}\right\}+\frac{\ep^{2\theta-2}}{\kappa^2}.
\end{align*}
As a result we obtain 
\begin{align*}
  \mean{\rho^{-2}_{\ep}(x,t)\cdot\mathbf{1}_{\{n(x,t)\geq 1\}}} 
  & \leq  \mean{\frac{1}{(n(x,t)\epsilon^{\theta-1})^{2}}\cdot\mathbf{1}_{\{n(x,t)\geq 1\}}}\leq
    \frac{3^2}{\ep^{2\theta-2}}\mean{\frac{1}{(n(x,t)+2)^2}\cdot\mathbf{1}_{\{n(x,t)\geq 1\}}}\\
  & \leq \frac{3^2}{\ep^{2\theta-2}}\mean{\frac{1}{(n(x,t)+2)^2}}
    \leq 3^2\left[\frac{\ep^{1-\theta}}{\kappa}\exp\left\{-\ep^{-(\theta-1)}\frac{\kappa}{2}\right\}+\frac{1}{\kappa^2}\right],
\end{align*}
which is uniformly bounded in $\ep$, $N$. Combining the estimates on $\{n=0\}$ and $\{n\geq 1\}$ gives the result.
\end{proof}

\begin{proof}[Adaptation of the proof of Proposition~\ref{p:2} under Assumption~(NG)] 
We need to check that~\eqref{eq:124} still holds, and also adapt~\eqref{eq:121}. The validity of~\eqref{eq:124} is a consequence of the theory of positive transition densities for degenerate diffusion stochastic differential equations, see~\cite[Section 3]{Herzog2015a} and~\cite{Mattingly2002a}. 
%% Comment FC22 here

Let us now consider $x\in D,t\in[0,T]$. We define $\Phi_t(z)$ to be the cumulative distribution function of $q_1(t)$. We need to
estimate
%\begin{align*}
$
  Z_{x,t}(y):=\Phi_t(x-y)+1-\Phi_t(x+y)
$
%\end{align*}
by providing a rapidly decaying estimate as $y\rightarrow+\infty$, similarly to~\eqref{eq:121}. We use Lemma~\ref{lem:4} to deduce
\begin{align}
  \label{eq:122}
  f_{q(t)}(q)\leq C\int_{\mathbb{R}}{M^{1/2-\alpha}(q,p)\m p}\leq Ce^{-kV(q)},
\end{align} 
where $f_{q(t)}$ denotes the probability density function of $q_1(t)$, where $\alpha\in(1/4,1/2)$, where
$k:=(1/2-\alpha)(2\gamma/\sigma^2)$, and where $M$ is given in Assumption~(NG). For $y\geq 3\max_{x\in D}{|x|}$, we consider the
limit
\begin{align*}
  \lim_{y\rightarrow +\infty}{\frac{Z_{x,t}(y)}{e^{-kV(y)}}} 
  & \leq \lim_{y\rightarrow +\infty}{\frac{\int_{\mathbb{R}\setminus[-y/2;y/2]}{e^{-kV(q)}\m q}}{e^{-kV(y)}}}
    \leq C\lim_{y\rightarrow +\infty}{\frac{-e^{-kV(y)}-e^{-kV(y)}}{V'(y)e^{-kV(y)}}}=0,
\end{align*}
where we have used~\eqref{eq:122} is the first inequality, and de L'H\^opital's rule and Assumption~(NG) for the second
inequality. The above limit, in combination with the growth rate of $V$ (at least quadratic thanks to Assumption~(NG)), guarantees
that
\begin{align*}
Z_{x,t}(y)\leq
     \begin{cases}
      1, & \mbox{ if } y \leq \overline{y}=\overline{y}(D,T,V), \\
      \exp\left(-\frac{y^2}{2\iota}\right), & \mbox{ if } y\geq\overline{y}.
     \end{cases}
\end{align*}
for some $\iota>0$. The above estimate replaces~\eqref{eq:121} in the remaining part of the proof, which is unchanged.
\end{proof}

\begin{rem}
  The growth condition for $V$ (i.e., the requirement $n\geq 1$, instead of $n>1/2$) is dictated by the Adaptation of Proof of
  Proposition~\ref{p:2}. This stricter condition is not necessary for the proofs of Lemma~\ref{lem:4} and Proposition~\ref{p:3}.
\end{rem}
%% Comment FC16 here

\paragraph{Acknowledgements} FC is supported by a scholarship from the EPSRC Centre for Doctoral Training in Statistical Applied
Mathematics at Bath (SAMBa), under the project EP/L015684/1.  JZ gratefully acknowledges funding by a Royal Society Wolfson
Research Merit Award and the Leverhulme Trust for its support via grant RPG-2013-261. All authors thank the anonymous referees for
their careful reading of the manuscript and their valuable suggestions.

%\bibliographystyle{plainnat-jz}
%\bibliographystyle{plainnat-my}
%% Syntax for ln: ln -s ~/research/biblio/johannes/jz.bib jz.bib
%\bibliography{jz}
%
%\def\cprime{$'$} \def\cprime{$'$} \def\cprime{$'$}
%  \def\polhk#1{\setbox0=\hbox{#1}{\ooalign{\hidewidth
%  \lower1.5ex\hbox{`}\hidewidth\crcr\unhbox0}}} \def\cprime{$'$}
%  \def\cprime{$'$}

\def\cprime{$'$} \def\cprime{$'$} \def\cprime{$'$}
  \def\polhk#1{\setbox0=\hbox{#1}{\ooalign{\hidewidth
  \lower1.5ex\hbox{`}\hidewidth\crcr\unhbox0}}} \def\cprime{$'$}
  \def\cprime{$'$}

\end{document}